\documentclass{amsart}
\usepackage{amssymb,latexsym,amscd}

\numberwithin{equation}{section}

\newtheorem{theorem}{Theorem}[section]
\newtheorem{proposition}[theorem]{Proposition}
\newtheorem{conjecture}[theorem]{Conjecture}
\newtheorem{corollary}[theorem]{Corollary}
\newtheorem{remark}[theorem]{Remark}
\newtheorem{lemma}[theorem]{Lemma}
\newtheorem{example}[theorem]{Example}
\newtheorem{definition}[theorem]{Definition}
\newtheorem{problem}[theorem]{Problem}

\newtheorem{maintheorem}[theorem]{Main Theorem}
\newtheorem{claim}[theorem]{Claim}

\def\endproof{\hfill$\square$\medskip}

\def\FF{\mathcal{F}}
\def\ZZ{\mathbb{Z}}
\def\QQ{\mathbb{Q}}
\def\CC{\mathbb{C}}
\def\C{\mathcal{C}}

\def\RR{\mathcal{R}}

\def\UU{\mathcal{U}}

\def\1{\mathbb{1}}
\def\l{\ell}

\def\gg{\mathfrak{g}}
\def\mm{\mathfrak{m}}
\def\nn{\mathfrak{n}}
\def\hh{\mathfrak{h}}
\def\VV{\mathcal{V}}

\def\bb{\mathfrak{b}}

\def\pp{\mathfrak{p}}

\def\R{\mathcal R}
\def\OO{\mathcal{O}}
 \def\zz{\mathfrak{z}}
 \def\gl{gl}

\begin{document}

\title{R-matrix Poisson Algebras and Their Deformations}
\author{S. Zwicknagl}
\address{Sebastian Zwicknagl,  Department of Mathematics, University of California, Riverside, 92521.}
\email{zwick@math.ucr.edu}
\maketitle
 \tableofcontents

\section{Introduction}\label{S:motivation}

We classify in the present paper Poisson brackets on modules over a semisimple complex Lie algebra which are based on classical $r$-matrices. We then quantize these Poisson structures in the spirit of the recent joint paper \cite{BZ} with A. Berenstein and show that we recover many well known examples  
of quantized coordinate rings of classical varieties. 

Let us briefly discuss the main results in the case of a simple Lie algebra $\gg$. 
 Let $\gg=\nn_-\oplus \hh\oplus \nn_+$ be a simple complex Lie algebra and let $\{E_\alpha|\alpha\in R_+\}$ (where $R_+$ is the set of positive roots of $\gg$)  be the standard basis of $\nn_+$ and $\{F_\alpha|\alpha\in R_+\}$ the  standard basis of $\nn_-$. Recall that $r=\sum_{\alpha\in R_+} E_\alpha\otimes F_\alpha\in\gg\otimes \gg$ is a classical $r$-matrix and $r^-=\sum_{\alpha\in R_+} E_\alpha\otimes F_\alpha-F_\alpha\otimes E_\alpha\in \gg\wedge\gg$ the antisymmetrized $r$-matrix. For each  $\gg$-module $V $ define a quadratic bracket $\{\cdot,\cdot\}$ on the symmetric algebra $S(V)$ by the formula:
\begin{equation}
\label{eq:bracket}
\{a,b\}=r^-(a\wedge b)=\sum_{\alpha\in R_+} E_\alpha(a)F_\alpha(b)-E_\alpha(b)F_\alpha(a)
\end{equation}
for $a,b\in S(V)$. In particular, if $\gg=sl_2(\CC)$, then $\{a,b\}=E(a)F(b)-E(b)F(a)$.

The bracket is, by construction, skew-commutative and satisfies the Leibniz rule. To determine whether it is Poisson for a pair $(\gg,V)$ one has to verify whether it satisfies the Jacobi identity.
Our first main result is the following theorem.
\begin{maintheorem}(Theorem \ref{th: class of Poisson})
\label{th: geom-decomp}
Let $V $ be a simple  finite-dimensional $\gg$-module. Assume that $(\gg,V)\ne (sp_{2n}(\CC),V_{\omega_1})$. Then the following are equivalent:
 
\noindent (a) The bracket \eqref{eq:bracket} on $S(V)$ is Poisson. 

\noindent (b) ${\bf c}(\Lambda^3 V)=\{0\}$, where ${\bf c}$ the canonical $\gg$-invariant in $\Lambda^3 \gg$ corresponding (under the identification $\gg^*\cong \gg$) to the Lie bracket $[\cdot,\cdot]:\gg\wedge \gg\to \gg$.

\noindent (c) $V$ is geometrically decomposable  as defined by R. Howe (\cite{Ho}).

\noindent (d) $Hom_\gg(\Lambda^3V,S^3V)=\{0\}$.

\noindent (e) $\Lambda^2V$ is simple.

\end{maintheorem}

If $(\gg,V)= (sp_{2n}(\CC),V_{\omega_1})$, then parts  (a), (b), and (d) of Theorem \ref{th: geom-decomp} hold, but parts (c) and (e) fail.  In Theorem \ref{th:class:poissonsemisimple} we classify all simple modules $V$ over a semisimple Lie algebra for which the bracket \eqref{eq:bracket} on $S(V)$ is Poisson.  The only, nontrivial, example of a simple module over a semisimple Lie algebra with this property is the natural module of $\gg=sl_n\times sl_m(\CC)$ for arbitrary $m,n\in \ZZ_{\ge0}$. 

 All pairs $(\gg,V)$ for which  the bracket \eqref{eq:bracket} on $S(V)$ is Poisson are classified  in \cite{ZW}.

We then continue to show that the deformation quantization of the $r$-matrix Poisson structure on a $\gg$-module $V_\lambda$ recovers the braided symmetric algebras of the $U_q(\gg)$-module $V_\lambda^q$. The braided symmetric algebra $S_q(V^q_\lambda)$ is quadratic $q$-deformations of the symmetric algebra of the $U(\gg)$-module $V_\lambda$ (see \cite{BZ} or Section \ref{se:BESP}) which are $U_q(\gg)$-module algebras. An important problem in \cite{BZ} is the question, for which $U_q(\gg)$-module $V^q$  the deformation is flat; i.e. one has $\dim((S_q(V^q))_n)=\binom{\dim(V)+n-1}{n}$ for all graded components $(S_q(V^q))_n$. The following result completely classifies such {\it flat} modules. 

\begin{maintheorem}(Theorem \ref{the:class. flat})
 A $U_q(\gg)$-modules $V^q$ is flat, if and only if  the bracket \eqref{eq:bracket}  defines a Poisson structure on the symmetric algebra of the classical limit $V$ of $V^q$. 
\end{maintheorem}

If $\gg$ is of type $A_n$, $B_n$ or $D_n$, then the braided symmetric algebras of the flat modules are the quantized coordinate rings of the classical varieties such as the quantum $m\times n$-matrices, the quantum Euclidean space (see e.g \cite{RTF}), quantum symmetric and quantum antisymmetric matrices (see e.g. \cite{Nou} and \cite{Str}). The braided exterior powers of the flat natural modules of quantized enveloping algebras of types $B_n$, $C_n$ and $D_n$ agree with the $q$-wedge modules constructed by Jing, Misra and Okado in \cite{JMO}. Our approach, thus, provides a natural unifying construction for these objects. Moreover, following the arguments in \cite[Ch. 5]{GY} one obtains that the braided symmetric algebras are the quantizations of  equivariant Poisson structures on partial flag varieties.

Theorem \ref{th: geom-decomp} shows that there is an apparent relation between $r$-matrix Poisson structures, flat modules and maximal parabolic with Abelian or Heisenberg type radicals and classical invariant theory as studied by Howe in \cite{Ho}. We use this connection to give the following explicit construction of braided symmetric algebras of flat simple modules. 
 
\begin{theorem}(Theorem \ref{th: Abelian ->flat quadratic}, Theorem \ref{th:Heisenberg}) In the notation of Theorem \ref{th: geom-decomp}, if the bracket \eqref{eq:bracket} is Poisson 
(including the case $(\gg,V)= (sp_{2n}(\CC),V_{\omega_1})$), then:

\noindent (a) There exists a unique simple Lie algebra $\gg'$ and a maximal parabolic subalgebra $\pp\subset \gg'$ such that $\gg$ is the semisimple part of the Levi factor of $\pp$ and $V$ is isomorphic (as a $\gg$-module) to the nil-radical $rad_\pp$ of $\pp$.

\noindent (b) The associated graded of the quantized enveloping algebra $U_q(rad_\pp)$ is the Kontsevich deformation $S_q(V^q)$  of the Poisson algebra $S(V)$ and carries a natural $U_q(\gg)$-module structure.
\end{theorem}
 
The paper is organized as follows: In Section \ref{se:Poisson alg} we study quadratic Poisson algebras. We introduce the notions of decorated space and bracketed algebras and show that the categories of decorated spaces and bracketed algebras are symmetric monoidal.  We show how decorated spaces define bracketed algebras and show under which conditions the bracket satisfies the Jacobi identity (Theorem \ref{th: flat Poisson structures} ) extending well known results  of Gelfand and Fokas  \cite{GF}. We also show that if the bracketed algebra associated to a tensor product of two decorated spaces is Poisson, then each of the factors must define a Poisson algebra (Theorem \ref{th: Tensor products of Poissons}).

 We use these results in Section \ref{se:Poissonmods} to classify all simple Poisson modules over simple Lie algebras (Theorem \ref{th: class of Poisson}) and finally all simple Poisson modules over semisimple Lie algebras (Theorem \ref{th:class:poissonsemisimple}). Since the proof of Theorem \ref{th: class of Poisson} is rather long we present it in Section \ref{se:proof of classification}. 
 
 Section  \ref{se:prlime} is devoted to the classification of flat simple modules (Theorem \ref{the:class. flat}). For the convenience of the reader we provide brief introductions to the quantized enveloping algebras $U_q(\gg)$, the category of their finite -dimensional modules (Section \ref{se:q-group}), the definition and basic properties  of braided symmetric and exterior algebras and powers (Section \ref{se:BESP}) and the classical limit (Section \ref{se:classical limit}). 
 
 In Section \ref{se:constru.of BSA} we construct the braided symmetric algebras as the quantized enveloping algebra of nilradicals (Theorem \ref{th: Abelian ->flat quadratic}), respectively their associated graded algebras (Theorem \ref{th:Heisenberg}). In Section \ref{se: qschcells} we prove a PBW-type theorem for quantum Schubert cells and study Levi actions on quantized nilradicals.  
 
 The results in this paper  open up many questions and suggest connections between the theory of braided symmetric algebras, cluster algebras, geometric crystals, equivariant Poisson structures and classical invariant theory.Appropriately, the paper concludes with a section on open questions and conjectures  (Section \ref{se:conj}).
   
 {\bf Acknowledgments} I would like to acknowledge the advice and support of A. Berenstein. Additionally, I am thankful for stimulating discussions with P. Etingof, A. Polishchuk and B. Shelton.

\section{Quadratic Poisson Brackets } 
\label{se:Poisson alg}

  \subsection{Decorated Spaces and Bracketed Algebras}
\label{se:decorated spaces}

In this section we will introduce the notion of a decorated space and relate it to bracketed algebras, which are commutative algebras with a skew-commutative bracket satisfying the Leibniz rule, but not necessarily the Jacobi identity.
Consider a linear tensor category $\C$ over a field $k$ of characteristic $char(k)=0$.  We will view $\C$ as a symmetric tensor category with the braiding given by the permutation of factors.
Define a decorated space  to be a pair $(V,\Phi)$ of an object $V$ of $\C$ and a morphism $\Phi:V\otimes V\to V\otimes V$ such that $\tau\circ \Phi=-\Phi\circ \tau$. Denote by $\mathcal{D}(\C)$ the {\it category of decorated spaces} whose objects are decorated spaces, and whose morphisms are structure preservingmaps in $\C$.  

One has the following result. 
\begin{lemma}
\label{le: decorated is tensor}
The category $\mathcal{D}(\C)$ is a braided symmetric category. More precisely we have the following:

\noindent(a) The unit object is $(k,{\bf 0})$, where ${\bf 0}: k\otimes k\to 0\in k$ is the trivial map.

\noindent(b) The tensor product $  (V\otimes V', \Phi'')= (V,\Phi)\otimes (V',\Phi')$ of two decorated  spaces $(V,\Phi)$ and $(V',\Phi')$ is defined for all $u,v\in V$ and $u',v'\in V'$ via 
\begin{equation}
\label{eq: decorated tensor product}
 \Phi''=   \Phi_{13}+\Phi'_{24} \ ,
 \end{equation}
 
where  $\Phi_{13}=\tau_{23}\circ(\Phi\otimes Id^{\otimes 2})\circ \tau_{23}$ and  $\Phi'_{24}=\tau_{23}\circ(Id^{\otimes 2}\otimes \Phi' )\circ \tau_{23}$.  
 
\noindent(c)  The symmetric braiding is given by the  morphisms $\mathcal{D}\tau :(V,\Phi)\otimes (V',\Phi')\to (V',\Phi')\otimes (V,\Phi)$  such that $\mathcal{D}\tau(V\otimes V')=V'\otimes V$ and $\mathcal{D}\tau(\Phi_{13}+\Phi'_{24})=\Phi'_{13}+\Phi_{24}$.  

\noindent(d) Moreover, the direct sum $(V,\Phi)\oplus (V',\Phi')=(V\oplus V',\Phi+\Phi')$ of two  decorated spaces $(V,\Phi)$ and $(V',\Phi')$ is a decorated space.
\end{lemma}

\begin{proof}
 It is easy to see that the tensor product $ (V,\Phi)\otimes (V',\Phi)$ of two decorated spaces is a  decorated space, and that the tensor product  is indeed associative. One can now show that $(k,{\bf 0})$ satisfies the axioms of the unit.  Parts (a) and (b) are proved. 
 Parts (c) and (d) are obvious. The lemma is proved. 
\end{proof}

\begin{lemma}
\label{le:extension to decorated}
Let $\C'$ be a symmetric linear category and let $\mathcal{F}: \C\to \C'$ be a covariant  monoidal-functor compatible with the braiding. Then, $\mathcal{F}$ defines a covariant functor $\mathcal{DF} : \mathcal {D}(\C)\to \mathcal {D} (\C')$, which takes an object  $(V,\Phi)$  to $(\mathcal{F}(V),\mathcal{F}(\Phi))$.
   \end{lemma} 

\begin{proof}
It suffices to show that  for every object $V$ of $\C$ and every  $\Phi\in End(V\otimes V)$ such that $\tau\circ\Phi= -\Phi)\circ\tau$ one has $\tau \circ\mathcal{F}(\Phi)= -\left(\mathcal{F}(\Phi)\circ\tau_{\mathcal{F}(V),\mathcal{F}(V)}\right)$. Since $\mathcal{F}$ is compatible with the braiding we compute:
$$ \tau_{\mathcal{F}(V),\mathcal{F}(V)}\circ\mathcal{F}(\Phi)= \mathcal{F}(\tau_{V,V}\circ\Phi )= \mathcal{F}(-\Phi\circ\tau_{V,V})=-\left(\mathcal{F}(\Phi)\circ\tau_{\mathcal{F}(V),\mathcal{F}(V)}\right)\ .$$

The lemma is proved.
\end{proof}

We next introduce    bracketed and Poisson algebras in the category $\C$. 
 
\begin{definition}
\label{def:bracketed  algebra}
(a) A bracketed algebra is a pair $(A,\{\cdot,\cdot\})$, where $A$ is a commutative algebra  in $\C$   and $\{\cdot,\cdot\}$ is a 
 a structure preserving bilinear map
$\{\cdot,\cdot \}:A\otimes A\to A$ satisfying:

\noindent (i)    anti-commutativity
\begin{equation}
\label{eq: super-anti-commutativity}
\{a,b\}+\{b,a\}=0\ ,
\end{equation}
for any $a,b\in A$.

\noindent (ii) the Leibniz rule 
\begin{equation}
\label{eq:right Leibniz rule}
\{a,bc \}= \{a,b \}c+b \cdot \{a,  c\}    
\end{equation}
 
for any $a,b,c\in A$.   

(b) A bracketed algebra is called Poisson, if $\{\cdot,\cdot\}$ satisfies the Jacobi identity:

$$\{a,\{b.c\}\}+\{c,\{a,b\}\}+\{b,\{c,a\}\}=0$$
for all $a,b,c\in A$.
\end{definition}

 We denote by $BAlg(\C)$ the category of bracketed algebras  in $\C$, whose morphisms are structure preserving algebra homomorphisms in $\C$.

\begin{remark}
 Let $(A,\{\cdot,\cdot\})$be a bracketed algebra  in $BAlg(\C)$. The bracket $\{\cdot,\cdot\}$ satisfies
\begin{equation}
\label{eq:left Leibniz rule}
 \{c a,b\}= c\cdot \{a,b\}+ \{c,b\}\cdot a\ .
\end{equation}

\end{remark}

Note the following obvious fact.

\begin{lemma}
 The Poisson algebras form a full subcategory of $BAlg(\C)$, the category of Poisson algebras.
\end{lemma}

 

We need the following fact.

\begin{lemma}
\label{ le: BAlg symm}
The category $ BAlg(\C)$ is a symmetric monoidal category.  More precisely:

 
 \noindent(a) The tensor product  $(A \otimes B, \{\ ,\ \}):= (A,\{\cdot,\cdot\}_{A})\otimes (B,\{\cdot,\cdot\}_{B}) $ of two objects  $(A,\{\ ,\ \}_{A})$ and $(B,\{\ ,\ \}_{B})$  is defined in the following way:  for all $a,a'\in A$, $b,b'\in B$ define, 
$$\{ a \otimes b,a' \otimes b'\}= \left( (1\otimes \mu_B)\circ( \{a, a'\}\otimes b\otimes b')+(\mu_A\otimes 1)\circ (a\otimes a \otimes \{ b ,b'\}) \right)\ ,$$
 where $\mu_A$ denotes the multiplication in $A$ and $\mu_B$ the multiplication in $B$.

 \noindent(b)  Its unit element is $(k, {\bf 0})$.

\noindent(c) The symmetric braiding is defined as 
$$\sigma( (A,\{\ ,\ \}_{A})\otimes (B,\{\ ,\ \}_{B})= (B,\{\ ,\ \}_{B})\otimes  (A,\{\ ,\ \}_{A})\ .$$

\end{lemma}

\begin{proof}
 To prove part (a) we have to show  that the tensor product  is associative.  We compute: 
$$\{ a\otimes (b\otimes c), a'\otimes (b'\otimes c')\}$$
$$=\left(\{\cdot,\cdot\}\otimes \mu_B\otimes \mu_C+ \mu_A \otimes  \{\cdot,\cdot\}\otimes \mu_C  + \mu_A\otimes \mu_B \otimes  \{\cdot,\cdot\} \right) \circ \sigma_{45}\circ\sigma_{23,4} (a\otimes (b\otimes c)\otimes a'\otimes (b'\otimes c')) \ ,$$
where $\sigma_{45}= Id_{A\otimes A\otimes B}\otimes \sigma_{C,B}\otimes Id_{C}$ and $\sigma_{23,4}= Id^{A}\otimes \sigma_{B\otimes C, A}\otimes Id_{B\otimes C}$.
 
 Similarly, we obtain that 
 $$\{ (a\otimes b)\otimes c, (a'\otimes b')\otimes c'\}$$
$$=\left(\{\cdot,\cdot\}\otimes \mu_B\otimes \mu_C+ \mu_A \otimes  \{\cdot,\cdot\}\otimes \mu_C +  \mu_A\otimes \mu_B \otimes  \{\cdot,\cdot\} \right) \circ \sigma _{23}\circ\sigma _{3,45} (a\otimes (b\otimes c)\otimes a'\otimes (b'\otimes c')) \ ,$$
  where $\sigma _{23}= Id_{A}\otimes \sigma_{B,A}$ and $\sigma _{3,45}= Id_{A\otimes B}\otimes \sigma_{C, A\otimes B}\otimes Id_{C}$.    
Since $\sigma$ is a braiding we  have
$\sigma_{23,4}= \sigma_{23}\circ\sigma_{34}$ and  $\sigma_{3,45}= \sigma_{45}\circ\sigma_{34}$ with $ \sigma_{34}= Id_{A\otimes B}\sigma_{C,A}\otimes Id_{B\otimes C}$.

Since $\sigma_{23}\circ\sigma_{45}=\sigma_{45}\circ\sigma_{23}$ it is now easy to verify that  
$\sigma_{23}\circ\sigma_{3,45}= \sigma_{45}\circ\sigma_{23,4}$ and hence, that the tensor product is indeed associative.

Part (a) is proved, and Parts (b) and (c) are obvious.
 Lemma \ref{ le: BAlg symm} is proved. 
\end{proof}

We have the following obvious facts.

 \begin{lemma}
\noindent(a)The assignment $A\to (A, {\bf0})$ defines a faithful tensor functor from the category of commutative algebras in $\C$ to the category of bracketed algebras.
 
 \noindent(b)The assignment $(A,\{\cdot,\cdot\})\to A$ defines a forgetful functor  from $BAlg(\C)$ to the category of algebras in $\C$. \end{lemma}

 \subsection{Symmetric Algebras and Bracketed Symmetric Algebras}
 
 Let $V$ be an object of $\C$.  The category $\C$ is Abelian, hence $Id_{V\otimes V}+ \sigma_{V,V}$ is a  morphism in $\C$ and the {\it symmetric square} of $S^2 V=Im(Id_{V\otimes V}+\sigma)$, respectively the {\it exterior square} $\Lambda^2 V=Ker(Id_{V\otimes V}+\sigma)$ are objects in $\C$. Define the {\it $n$-th symmetric power} as 
  $$ S^{n}V= S^{2}V\otimes V^{\otimes (n-2)}\cap V\otimes S^{2}V\otimes V^{\otimes (n-2)}\cap\ldots\cap V^{\otimes (n-2)}\otimes S^{2}V\subset V^{\otimes n}\ .$$ 
  
Similarly, define the {\it $n$-th exterior power} as 
  $$ \Lambda^{n}V= \Lambda^{2}V\otimes V^{\otimes (n-2)}\cap V\otimes \Lambda^{2}V\otimes V^{\otimes (n-2)}\cap\ldots\cap V^{\otimes (n-2)}\otimes \Lambda^{2}V\subset V^{\otimes n}\ .$$

 Define the {\it symmetric algebra} $S(V)=T(V)/\langle \Lambda^2 V\rangle$ as the quotient of the tensor algebra of $V$ by the two-sided ideal generated by  the exterior square.

 The following facts are immediate for a symmetric linear category where the braiding is given by the permutation of factors.
 
\begin{proposition}
\label{pr:S is a functor} 
\noindent(a)  Let $V$ be an object of $\C$. The symmetric algebra $S(V)$ is a commutative  algebra in $\C$.  

\noindent(b) The assignment $V\mapsto S(V)$  is functorial.    

\noindent(c) The functor is exponential, i.e. $S(U\oplus V)\cong S(U)\otimes_{\sigma} S(V)$.    
   
\end{proposition} 
 
 Define morphisms $\sigma_{i,i+1}\in End_\C(V^{\otimes n})$   for all $1\le i\le (n-1)$ by $\sigma_{i,i+1}=Id^{\otimes (i-1)}\otimes \sigma\otimes Id^{\otimes(n-i-1}$. Define for any permutation $\tau$  in the symmetric group  $\mathfrak{S}_{n}$ the morphism $\sigma_{\tau}\in End(V^{\otimes n})$ as $\sigma_{\tau}= \sigma_{i_{r},i_{r}+1}\circ\ldots \circ \sigma_{i_{1},i_{1}+1}$, where $\tau= (i_{r},i_{r}+1)\circ\ldots\circ (i_{1},i_{1}+1)$ is a presentation of $\tau$ consisting of simple transpositions and $\sigma_{i,i+1}$ defined as above.  Note that  $\sigma_{\tau}$ is well defined,  independent of the choice of presentation of $\tau$. 
 
 Recall the definition of the $n$-th braided factorial:
$$[n]!_{\sigma}= [n]!_{V,\sigma}=[n]!_{\sigma}: V^{\otimes n}\to V^{\otimes n}\ , [n]!_{\sigma}= \sum_{\tau\in \mathfrak{S}_{n}}\sigma_{\tau}\ .$$ 
By definition $[n]!_{\sigma}\circ \sigma_{i}=\sigma_{i}\circ [n]!_{\sigma}=[n]!_{\sigma}$. Similarly, define the $n$-th braided skew-factorial $[n]!_{-\sigma}=[n]!_{-\sigma,V}=\sum_{\tau\in \mathfrak {S}_{n}} (-1)^{\l(\tau)}\sigma_{\tau}$, where $\l(\tau)$ denotes the length of the permutation $\tau$. The following fact is well known.
 
\begin{proposition}
\label{pr:symmetric spaces }
 For any object $V$ of $\C$ one has:
 
\noindent(a)  $S^{n}V=Im([n]!_{V,\sigma})$ and $\Lambda^{n}V= Im([n]!_{V,-\sigma})$. 

\noindent(b)  $\langle \Lambda^2V\rangle_n=Ker~ [n]!_\sigma$, where $\langle \Lambda^{2}V\rangle _{n} $ is  the degree $n$ component of the ideal  $\langle \Lambda^{2}V\rangle$.  Equivalently, the composition
$\phi_{n}: S^{n}V\hookrightarrow T V\twoheadrightarrow S(V)$  is an   isomorphism between  $S^{n}V$ and  $S(V)_{n}$, where   $S(V)_{n}$ is the $n$-the graded component of $S(V)$.




\end{proposition}

 Let $V$ be an object in the symmetric tensor category $\C$ and let  $\Phi: V\otimes V\to V\otimes V$ be a morphism in $\C$.   For $1\le i<j\le n$ define $\Phi_{i,j}:V^{\otimes n}\to V^{\otimes n}$  as follows. First, $\Phi_{i,i+1}:\in End(V^{\otimes n})= Id^{\otimes i-1}\otimes \Phi\otimes Id^{\otimes n-i-1}$ and, recursively,  for $i<j-1$ 
\begin{equation}
\label{eq: Phi i,j}
\Phi_{i,j}=\sigma_{(j-1,j)}\circ\Phi_{i,j-1}\circ\sigma_{(j-1,j)}\ . 
\end{equation}

\begin{proposition}
\label{pr: tau on Phi}
For any $i<j$ and any $\tau\in \mathfrak{S}_{n}$  such that $\tau(i)<\tau(j)$ one has
$\Phi_{\tau(i),\tau(j)}=\sigma_{\tau}\circ\Phi_{i,j}\circ\sigma_{\tau^{-1}}$.  
\end{proposition}
\begin{proof}
We first need the following fact.  
\begin{lemma}
\label{le: transposition and Phi}
 If $1\le i<j\le n$, $ m\le n-1$ and $\tau=(m,m+1)$, then   $\sigma_{m,m+1}\circ\Phi_{i,j}\circ\sigma_{m,m+1}= \Phi_{\tau(i),\tau(j)}$  
  \end{lemma}
  \begin{proof}
Clearly, we obtain
 \begin{equation}
 \label{eq: commutating permutations}
 \sigma_{ m,m+1}\circ\Phi_{i,j}\circ\sigma_{m,m+1}=\Phi_{i,j}\ , 
 \end{equation}
  if $\{ m,m+1\}\cap\{i,j\}=\emptyset$. 
 
The assertion holds by definition \eqref{eq: Phi i,j} if $m=j$ or $m=j-1$. 
It remains the case, when $m=i$ or $m+1=i$. Since $\sigma_{m,m+1}$ is an involution, it suffices to prove the assertion for $m+1=i$; i.e. we have to show that $\sigma_{i-1,i}\circ\Phi_{i,j}\circ\sigma_{i-1,i}=\Phi_{i-1,j}$. We need the following fact.
\begin{lemma}
Let $\tau= (j-1,j)\circ (j-2,j-1)\circ\ldots \circ (i-1,i)$. Then we have for $\Phi\in End(V\otimes V)$ that
$$ \sigma_{\tau }\circ\Phi_{i,j}\circ \sigma_{\tau^{-1}}= \Phi_{i-1,j-1}=\Phi_{\tau(i),\tau(j)}\ .$$
\end{lemma}
 \begin{proof}
The braid  relation  of the symmetric group $\mathfrak{S}_n$ yields  that  $\sigma_{i,i+1}\circ\sigma_{i,i-1}\circ\Phi_{i,i+1}=\Phi_{i-1,i}\circ\sigma_{i,i-1}\circ\sigma_{i,i+1}$.  We obtain through repeated application that  $ \sigma_{\tau}\circ \Phi_{i ,j}= \Phi_{i-1 ,j-1}\circ\sigma_{\tau}$. The lemma is proved.  
  \end{proof}
  
 Now, we compute 
 $$ \sigma_{i-1,i}\circ\Phi_{i,j}\circ \sigma_{i-1,i}= \sigma_{i-1,i}\circ(\sigma_{\tau})^{-1}\circ\Phi_{i-1,j-1}\circ\sigma_{\tau}\circ\sigma_{i-1,i}$$
 $$=\sigma_{i,i+1}\circ\ldots\circ \sigma_{j-2,j-1}\circ \sigma_{j-1,j}\circ\Phi_{i-1,j-1}\circ \sigma_{j-1,j}\circ\sigma_{j-2,j-1}\circ\ldots\circ\sigma_{i,i+1}\ .$$
$$  \sigma_{i,i+1}\circ\ldots\circ \sigma_{j-2,j-1}\circ\Phi_{i-1,j}   \circ\sigma_{j-2,j-1}\circ\ldots\circ\sigma_{i,i+1}\ .$$

We obtain applying \eqref{eq: commutating permutations}  multiple times that
$$  \sigma_{i,i+1}\circ\ldots\circ \sigma_{j-2,j-1}\circ\Phi_{i-1,j}   \circ\sigma_{j-2,j-1}\circ\ldots\circ\sigma_{i,i+1}=\Phi_{i-1,j}\ .$$
We now have $\sigma_{i-1,i}\circ\Phi_{i-1,j}\circ\sigma_{i-1,i}=\Phi_{i,j}$.
 Lemma \ref{le: transposition and Phi} is proved.
 \end{proof}

We can now prove Proposition \ref{pr: tau on Phi} by induction on the length $\l(\tau)$ of $\tau$. The inductive base  is provided by Lemma \ref{le: transposition and Phi}.
 Next , let $\tau\in \mathfrak{S}_{n}$ such that $\l(\tau)>1$.  Write $\tau=\tau'\circ\sigma_{m,m+1}$ such that $\l(\tau')=\l(\tau)-1$. It is easy to see that $\tau'(i)<\tau'(j)$. Otherwise  we would have $\tau(i)=m, \tau(j)=m+1$, and $\tau'(i)=m+1$ and $\tau'(j)=m$. This implies that we would obtain a reduced expression $\tau'=\tau''\circ\sigma_{m',m'+1}\circ\tau'''$ with $\l(\tau')=\l(\tau'')+\l(\tau''')+1$, such that $\tau''(i)=m,\tau''(j)=m+1$.   It is now easy to see that $\tau=\tau''\circ\tau'''$, and hence $\l(\tau)<\l(\tau')$.  We can now apply the inductive hypothesis to $\tau'$ to obtain $\sigma_{\tau'}\circ\Phi_{i,j}\circ\sigma_{(\tau')^{-1}}= \Phi_{\tau'(i),\tau'(j)}$.   Note that $\tau^{-1}=\sigma_{m,m+1}\circ(\tau')^{-1}$. Lemma \ref{le: transposition and Phi} implies that
  $$\sigma_{m,m+1}\circ \Phi_{\tau'(i),\tau'(j)}\circ\sigma_{m,m+1}= \Phi_{\tau(i),\tau(j)}\ .$$  
  Proposition \ref{pr: tau on Phi} is proved.
 \end{proof}

 We now define a map  $\Phi^{(m,n)}:V^{\otimes (m+n)}\to V^{\otimes (m+n)}$ by the formula:
$$\Phi^{(m,n)}:=\sum_{i=1}^m\sum_{j=m+1}^{m+n} \Phi_{i,j}\ .$$

  For any $a\in S^mV$, $b\in S^nV$, and   $\hat a\in V^{\otimes m}$ and $\hat b\in V^{\otimes n}$  such that  $\frac{1}{m!}[m]!_\sigma(\hat a)=a$ and $\frac{1}{n!} [n]!_{\sigma}(\hat b)=b$ we define the  morphism $\{\cdot,\cdot\}^{(m,n)}:S^{m}V\otimes S^{n}V\to S^{m+n}V$ by:
   $$\{a,b\}_{\Phi}^{(m,n)}:=\frac{1}{(m+n)!}[m+n]!_\sigma\Phi^{(m,n)}(\hat a\otimes \hat b)\ .$$
  The above morphism is well-defined because of the following result. Recall that  $\mathfrak{S}_{m}\times \mathfrak{S}_{n}$ embeds naturally in $\mathfrak{S}_{m+n}$. 
\begin{lemma}
If $\tau=(\tau_{1},\tau_{2})\in  \mathfrak{S}_{m}\times \mathfrak{S}_{n}\subset \mathfrak{S}_{m+n}$,   then
 $\sigma_{\tau}\circ\Phi^{(m,n)}=\Phi^{(m,n)}\circ\sigma_{\tau}$.

\end{lemma}
\begin{proof}
 One has by Lemma  \ref{pr: tau on Phi} that for any $\tau\in \mathfrak{S}_{n+m}$ and any $1\le i<j\le m+n$ that $\sigma_{\tau}\circ\Phi_{\tau(i),\tau(j)}=\Phi_{ i,j}\circ\sigma_{\tau}$. We compute
 $$ \Phi^{(m,n)}\sigma_{\tau}= \sum_{i=1}^m\sum_{j=m+1}^{m+n} \Phi_{i,j}\circ\sigma_{\tau}=\sum_{i=1}^m\sum_{j=m+1}^{m+n}\sigma_{\tau}\circ \Phi_{\tau_{1}^{-1}(i),\tau_{2}^{-1}(j)}= \sigma_{\tau}\circ \Phi^{(m,n)}\ .$$
  The lemma is proved.
\end{proof}

 We have the following result.
\begin{proposition}
\label{pr: Phi bracketed structure}
Let $(V,\Phi)$ be a decorated space.  The pair $(S(V),\{\cdot,\cdot\}_{\Phi})$, where $\{\cdot,\cdot\}_{\Phi}=\bigoplus\limits_{m,n\in\ZZ_{\ge0}}\{\cdot,\cdot\}_{\Phi}^{(m,n)}$ is a bracketed algebra  in $\C$.
\end{proposition}

\begin{proof}
Prove anti-commutativity  \eqref{eq: super-anti-commutativity} first. We need the following fact.
\begin{lemma}
\label{le:Phi skewcommutes}
\noindent(a)Let $(V,\Phi)$ be a decorated space. Then for all $n\ge2$ and  $i<j\le n$ one has
$$\sigma_{(i,j)}\circ\Phi_{i,j}=-\Phi_{i,j}\circ\sigma_{(i,j)}\ .$$

\noindent(b) If  $i<j\le n$ and  $\tau\in\mathfrak{S}_{n}$ such that  $\tau(i)>\tau(j)$, then
$$\sigma_{\tau}\circ\Phi_{i,j}\circ\sigma_{\tau^{-1}}=-\Phi_{\tau(j),\tau(i)}\ .$$
\end{lemma}

\begin{proof}
Prove (a) first.
Note that $\sigma_{(i,i+1)}\circ\Phi_{i,i+1}\circ\sigma_{(i,i+1)}=-\Phi_{i,i+1}$ and that  $(i,j)=(j-1,j)\circ (j-2,j-1)\circ (i,i+1)\circ(i+1,i+2)\ldots (j-1,j)$.  Denote  $\tau'= (j-1,j)\circ(j-2,j-1)\circ (i+1,i+2)$. We compute using  Proposition  \ref{pr: tau on Phi} 
$$ \sigma_{(i,j)}\circ\Phi_{i,j}\circ\sigma_{(i,j)^{-1}}=\sigma_{\tau'}\circ\sigma_{(i,i+1)}\circ \sigma_{\tau^{'-1}}\circ\Phi_{i,j}\circ\sigma_{\tau'}\circ\sigma_{(i,i+1)}\circ \sigma_{\tau^{'-1}}$$
$$= \sigma_{\tau'}\circ\sigma_{(i,i+1)}\circ\Phi_{i,i+1}\circ\sigma_{(i,i+1)}\circ \sigma_{(\tau')^{-1}}
=-\sigma_{\tau'}\circ\Phi_{i,i+1}\circ\sigma_{\tau'}=-\Phi_{i,j} \ .$$
Part (a) is proved.

Prove (b) now. We clearly have $\tau'(i)<\tau'(j)$ for $\tau'=(\tau(j),\tau(i))\circ\tau$. Therefore, we have by Lemma  \ref{pr: tau on Phi} and Part (a)
$$-\Phi_{\tau(j),\tau(i)}=\sigma_{(\tau(j),\tau(i))}\circ\Phi_{\tau(j),\tau(i)}\circ \sigma_{(\tau(j),\tau(i)}=\sigma_{(\tau(j),\tau(i)})\circ\sigma_{\tau'}\circ\Phi_{i,j}\circ\sigma_{\tau'}\circ\sigma_{(\tau(j),\tau(i)}$$
$$= \sigma_{\tau}\circ\Phi_{i,j}\circ\sigma_{\tau^{-1}} \ .$$
Part (b) is proved. The lemma is proved. 
\end{proof}

Let $\tau\in \mathfrak{S}_{m+n}$ be the permutation, which  sends $(1,2,\ldots m,m+1,\ldots n+m)\mapsto (n+1,n+2,\ldots m+n,1,2\ldots n)$.  Lemma \ref{le:Phi skewcommutes} yields that  for $i\le n< j\le m+n$  one has $\Phi_{i,j}\sigma_{\tau}=-\sigma_{\tau}\Phi_{\tau(j),\tau(i)}$.
Therefore, we have
$$\Phi^{(n,m)}\circ\sigma_{\tau}=\sum_{i=1}^m\sum_{j=m+1}^{m+n} \Phi_{i,j}\circ\sigma_{\tau}= -\sigma_{\tau}  \sum_{i=1}^m\sum_{j=m+1}^{m+n} \Phi_{\tau(j),\tau(i)}\ .$$

This implies that $\{b,a \}^{(n,m)}_{\Phi} =-\{ a,b\}^{(m,n)}_{\Phi}$ for $a\in S^{m}V$ and $b\in S^{n}V$, hence $\{a,b\}_{\Phi}=-\{b,a\}_{\Phi}$. Anti-commutativity is proved.

It remains to verify the Leibniz identity \eqref{eq:right Leibniz rule}. Let $a\in S^{n}V, b\in S^{m}V$ and $c\in S^{\l}V$ and let  $\hat a\in V^{\otimes n}, \hat b\in V^{\otimes m}$ and $\hat c\in V^{\otimes \l}$   be representatives of $a,b$ and $c$, respectively.  
Denote by $\tau'\in\mathfrak{S}_{m+n+\ell}$ the permutation $\tau'(1,\ldots,n+m+\ell)=(n+1,\ldots,n+m,1,\ldots,n,n+m+1\ldots n+m+\ell)$. We compute

$$\{ a,b\cdot c\}_{\Phi}=\{a,b\cdot c\}_{\Phi}^{(n,m+\l)} =\frac{1}{(n+m+\l)!}[n+m+\l]!_\sigma\Phi^{(n,m+\l)}(\hat a\otimes \hat b\otimes \hat c)$$
$$= \frac{1}{(n+m+\l)!}[n+m+\l]!_\sigma \left(\frac{[n+m]!_\sigma\otimes [\ell]!_\sigma}{(n+m)!\ell!}\sum_{i=1}^{n}\sum_{j=n+1}^{m+n}\Phi_{i,j}\right)(\hat a\otimes \hat b\otimes \hat c) $$
$$+ \frac{1}{(n+m+\l)!}[n+m+\l]!_\sigma \left( \frac{[m]!_\sigma\otimes[n+\ell]!_\sigma}{m!(n+\ell)!}\tau'(\sum_{i=1}^{n}\sum_{j=n+m+1}^{n+m+\l}\Phi_{i,j})\right)(\hat a\otimes \hat b\otimes \hat c) $$
$$= \{ a,b\}\cdot c+b\cdot\{a,c\} \ .$$
 
  Therefore the Leibniz rule  holds. Proposition \ref{pr: Phi bracketed structure} is proved.
\end{proof}

 We will denote by $S(V,\Phi)$ the bracketed algebra $(S(V),\{\cdot,\cdot\}_{\Phi})$ from Proposition \ref{pr: Phi bracketed structure} and refer to it as the {\it symmetric algebra} of the decorated space $(V,\Phi)$.  We have the following result.
 
 \begin{proposition}
 \label{pr: decorated to bracketed}
 The correspondence $ (V,\Phi)\mapsto S(V,\Phi)$  defines a faithful exponential functor from the category  of decorated spaces to the category of bracketed algebras $BAlg(\C)$.
\end{proposition}

\begin{proof} 
It is easy to verify that the correspondence is functorial and faithful. It remains to check that it is exponential. By Proposition \ref{pr:S is a functor}  one has $S(V\oplus V')=S(V)\otimes S(V')$, and we obtain the bracket defined by 
$$\{u+u',v+v'\}_{\Phi}=(\Phi+\Phi')((u+u')\wedge (v+v'))$$  for all $u,v\in V$ and $u',v'\in V'$. The proposition is proved.
 \end{proof}

 \subsection{Bracketed Poisson Algebras}
 \label{se:brack-Poisson}
 Denote by $J: A^{\otimes 3}\to A$ the Jacobian map defined by 
  
\begin{equation}
\label{eq;  Jacobi}
 J= F+  F\circ \sigma_{12}\circ\sigma_{23}+F\circ\sigma_{23}\circ\sigma_{12}\ , 
 \end{equation} 
 where $F: A^{\otimes 3}\to A, F(a,b,c)=\{a,\{b,c\}\}$ and $\sigma_{12}=\sigma\otimes Id$ and $\sigma_{23}=Id\otimes \sigma$.  
  
  The following fact is obvious.
   \begin{lemma}  (\cite[Definition 2.23]{BZ})
A bracketed algebra $(A,\{\cdot,\cdot\})$ is Poisson, if and only if $J(A^{\otimes 3})=0$.
 \end{lemma}

  Define the Jacobian ideal $\langle J_\Phi\rangle$ as the two-sided (bracketed) ideal in the bracketed algebra $S(V,\Phi)$ generated by the image of the Jacobian map. We call the quotient of $\overline{S(V)}_\Phi$ by $\langle J_\Phi\rangle$  the Poisson closure. The bracket $\{\cdot,\cdot\}_\Phi$ induces a bracket $\overline{\{\cdot,\cdot\}_\Phi}$ on $\overline{S(V)}_\Phi$, because  $\langle J_\Phi\rangle$ is by definition closed under the bracket.

\begin{definition}
\label{def:reduced symmetric algebra}
The {\it reduced symmetric algebra} $\overline {S(V,\Phi)}$ is the bracketed algebra $\overline {S(V,\Phi)}=(\overline{S(V)}_\Phi,\overline{\{\cdot,\cdot\}_\Phi})$. \end{definition}

 We have the following result.
  
\begin{proposition}
\label{pr:univ poisson closure}
\noindent(a) The reduced symmetric algebra $\overline {S(V,\Phi)}$ is Poisson.

\noindent(b) The reduced symmetric algebra $\overline {S(V,\Phi)}$ has the following universal property: Any homomorphism of bracketed algebras from $ S(V,\Phi)$ to a Poisson algebra $P$ factors through  $\overline {S(V,\Phi)}$.

\noindent(c) The assignment $(V,\Phi)\mapsto \overline {S(V,\Phi)}$ defines a  functor from the category of decorated spaces to the category of Poisson algebras. Moreover,  if $(V'',\Phi'')= (V,\Phi)\oplus(V',\Phi')$, then there exists a surjective homomorphism $\overline{S(V,\Phi)}\otimes \overline{S(V',\Phi')}\to \overline{S(V'',\Phi'')}$.  \end{proposition}   
\begin{proof}
Prove (a) first. Let $\overline a,\overline b,\overline c\in \overline {S(V,\Phi)}$ and let $a,b,c\in S(V,\Phi)$ be representatives of the equivalence classes of $\overline a,\overline b,\overline c$, respectively. Then 
$J(a,b,c)$ is a representative of the class of $\overline J(\overline a,\overline b,\overline c)$, where $\overline J$ is the induced Jacobian map on  $\overline {S(V,\Phi)}$. By definition, $J(a,b,c)\in \langle J_\Phi\rangle$, hence $\overline J(\overline a,\overline b,\overline c)=0\in \overline {S(V,\Phi)}$. Part(a) is proved.

 Prove (b) next.  Let $P$ be a Poisson algebra and $\rho: S(V,\Phi)\to P$ a homomorphism of bracketed algebras.  It is easy to see that $\langle J\rangle$ is contained in the kernel of $\rho$. Hence, $\rho$ factors through $\overline {S(V,\Phi)}$. Part(b) is proved. 

 Prove (c) now. Let $(V ,\Phi )$ and $(V',\Phi')$ be decorated spaces,  $(V'',\Phi'')=(V,\Phi)\oplus(V',\Phi')$,  and $\langle J \rangle$, resp.   $\langle J'\rangle$, the Jacobean ideal in $S(V ,\Phi )$, resp. $S(V' ,\Phi')$. It is clear that the ideals generated by $\langle J \rangle$ and $\langle J' \rangle$  in $S(V'',\Phi'')$ are contained in the Jacobian ideal  $\langle J''\rangle$ of $S(V'',\Phi'')$. Part(c) follows  since $S(V'',\Phi'')=S(V,\Phi)\otimes S(V',\Phi')$ by Proposition \ref{pr: decorated to bracketed}. The proposition is proved.
\end{proof}

  Due to Proposition \ref{pr:univ poisson closure}(a) we will sometimes refer to $\overline {S(V,\Phi)}$ as the Poisson closure of $ S(V,\Phi)$ (see also \cite[Section 3.1]{BZ}).

We will now discuss, when $ S(V,\Phi)$ is Poisson; i.e., when $ S(V,\Phi)=\overline {S(V,\Phi)}$.  
 
 For any $\Phi\in End(V\otimes V)$   define  the   Schouten square $[[ \Phi  , \Phi ]]\in End(V\otimes V\otimes V)  $  by:

$$[[ \Phi,\Phi]]=[\Phi_{12},\Phi_{13} ]+ [\Phi_{12},\Phi_{23}]+[\Phi_{13},\Phi_{23}]\ ,   $$
where  $[a,b]=a\circ b-b\circ a$ denotes the usual commutator. 


The Schouten square  has the following very important property.

 \begin{lemma}
 \label{le: Schouten v sigma}
 Let $V$ be an object of $\C$ and let $\Phi\in End(V\otimes V)$ such that $\Phi\circ\sigma=-\sigma\circ\Phi$. Then,
 $\sigma_{i,i+1}\circ[[ \Phi,\Phi]]\circ\sigma_{i, i+1}=-[[\Phi,\Phi]]$for $i=1,2$.
\end{lemma}
\begin{proof}
 Straightforward computation yields:
$$\sigma_{12}\circ[[\Phi,\Phi]]\circ\sigma_{12}=\sigma_{12}\circ [\Phi_{12},\Phi_{13} ]+ [\Phi_{12},\Phi_{23}]+[\Phi_{13},\Phi_{23}]\circ\sigma_{12}$$
$$= -[\Phi_{12},\Phi_{23}]-[\Phi_{12},\Phi_{13}+[\Phi_{23},\Phi_{13}])= -[[\Phi,\Phi]]\ .$$
 Similarly, we compute
 $$\sigma_{23}\circ[[\Phi,\Phi]]\circ\sigma_{23}= [\Phi_{13},\Phi_{12}]-[\Phi_{13},\Phi_{23}]-[\Phi_{12},\Phi_{23}]=-[[\Phi,\Phi]]\ .$$
 The lemma is proved.
\end{proof}

We call a decorated space $(V,\Phi)$ Poisson, if the  symmetric algebra $S(V,\Phi)$ of the decorated space $(V,\Phi)$ is Poisson.   
  
\begin{theorem}
\label{th: flat Poisson structures}
 Let $(V,\Phi)$ be a decorated space. The following are equivalent:

\noindent(a)  $(V,\Phi)$ is Poisson

\noindent(b) $\Phi$ satisfies the equation 
\begin{equation}
\label{eq:CYBE}
[3]!_{\sigma} \circ[[ \Phi,\Phi]]=0 \ .
 \end{equation}
 
 \noindent(c) $\Phi$ satisfies the equation
 $$ [[ \Phi,\Phi]]\circ [3]!_{-{\sigma}}=0\ .$$
 
 \noindent(d) $\Phi$ satisfies 
 $$[[ \Phi,\Phi]]\big |_{\Lambda^3 V}=0\ ,$$
 where $ \big |_{\Lambda^3 V}$ denotes the restriction to $\Lambda^3 V$.
 \end{theorem}

\begin{proof}

The equivalence of (a) and (b) is well known and proved in  \cite[Theorem 3.1]{GF}. 
For the convenience of the reader we nevertheless  prove here that (a) equivalent (b).

We need the following fact.
\begin{lemma}
\label{le:J vs [[,]] }
One has
$$im (J)\cap S^{3}(V)=[3]!_{\sigma} \circ[[ \Phi,\Phi]]\ .$$
 
\end{lemma}
\begin{proof}
 Define the  lifted Jacobian $J': V^{\otimes 3 }\to V^{\otimes 3}$   by
 $$J'(x,y,z):= G+  G\circ \sigma_{12}\circ\sigma_{23}+G\circ\sigma_{23}\circ\sigma_{12} , $$
 where $G: V^{3}\to V^{\otimes 3}$   is the  morphism given by 
 $$G =\Phi_{12}\circ (1\otimes 1\otimes 1+ \sigma_{23})\circ \Phi_{23}= \left(\Phi_{23}\circ \Phi_{12}+ \sigma_{12}\circ \Phi_{13}\circ \Phi_{12}\right) \ .$$   
  
By definition, $J(x,y,z)=\frac{[3]!_{\sigma}}{3!}(J'(\overline x,\overline y,\overline z))$ for all $x,y,z\in S(V)$ and all $\overline x,\overline y,\overline z\in T(V)$ such that $ [3]!_{\sigma} (\overline x)=3!\cdot x,[3]!_{\sigma} (\overline y)=3!\cdot y$ and $[3]!_{\sigma} (\overline z)=3! \cdot z$ .

 One has 
$[3]!_{\sigma}\circ \sigma_{i,j}=[3]!_{\sigma}$ for all $i,j=1,2,3$. Therefore,  we obtain:
$$[3]!_{\sigma}\circ G=[3]!_{\sigma}\circ \left(\Phi_{23}\circ \Phi_{12}+ \sigma_{12}\circ \Phi_{13}\circ \Phi_{12}\right)$$
$$= [3]!_{\sigma}\circ \left(\Phi_{23}\circ \Phi_{12}+\Phi_{13}\circ \Phi_{12}\right)  \ .$$
Similarly,
$$[3]!_{\sigma}\circ G\circ \sigma_{12}\circ  \sigma_{23}= [3]!_{\sigma}\circ \left(-\sigma_{12}\circ\sigma_{23} \circ \Phi_{12}\circ\Phi_{13}+ \sigma_{23}\circ \Phi_{23}\circ\Phi_{13} \right)$$
$$ =[3]!_{\sigma}\circ \left(-  \Phi_{12}\circ\Phi_{13}+\Phi_{23}\circ\Phi_{13} \right)     \ . $$
  
$$[3]!_{\sigma}\circ G\circ  \sigma_{23}\circ \sigma_{12} = [3]!_{\sigma}\circ   \left(- \Phi_{13}\circ \Phi_{23}-\Phi_{12}\circ \Phi_{23}\right)    \  .$$
Combining these equations,
we obtain  :
$$[3]!_{\sigma}\circ (G+ G\circ \sigma_{12}\circ  \sigma_{12}
+ G\circ \sigma_{23}\circ \sigma_{12})=-[3]!_{\sigma}\circ[[ \Phi,\Phi]]$$
and thus:
\begin{equation}
\label{eq:Schouten=jacobi}
3!\cdot J(x,y,z)= -([3]!_{\sigma}\circ[[ \Phi,\Phi]])(x\otimes y\otimes z) \ .
\end{equation}
 for all $x,y,z\in V$.
 The lemma is proved.
 \end{proof}
 
We need the following fact which generalizes  \cite[ Lemma 3.7]{BZ}.
\begin{lemma} 
Let $(A,\{\cdot,\cdot\})$ be a  bracketed $\ZZ_{\ge0}$-graded algebra  in $\C$  generated by $A_{1}$  and such that $A_{0}\cong k$. Then $A$ is Poisson  if and only if the Jacobian  (see \eqref{eq;  Jacobi})  vanishes on $(A_1)^3$.
\end{lemma}

\begin{proof}
 We proceed by induction in homogeneity degrees $\ell= n+m+k$ of monomials $u\cdot v\cdot w$ with  $u\in A_{n},v\in A_{m}, w\in A_{k}$. We start with the base of induction, which is the assumption: suppose that for all $u',v',w'\in A_{1}$ one has $J(u',v',w')=0$.
  Now let   $u\in A_{n},v\in A_{m}, w\in A_{k}, z\in A_{1}$  and assume that the assertion holds for  $\ell= n+m+k$.
We compute using the inductive hypothesis and the Leibniz rules \eqref{eq:left Leibniz rule} and   \eqref{eq:right Leibniz rule}:
 $$J(a,b,c\cdot d)=J(a,b,c)\cdot d+ c (J(a,b, d)) =0\ .$$
Since  for all $u,v,w\in A$ one has $J(u,v,w)=J(w,u,v)=J(v,w,u)$, the the assertion holds for  $ n+m+k=\ell+1$. This implies that $A$ is indeed Poisson.
The lemma is proved.  
\end{proof}

 The above lemma implies that $J(S(V)^3)=0$ if and only if  
 $\Phi$ satisfies
(\ref{eq:CYBE}). Therefore, (a) and (b) are equivalent.

 We will now prove the equivalence of (b) and (c). It follows from  Lemma \ref{le: Schouten v sigma}(a) that $[3]!_{\sigma} \circ[[ \Phi,\Phi]]= [[ \Phi,\Phi]]\circ [3]!_{-{\sigma}}$. Therefore,    $[3]!_{\sigma} \circ[[ \Phi,\Phi]]=0$, if and only if  $ [[ \Phi,\Phi]]\circ [3]!_{-{\sigma}}=0$ , and (b) and (c) are equivalent.

Parts (c) and (d) are clearly equivalent.
Theorem \ref{th: flat Poisson structures} is proved.
\end{proof} 
 
 We will now employ Theorem \ref{th: flat Poisson structures} to study Poisson structures on subspaces and tensor products of decorated spaces.  
 
  First, note the following   fact.
 \begin{proposition}
 \label{pr: subspace split}
Let $V, V'$ and $V''$ be objects of $\C$ such that  $V=V'\oplus V''$. Let additionally,  $(V,\Phi)$  and $(V',\Phi')$ be decorated spaces such that for all $v_1\otimes v_2\in V'\otimes V'$   one has $\Phi'(v_1\otimes v_2)=\pi_{V',V'}\circ\Phi(v_1\otimes v_2)$, where $\pi_{V',V'}: V\otimes V\to V'\otimes V'$ denotes the canonical projection.  
If $(V,\Phi)$ is Poisson, then $(V',\Phi')$ is Poisson.  \end{proposition}
   
   \begin{proof}
 One has $ \Lambda^{3}V=\bigoplus_{i=0}^{3} \Lambda^{i}V'\otimes \Lambda^{3-i}V''$, and $S^{3}V=\bigoplus_{i=0}^{3} S^{i}V'\otimes S^{3-i}V''$. Clearly, $[[\Phi,\Phi]]$ defines a map $[[\Phi,\Phi]]':\Lambda^{3}V'\to S^{3}V$ and it follows from our assertion that    $[[\Phi,\Phi]]'=[[\Phi',\Phi']]:\Lambda^{3}V'\to S^{3}V\subset S^3 V$. If $(V,\Phi)$ is Poisson , then $[[\Phi,\Phi]]'(\Lambda^{3}V')=[[\Phi',\Phi']](\Lambda^{3}V')=0$, and hence $(V',\Phi')$ is Poisson.  \end{proof}

 The following result relates Poisson  structures and tensor products.
 \begin{theorem}
 \label{th: Tensor products of Poissons}
Let $(U,\Phi)$ and $(V,\Phi')$ be decorated spaces. If their tensor product  $(U\otimes V,\Phi'')=(U,\Phi)\otimes (V,\Phi')$  is Poisson,  then $(U,\Phi)$ and $(V,\Phi')$ are Poisson.  \end{theorem}

\begin{proof}
  
 Let $\tilde\sigma$ be the ''shuffle'' $U^{\otimes 3}\otimes V^{\otimes 3}\widetilde \to(U\otimes V)^3  $.  Abbreviating
$$U^{3,0}=S^3 U=(S^2U\otimes U)\cap (U\otimes S^2 U), U^{2,1}=(S^2U\otimes U)\cap (U\otimes \Lambda^2 U)$$
$$U^{0,3}=\Lambda^3 U=(\Lambda^2U\otimes U)\cap (U\otimes \Lambda^2 U), U^{1,2}=(\Lambda^2U\otimes U)\cap (U\otimes S^2 U)$$
and the same for $V$, we have the following containments for $\Lambda^3(U\otimes V)$ and $S^3(U\otimes V)$:
$$\Lambda^3 (U\otimes V)\supseteq \bigoplus_{i+j=3} \tilde\sigma(U^{i,j}\otimes V^{j,i}) \ , $$
$$S^3(U\otimes V)\supseteq \bigoplus_{i+j=3} \tilde\sigma(U^{i,j}\otimes V^{i,j})\ .$$
  Since $\Phi''=\Phi_{13}+\Phi_{24}$ and $[\Phi_{13},\Phi'_{24}]=0\in End( (U\otimes V)^{\otimes 2})$ we have that

$$[[\Phi'',\Phi'']]=[[\Phi_{13}+\Phi'_{24}, \Phi_{13}+\Phi'_{24}]]=[[\Phi_{13} , \Phi_{13}]]+[[ \Phi'_{24},  \Phi'_{24}]] \ .$$
Note that $\Phi(\Lambda^2 U)\subseteq S^2U $ (resp. $\Phi(\Lambda^2 V)\subseteq S^2V$ and $\Phi(S^2 U)\subseteq \Lambda^2 U$ (resp.  $\Phi(S^2 U)\subseteq \Lambda^2 U$). Hence
$$[[\Phi_{13} , \Phi_{13}]](\tilde\sigma( U^{i,j}\otimes V^{j,i}))\subseteq \tilde\sigma( U^{j,i}\otimes V^{j,i})\ ,   [[\Phi'_{24} , \Phi_{24}']](\tilde\sigma( U^{i,j}\otimes V^{j,i}))\subseteq \tilde\sigma( U^{i,j}\otimes V^{i,j}) $$
 for all $i+j=3$. This implies that  
  $$[[\Phi'',\Phi'']](\tilde\sigma( U^{i,j}\otimes V^{j,i}))\subset \tilde\sigma   (U^{j,i}\otimes V^{j,i}+U^{i,j}\otimes V^{i,j})\ .$$ 

 Theorem \ref{th: Tensor products of Poissons} now follows as the special cases $i=3$, $j=0$ and $i=0$, $j=3$ from the following more general obvious result.

\begin{lemma}
If $(U,\Phi)\otimes (V,\Phi')$ is Poisson, then 
 $$[[\Phi'',\Phi'']](\tilde\sigma( U^{i,j}\otimes V^{j,i}))=\{0\}\subseteq \tilde\sigma   (U^{j,i}\otimes V^{j,i}+U^{i,j}\otimes V^{i,j})$$
for all $i+j=3$.
\end{lemma}
  
Theorem \ref{th: Tensor products of Poissons} is proved.  
  \end{proof}

The  following example shows that the converse of Theorem \ref{th: Tensor products of Poissons} does not hold.
   
\begin{example}
 Let $V=V'=\CC^2$ with standard basis $\{e_1,e_2\}$, and let $\Phi(e_i\otimes e_j)=sign(i-j) (e_j\otimes e_i)$ and $\Phi'(e_i\otimes e_j)=\lambda \cdot sign(i-j) (e_j\otimes e_i)$. Clearly, both $(V,\Phi)$ and $(V,\Phi')$ are Poisson, because $\Lambda^3 \CC^2=\{0\}$, but straightforward calculation shows that $(V,\Phi)\otimes (V,\Phi')$ is Poisson, if and only if $\lambda=\pm 1$.   
\end{example}

We conclude this section with an apparently well known and useful observation regarding a general operator $\Phi: V\otimes V\to V\otimes V$ that satisfies the identity $[[\Phi,\Phi]]=0$, the classical Yang-Baxter-Equation.  However, for the reader's convenience we give a proof.
 
\begin{proposition}
\label{pr: r^{-}in r^{+}}
Let $\Phi$ be an operator such that $[[\Phi,\Phi]]=0$, and define $\Phi^{+}=\frac{1}{2}(\Phi+\tau(\Phi))$ and $\Phi^{-}=\frac{1}{2}(\Phi-\tau(\Phi))$. One has
$$[[\Phi^{-},\Phi^{-}]]=-[\Phi^{+}_{12},\Phi^{+}_{23}]\ .$$

\end{proposition}
\begin{proof}
First note the following fact.
\begin{lemma}
Let  $\Phi: V\otimes V\to V\otimes V$ satisfy $[[\Phi,\Phi]]=0$. Then $\Phi^{op}=\sigma\circ\Phi$ satisfies $[[\Phi^{op},\Phi^{op}]]=0$, 
\end{lemma}

\begin{proof}
If $[[\Phi,\Phi]]=0$, then also $\sigma_{13}\circ[[\Phi,\Phi]]\circ\sigma_{13}=0$. Hence,
$$0=[ \Phi_{32},\Phi_{31}]+[\Phi_{32},\Phi_{21}]+[\Phi_{31},\Phi_{21}]=-[[\Phi^{op},\Phi^{op}]]\ .$$
The lemma is proved. 
\end{proof}

We need the following lemma.
\begin{lemma}
 \label{le: r-to r+}
Let $\Phi$ be an operator such that $[[\Phi,\Phi]]=0$. In the notation of Proposition \ref{pr: r^{-}in r^{+}}  one has the following identity:
$[[\Phi^{-},\Phi^{-}]]=-[[ \Phi^{+},\Phi^{+}]]$.
\end{lemma}

\begin{proof}
Since   $[[\Phi,\Phi]]=[\Phi^{op},\Phi^{op}]]=0$ we obtain that $[[\Phi^{-},\Phi^{-}]]=\frac{1}{2}[[\Phi,-\Phi^{op}]]=-\frac{1}{2}[[\Phi,\Phi^{op}]]$ and $[[\Phi^{+},\Phi^{+}]]=\frac{1}{2}[[\Phi,\Phi^{op}]]$.
 The lemma is proved.
\end{proof}

By Lemma \ref{le: Schouten v sigma} we obtain that
$[[\Phi^{-},\Phi^{-}]]-\sigma_{13}\circ[[\Phi^{-},\Phi^{-}]]\circ\sigma_{13}= 2[[\Phi^{-},\Phi^{-}]]$. 
Using Lemma \ref{le: r-to r+}, we obtain 
 $$2[[\Phi^{-},\Phi^{-}]]=-[[\Phi^{+},\Phi^{+}]]+\sigma_{13}\circ[[\Phi^{+},\Phi^{+}]]\circ\sigma_{13}=-2[\Phi_{12}^{+},\Phi_{23}^{+}]\ .$$
Proposition \ref{pr: r^{-}in r^{+}} is proved.
 \end{proof}

\section{Poisson Modules over Lie Algebras}
\label{se:Poissonmods} 

\subsection{Definition and Basic Properties of Poisson Modules}
\label{se:Definition and basic properties of Poisson modules}

Let $(\gg,(\cdot,\cdot))$  be a quadratic complex Lie algebra; i.e. a complex Lie algebra $\gg$ with  a symmetric invariant bilinear form $(\cdot,\cdot):\gg\otimes \gg\to \CC$. Clearly $(\cdot,\cdot)\in (\gg\otimes \gg)^*\cong \gg^*\otimes \gg^*$.  The form  $(\cdot,\cdot)$ defines an isomorphism between $\gg^*$ and $\gg$, and
 under this isomorphism we can identify the form with a symmetric $\gg$-invariant element  $(\cdot,\cdot)=c\in S^2(\gg)\subset \gg\otimes \gg$.  In the case when $\gg$ is semisimple and $(\cdot,\cdot)$ is the Killing form, $c$ is known as the Casimir element.
 Similarly, note that the Lie bracket $[\cdot,\cdot]:\gg\wedge \gg\to \gg$ defines an element $[\cdot,\cdot]:\gg^*\otimes\gg^*\otimes \gg$ and  we obtain under the isomorphism above the canonical element 
 ${\bf c}=[\cdot,\cdot]\in \gg^3$. Observe the following facts.
 
 \begin{lemma}
 \label{le:c-inv.and skew}
 (a) The canonical element ${\bf c}$ is $\gg$-invariant and totally skew symmetric; i.e., ${\bf c}\in (\Lambda^3 \gg)^{\gg}$.
 
 \noindent(b) The elements $c\in S^2\gg$ and ${\bf c}\in \Lambda^3 \gg$ are related by 
 \begin{equation}
  \label{eq:can.element} 
  {\bf c}=-[c_{12}, c_{23}]\ ,
  \end{equation}
   where $c_{12}=c\otimes 1$ and $c_{23}=1\otimes c$.

 \end{lemma}   
 
 \begin{proof}
 Prove (a) first. 
 By definition ${\bf c}= c_{(1)}\otimes c_{(1)}\otimes [c_{(2)},c_{(2)}]=[c_{13},c_{23}]$, where $c=c_{(1)}\otimes c_{(2)}$. Note that $[c_{(2)},c_{(2)}]\ne 0$, despite Sweedler's notation being very suggestive.  
 
Since $c$ is $\gg$-invariant, it is easy to see that ${\bf c}$ is $\gg$-invariant, as well.
We have to prove that ${\bf c}$ is anti-symmetric. We will show first that ${\bf c}$ indeed anti-commutes with the permutation $\sigma_{13}$; i.e., $[c_{13}, c_{12}]=-[c_{13}, c_{23}]$. Since $c$ is $\gg$-invariant  we have  $[c, g\otimes 1]=-[c, 1\otimes g]$ for all $g\in \gg$. Now let $c=c_{(1)}\otimes c_{(2)}$. We obtain, 
$$[c_{13}, c_{12}]=[c_{(1)}, c_{(1)}]\otimes c_{(2)}\otimes c_{(2)}=-c_{(1)}\otimes c_{(1)}\otimes [c_{(2)}, c_{(2)}]=-[c_{13}, c_{23}]\ .$$

We can show analogously that ${\bf c}$ anti-commutes with $\sigma_{23}$, as well. Part(a) is proved and (b) follows immediately. The lemma is proved.
 \end{proof}

 Note that ${\bf c}$ defines for each finite-dimensional $\gg$-module $V$ a $\gg$-module homomorphism    ${\bf c}:\Lambda^3 V\to S^3 V$.
  We make the following definition, and then explain, how it is connected  to the Poisson decorated spaces introduced in Section \ref{se:brack-Poisson}.
 
\begin{definition}
\label{def: Poisson-modules}
Let $(\gg,(\cdot,\cdot))$ be a Lie algebra with  a symmetric invariant bilinear form. We say that a finite-dimensional $\gg$-module $V$ is Poisson, if 
$${\bf c}(\Lambda^3 V)=\{0\}\in S^3 V\ .$$
\end{definition} 

We immediately obtain the following sufficient condition guaranteeing that a $\gg$-module $V$ is Poisson.

\begin{proposition} 
\label{pr:0-hom=poi} 
 Let $V$ be a finite-dimensional $\gg$-module. If $Hom_\gg (\Lambda^3 V, S^3 V)=\{0\}$, then $V$ is Poisson.
\end{proposition} 

 We make the following definition.
\begin{definition}
\label{def:r-matrices}
(a) Let $\gg$ be a  complex  semisimple Lie algebra. An element $ r\in \gg\otimes \gg$ is called a classical $r$-matrix  if $r$ satisfies  the 

\noindent(i) the classical Yang-Baxter-equation, i.e. 
$$[[r,r]]=[r_{12},r_{13}]+[r_{12},r_{23}]+[r_{13},r_{23}]=0 $$

\noindent(ii) and $r+\tau(r)=2c$, where $c$ is the Casimir element of $\gg$.

\end{definition} 

\begin{example}
Let $\gg=sl_2(\CC)$. The standard $r$-matrix is $r=E\otimes F$, where $E,F,H$ are the elements of  the standard basis of $sl_2(\CC)$. 
\end{example} 

 The classical $r$-matrices have been classified in the celebrated paper \cite{BD} in terms of {\it Belavin-Drinfeld triples}.  

Consider a classical $r$-matrix $r$ and its antisymmetrized $r$-matrix  $r^{-}=\frac{1}{2}(r-\tau(r)$, and a finite-dimensional $\gg$-module $V$. The element $r^-\in\gg\otimes \gg$ acts on $V\otimes V$ and the corresponding decorated space $(V,r^{-})$ 
 defines  a bracket  on the symmetric $\gg$-module algebra $S(V)$ defined as  $\{ u,v\}_{r^{-}}=r^{-}(u\wedge  v)$ on all  $u,v\in S(V)$ as constructed in Proposition \ref{pr: Phi bracketed structure}.  
We have the following result.

\begin{proposition}
\label{pr:cv=0=poisson} Let $\gg$ be a complex semisimple Lie algebra, $(\cdot,\cdot)$ the Killing form and  $r$ a classical $r$-matrix, and  let $V$ be a finite-dimensional $\gg$-module. The decorated space $(V,r^{-})$ is Poisson if  and only if  $V$ is Poisson.  
 \end{proposition}

  \begin{proof}
 We have to prove that $(V, r^-)$ is Poisson, if and only if ${\bf c}(\Lambda^{3}V)=\{0\}\in S^3 V$. We obtain from  Proposition \ref{pr: r^{-}in r^{+}} that 
 $$[[r^{-}, r^{-}]]=[r^+_{12},r^+_{23}]=[c_{12}, c_{23}]={\bf c}\ .$$ The assertion now follows from Theorem \ref{th: flat Poisson structures}.
 \end{proof}
 
We note the following facts.

\begin{lemma}
Let $(\gg,(\cdot,\cdot))$ be a quadratic algebra and let $V=V^\gg$ be a trivial $\gg$-module. Then $V$ is Poisson. 
\end{lemma}
\begin{proof}
Obvious.
\end{proof}
\begin{lemma}
\label{le:sums of quadr. algebras}
 Let $(\gg_1,(\cdot,\cdot)_1)$ and $(\gg_2, (\cdot,\cdot)_2$ be quadratic Lie algebras and  let $c_1\in S^2 \gg_1$ and $c_2\in S^2 \gg_2$ be the elements corresponding to $(\gg_1,(\cdot,\cdot)_1)$ and $(\gg_2, (\cdot,\cdot)_2$. Then $(\gg, (\cdot,\cdot))$ is a quadratic Lie algebra with 
$$(g_1+g_2,g_1'+g_2')=(g_1,g_1')_1+(g_2,g_2')_2 $$
for $g_1,g_1'\in\gg_1$ and $g_2,g_2'\in\gg_2$. 
Moreover, one has  $c=c_1+c_2$. 
\end{lemma} 
 
\begin{proof}
 The assertion follows from the fact that the subalgebras $(\gg_1,0)\in\gg$ and $(0,\gg_2)\in\gg$ commute.
 \end{proof}

   The following technical result will be of particular importance for the classification of Poisson modules over a semisimple Lie algebra $\gg$, as it allows to restrict to certain good subalgebras, such as Levi subalgebras (see Proposition \ref{pr:Levisub}).

 \begin{proposition}
 \label{pr: subalgebras}
 Let $(\gg,(\cdot,\cdot))$ be a  quadratic Lie-algebra. Denote by $c\in S^{2}(\gg)$ the $\gg$-invariant element corresponding to $(\cdot,\cdot)$.  Let $\gg_{sub}\subset \gg$ be a subalgebra such that $\gg_{sub}\cap \gg_{sub}^\perp=\{0\}$.   Denote by $c_{sub}\in S^{2}(\gg_{sub})$ the $\gg$-invariant element corresponding to $(\cdot,\cdot)_{\gg_{sub}}$. 
 
 \noindent(a) One has  ${\bf c}={\bf c_{sub}}+{\bf c''}$, where ${\bf c''}\in  \gg_{sub}^{\perp}\wedge \gg\wedge \gg$ in the notation of Appendix \ref{se:appendix}.
 
\noindent (b)  Let $V$ be a $\gg$-module and $V_1\subset V$ a $\gg_{sub}$-module  such that  $\gg_{sub}^{\perp}(V_1)\cap V_1=\{0\}$. If the canonical element  ${\bf c}\in\Lambda^3 \gg$ defined in \eqref{eq:can.element} satisfies ${\bf c}(\Lambda^{3}V)=\{0\}$, then the element    ${\bf c_{sub}}\in\Lambda^3 \gg_{sub}$ satisfies ${\bf c_{sub}}(\Lambda^{3}V_{1})=\{0\}$.

\end{proposition} 
\begin{proof}
Prove (a) first.
By definition  we can express the element  $c\in\gg\otimes \gg$  corresponding to $(\cdot,\cdot)$ as
$c=c_{sub}+c_{rest}$, where  $c_{sub}\in\gg_{sub}\otimes \gg_{sub}$ and   $c_{rest}^{'}\in \gg_{sub}^\perp\otimes \gg\oplus \gg\otimes \gg_{sub}^\perp$.  We obtain from \eqref{eq:can.element} that
$${\bf c}=[(c_{sub}+c_{rest})_{12}, (c_{sub}+c_{rest})_{23}]={\bf c_{sub}}+{\bf c''}\ ,$$

where ${\bf c''}\in \gg_{sub}^\perp\wedge\gg\wedge\gg$. Part(a) is proved.

Prove (b) now.
Recall that  $S^{3}V\cong \bigoplus_{ i=0}^{3} S^{i}V_{1}\boxtimes S^{3-i}V_{2}$. One has  ${\bf c_{sub}}(\Lambda^3 V_1)\subset S^3 V_1$. Clearly,
$${\bf c''}(V_1\otimes V_1\otimes V_1)\subset (S^3 V_1)^c \ ,$$  in the notation of Appendix \ref{se:appendix}.
  If $V$ is Poisson, then ${\bf c}(\Lambda^3 V_1)=\{0\}$, and hence ${\bf c_{sub}}(\Lambda^3 V_1)=\{0\}$ and ${\bf c''}(\Lambda^3 V_1)=\{0\}$. This implies directly that $V_1$ is Poisson as a $\gg_{sub}$-module.

  Proposition \ref{pr: subalgebras} is proved.
\end{proof}

If $\gg$ is a reductive Lie algebra we have the following fact.

 \begin{proposition}
 \label{pr:ssvred}
 Let $\gg$ be a reductive Lie algebra and $\kappa(x,y)=tr(ad(x)ad(y))$. 
 The Lie algebra $\gg$ splits as $\gg=\gg'\oplus \zz$ into a semisimple part  $\gg'$  and a central subalgebra $\zz$. A finite-dimensional $(\gg,\kappa)$-module $V$ is Poisson, if and only if $V$ is Poisson under the restriction to $(\gg',\kappa_{\gg'})$, where $\kappa_{\gg'}$ denotes the restriction of $\kappa$ to $\gg'$, the Killing-form.
 \end{proposition}
 
 \begin{proof}
 Since $\zz.V=\{0\}$ we obtain that $V$ is Poisson only if $V$ is Poisson as a $\gg'$-module  by applying Proposition  \ref{pr: subalgebras} (b) to $V=V'$.
 In order to prove the other direction note that since $\gg$ and $\zz$ commute we have $(\gg,\kappa)=(\gg',\kappa_{\gg'})\oplus (\zz,\kappa_\zz)$ and can  apply Lemma \ref{le:sums of quadr. algebras}  to obtain that $c=c_{\gg'}+c_\zz$. The Lie algebra $\zz$ is Abelian, and therefore, the form vanishes on $\zz\otimes \zz$ and $c_\zz= 0$.  This implies that $c=c_{\gg'}$ and ${\bf c}={\bf c_{\gg'}}$. The assertion now follows immediately.
 \end{proof}

  \subsection{Classification of Poisson Modules over Semisimple Lie Algebras}
  \label{se: poisson over simple}
  
 In this section we will classify all simple Poisson modules over a semisimple Lie algebra $\gg$. By Proposition \ref{pr:ssvred} we immediately obtain a classification of all simple modules over reductive Lie algebras.  First we will introduce some notation. Choose a Borel subalgebra $\bb\subset\gg $ and denote by $\hh$ and  $\nn^+$ the corresponding Cartan  and uper nilpotent subalgebras, and, similarly,  by  $\bb^-$ and $\nn^-$ the lower Borel and nilpotent sublagebras. By $W(\gg)$ we shall denote the Weyl group of $\gg$ and by $(\cdot,\cdot)_{\hh}$ and $(\cdot,\cdot)_{\hh^*}$ the standard inner product on $\hh$ and $\hh^*$, which we identify via the inner product.  Denote by $R(\gg)\subset \hh^*$ the set of roots, by $R^+(\gg)$ (resp. $R^-(\gg)$) the set of positive (resp. negative) roots and  by $\Delta=\{\alpha_1,\ldots,\alpha_n\}$ the set of simple  roots.  Denote by $E_{\alpha}$ for $\alpha\in R(\gg)$ and $H_{\alpha}\subset \hh$, $\alpha\in \R^+(\gg)$ the standard generators of $\gg$ with the property that $[E_\alpha,E_{-\alpha}]=H_\alpha=\check\alpha =2\frac{\alpha}{(\alpha,\alpha)}\in \hh\subset \gg$. We will also use the notation $P(\gg)$ for the weight-lattice of $\gg$ and $\omega_i$ for the $i$-th fundamental weight.

 We now introduce the notion of {\it geometrically decomposable} modules following \cite[ch.4]{Ho}. Let $\gg$ be a reductive Lie algebra and $V$ a $\gg$-module and $U\subset V$ be a $\bb$-module. Denote by $\det(U)$ the one-dimensional subspace $\det(U)=\Lambda^{top}U$ of $\Lambda(U)$.  Clearly, $\det(U)$  is a $\bb$-submodule of $\Lambda (V)$, therefore, every $u\in \det(U)$ is a  highest weight vector in $\Lambda(V)$. In analogy to \cite[ch. 4.6]{Ho}, we  call a highest weight vector $v\in\Lambda(V)$ {\it geometric}, if $v\in\det(U)$ for some $\bb$-module $U\subset V$. 
 \begin{definition}\cite[ch. 4.6]{Ho}
 \label{def: geometrically decomposable}
 A $\gg$-module $V$ is called {\it geometrically decomposable},  if  $\Lambda V$ is generated as a $\gg$-module by geometric highest weight vectors.
  \end{definition}

The following result is the first main theorem of this section. 
 
  \begin{maintheorem}
 \label{th: class of Poisson}
 Let $\gg$ be a simple complex Lie algebra, and let $V$ be a non-trivial simple  finite-dimensional $\gg$-module.
Then the following are equivalent:
 
\noindent (a) The module $V$ is Poisson. 

\noindent (b)  The decorated space  $(V,r^{-})$ is Poisson for any classical $r$-matrix $r\in \gg\otimes \gg$ .  

\noindent (c) $Hom_\gg(\Lambda^3V,S^3V)=\{0\}$.
 
 \noindent(d) $\Lambda^2 V$ is simple  or $(\gg,V)=(sp_{2n}(\CC),\CC^{2n})$ for some $n$.
 
 \noindent(e) The module $V$ is a geometrically decomposable $\gg$-module or $(\gg,V)=(sp(2n), V_{\omega_1})$.

\noindent(f) The pair $(\gg,V)$ is one of the following:
 
  (i) $(sl_n(\CC),V_\lambda)$ where $\lambda\in\{\omega_1,2\omega_1,\omega_2,\omega_{n-2},\omega_{n-1}, 2\omega_{n-1}\}$.

 (ii) $(so(n),V_{\omega_1})$,$(so(5),V_{\omega_2})$, $(so(8),V_{\omega_3})$, $(so(8),V_{\omega_4})$,    $ (so(10),V_{\omega_4})$ and $  (so(10),V_{\omega_5})$.

 (iii) $(sp(2n), V_{\omega_1})$ and $(sp(4),V_{\omega_2})$.

(iv) $(E_6,V_{\omega_1})$ and $(E_6, V_{\omega_6})$. 
 
 \end{maintheorem}

We prove the theorem using the following strategy. The equivalence of (a) and (b) is proved in Proposition \ref{pr:cv=0=poisson}. The implication (c) implies (a) follows from Proposition \ref{pr:0-hom=poi}. To prove that (a) yields (c) and (f)  and, we will give necessary conditions a dominant weight $\lambda\in P^+(\gg)$ has to satisfy, if $V_\lambda$ is Poisson. We will then show that  $Hom_\gg(\Lambda^3V_\lambda,S^3V_\lambda)=\{0\}$ for all $\lambda\in P^+(\gg)$ satisfying these necessary conditions proving that (f) yields (c). In order to show that (a) and (d) are equivalent, we prove that $V$ is Poisson if $\Lambda^2 V$ is simple and that $\Lambda^2 V$ is simple for all pair $(\gg,V)$ in (f). The equivalence of (a) and (e) then follows from the classification of simple geometrically decomposable modules in \cite{Ho}. Since the proof is rather lengthy we refer it to  Section \ref{se:proof of classification}

We can generalize Theorem \ref{th: class of Poisson} to the case of semisimple Lie algebras. Consider a semisimple Lie algebra $\gg$ and a finite-dimensional $\gg$-module $V$. If $\gg=\bigoplus_{i=1}^n \gg_i$, where $\gg_i$ are simple Lie algebras,  denote by $V_{\lambda_1,\ldots,\lambda_n}$   the simple $\gg$-module of highest weight $(\lambda_1,\ldots, \lambda_n)\in P(\gg_1)\oplus\ldots P(\gg_n)\cong P(\gg)$. Denote by the  support $supp_\gg(V)$ of a $\gg$-module $V$ the product of all simple factors $\gg_i$ for which $\gg_i(V)\ne \{0\}$.  We have the following classification result.  

\begin{theorem}
\label{th:class:poissonsemisimple}
Let $\gg$ be a semisimple Lie algebra and $V$ a simple $\gg$-module. The following are equivalent:

\noindent(a) $V$ is Poisson.

\noindent(b) The pair $(supp(\gg),V)$ is listed in Theorem  \ref{th: class of Poisson} (f) or $(supp_\gg(V),V)=(sl_m\times sl_n, V_{\omega_1,\omega_1})$ where $V_{\omega_1,\omega_1}$ is the natural  $sl_m\times sl_n$-module.
 
\end{theorem}

\begin{proof}
Recall that the $m\times n$-matrices $Mat_{m\times n}(\CC)$ can be given a $gl_m\times gl_m$-module such that $Mat_{m\times n}(\CC)\cong V_{\omega_m,\omega_n}\cong V_{\omega_1,\omega_1}^*$ with $gl_m$ acting on the left and $gl_n$ acting on the right. This action yields a $\gl_m\times gl_n$-module algebra  structure on $\CC[Mat_{m\times n}]=S(V_{\omega_1,\omega_1})$.
  It is well known that the $r$-matrix bracket defines a Poisson structure on the algebra $\CC[Mat_{m\times n}]=S(V_{\omega_1,\omega_1})$ via 
  $$r^-(x_{ij}\otimes x_{k\l})=(sign(i-k)+sign(j-l)) x_{kj}x_{i\l}\ .$$ 
  
  It remains to show that if $supp_\gg (V)$ is non-simple and  $(supp_\gg(V), V)\ne  (sl_m\times sl_n, V_{\omega_1,\omega_1})$, then $V$ is not Poisson.

 

 Let $r_1,\ldots r_n$ be classical $r$-matrices for $\gg_1,\ldots,\gg_n$. It is easy to see that $r=r_1+\ldots r_n$ is a classical $r$-matrix for $\gg$.
  Recall that as a vector space $V$ can be decomposed as a tensor product $V=V_{\lambda_1}\otimes \ldots \otimes V_{\lambda_n}$, where $V_{\lambda_i}$ is a simple $\gg_i$-module. 
 The decorated space $(V,r^-)$ decomposes as a tensor product $(V,r^-)=(V_{\lambda_1},r^-_1)\otimes\ldots\otimes  (V_{\lambda_n},r^-_n)$.
   It now follows from Theorem \ref{th: Tensor products of Poissons}  that if a simple $\gg$-module $V$ is  Poisson, then each $V_{\lambda_i}$ is Poisson as a $\gg_i$-module as are all the products $V_{\lambda_i}\otimes V_{\lambda_{i+1}}$ as $\gg_i\oplus\gg_{i+1}$-modules. It therefore suffices to show the following.  First, let $\gg_1$ and $\gg_2$ be simple Lie algebras and let  $V_{\lambda_1}$ and $V_{\lambda_2}$ be simple Poisson $\gg_1$- (resp. $\gg_2$)-modules and  and $(\gg_2, V_{\lambda_2})\ne (sl_k,V_{\omega_1})$ for some $k\ge 2$. Then $V_{\lambda_1,\lambda_2}$ is not Poisson. Second, we have to prove that the natural  $sl_\ell\times sl_m\times sl_n$-module$V_{\omega_1,\omega_1,\omega_1}$  is not Poisson for all $\ell,m,n\ge 2$.
 
We can further reduce the list of cases to investigate by considering the embedding of some Levi subalgebra in $\gg$ and making  use of Proposition \ref{pr:Levisub}.  We need the following result. 
 
 \begin{proposition}
 \label{pr:list semisimple levis}

 \noindent(a)If $\gg=sl_2\times sl_2$, then $V=V_{2,i}$ is not Poisson, where $V_i$ denotes the  $n+1$-dimensional simple $sl_2$-module.

 
 
  \noindent(b) If $\gg=sl_2\times sl_2\times sl_2$, then $V_{1,1,1}$ is not Poisson.
 
 \noindent(c) If $\gg=sl_2\times sp(4)$, then $V=V_{1,\omega_1}$ is not Poisson.

 \end{proposition}
 
 \begin{proof}
   Let $\gg=\gg_1\oplus \gg_2$ be a semisimple Lie algebra.  We have $c=c_1+c_2$, where $c$, $c_1$ and $c_2$ are the Casimir elements of $\gg$, $\gg_1$ and $\gg_2$, respectively and ${\bf c}={\bf c_1}+{\bf c_2}$. 

We will first prove case (a).  Let $\gg=sl_2\times sl_2$ and let $V=V_{i,2}$, $i\ge1$, be a simple $\gg$-module. Denote by $V_i$ and $V_2$ the corresponding simple $sl_2$-modules. Since $V\cong V_i\otimes V_2$ as vector spaces we can choose non-zero vectors $(u\otimes v), (u'\otimes v'),(u\otimes v'')\in V_i\otimes V_2\subset V$ such that $u\in V_i(i)$, $u'=F(u)$ and $v\in V_2(2)$, $v'=F(v)$ and $v''\in V_2(-2)$, where $V_k(\ell)$  denotes the $\ell$-weight space of $V_k$.  Abbreviate ${\bf uv}=(u\otimes v)\wedge(u'\otimes v')\wedge(u\otimes v'')\in \Lambda^3 V$.

 Note that if $\gg=sl_2$, then ${\bf c}_{sl_2}=E\wedge F\wedge H$  and ${\bf c}_{sl_2\times sl_2} =E_1\wedge F_1\wedge H_1+E_2\wedge F_2\wedge H_2$. It is easy to verify that
 
 $$ E_2\wedge F_2\wedge H_2({\bf uv}) \in V(i,2)\cdot V(i-2,-2)\cdot V(i,0)\oplus V(i,0)\cdot V(i-2,2)\cdot V(i,-2)\subset S^3V\ .$$
 
 Similarly we obtain that 
 $$E_1\wedge F_1\wedge H_1({\bf uv})\in \left(V(i,2)\cdot V(i-2,-2)\cdot V(i,0)\oplus V(i,0)\cdot V(i-2,2)\cdot V(i,-2)\right)^c \ .$$
 
Hence, ${\bf c}({\bf uv})\ne 0$ and $V$ is not Poisson. Part (a) is proved.

Prove (b) next. Denote by $V_{111}$ the $8$-dimensional natural  $(sl_2)^3(\CC)$-module.  Choose a basis  of $V_{1,1,1}$ with basis vectors $x_{j_1,j_2,j_3}$, $j_i\in\{0,1\}$, such that $ x_{j_1,j_2,j_3}$ is a weight vector of weight $(1-2j_i)$ of each subalgebra $\gg_i$, the $i$-th copy of $sl_2(\CC)$ in $\gg$. Moreover, we can choose the basis  such that $F_1(x_{ijk})=\delta_{i,0} x_{1,j,k}$ and $E_1(x_{ijk})=\delta_{i,1} x_{0,jk}$, and analogously for $E_2,F_2, E_3$ and $F_3$. 

We have ${\bf c}= \sum_{i=1}^3- E_i\wedge H_i\wedge F_i$.

It is easy to compute  that
$${\bf c}(x_{111}\wedge x_{000}\wedge x_{100})=$$
$$ -x_{011}x_{100}^2+x_{111}x_{100}x_{000}+x_{101}x_{010}x_{100}-x_{101}x_{000}x_{110}+x_{110}x_{001}x_{100}-x_{110}x_{000}x_{101} \ne0\ .$$
 Therefore ${\bf c}(\Lambda^3 V_{111})\ne \{0\}$, hence $V_{111}$ is not Poisson. Part(b) is proved.
  
 It remains to  prove part (c). Let $V_{\omega_1}$ be the four-dimensional natural $sp(4)$-module and let $V_{1,\omega_1}$ be the natural $sl_2\times sp(4)$-module. As in the proof of parts (a) and (b) choose $u,u'\in V_1$ such that $ u\in V_1(1)$ and $u'\in V_1(-1)$ and $v \in V_{\omega_1}(\omega_1)$, $v'=F_{\alpha_1}(v)$ and $v''\in V_{\omega_1}(-\omega_1)$. Denote  ${\bf uv}=(u\otimes v)\wedge(u'\otimes v')\wedge(u\otimes v'')\in \Lambda^3 V$. We have by \eqref{eq:c-explixit}
 ${\bf c}=E\wedge F\wedge H+\sum_{\alpha, \beta\in R(sp(4))}\frac{(\alpha,\alpha)(\beta,\beta)}{4} E_{\alpha}\wedge E_{\beta}\wedge [E_{-\alpha}, E_{-\beta}]$. We obtain that
 
 $$ E_{\alpha_1}\wedge E_{-\alpha_1}\wedge H_{\alpha_1}({\bf uv})\in V(-1, \omega_1)\cdot V(1,\omega_1-\alpha_1)\cdot V(1,-\omega_1)$$
and observe that indeed
 $${\bf c}- \frac{(\alpha_1,\alpha_1)^2}{4}E_{\alpha_1}\wedge E_{-\alpha_1}\wedge H_{\alpha_1}({\bf uv})\in (V(-1, \omega_1)\cdot V(1,\omega_1-\alpha_1)\cdot V(1,-\omega_1))^c\ .$$

This implies that ${\bf c}({\bf uv})\ne 0$ and  that $V_{1,\omega_1}$ is not Poisson.  Part (c) and the proposition are proved.
   
 \end{proof}   
 
 We now return to the proof of Theorem \ref{th:class:poissonsemisimple}.
 Now let $\gg=\gg_1\oplus\gg_2$ such that $\gg_1$ and $\gg_2$ are two simple Lie algebras, and $V_{\lambda_1}$ and $V_{\lambda_2}$ simple  Poisson $\gg_1$, respectively $\gg_2$-modules and assume that $(\gg,V)\ne(sl_m\times sl_n,V_{\omega_1,\omega_1})$. We will list the semisimple part $\gg'\subset \gg$ of the  Levi subalgebras  and the corresponding simple $\gg'$-module $V'\subset V_{\lambda_1,\lambda_2}$ verifying that $V_{\lambda_1,\lambda_2}$ is not Poisson. First, note  the following fact.

 \begin{proposition}
 \label{pr:cubesarenotpoisson}
 Let $\gg$ be a semisimple Lie algebra and $V$ a simple $\gg$-module such that $supp_\gg(V)$ has at least three simple factors. Then $V$ is not Poisson.
 \end{proposition}

 \begin{proof}
 
 Note the following fact.
 
 \begin{lemma}
 Let $\gg=sl_2^3$ and let $V=V_{i,j,k}$ be a simple finite-dimensional $\gg$-module with $0\notin\{i,j,k\}$. Then $V$ is not Poisson.
 \end{lemma}
 
 \begin{proof}
 Since $V_\ell$ is not Poisson if $\ell\ge 3$ by Theorem \ref{th: class of Poisson} (f), we obtain from Theorem \ref{th: Tensor products of Poissons} that  if $V_{i,j,k}$ is Poisson, then $i,j,k\le 2$. 
 If $i=j=k=1$, then the assertion of the lemma agrees with the assertion of  Proposition \ref{pr:list semisimple levis} (b). Now suppose, without loss of generality, that $j=1$. Then $V_{i,j}$ is not Poisson by  Proposition \ref{pr:list semisimple levis} (b) and  $V=V_{i,j,k}$ is not Poisson by Theorem \ref{th: Tensor products of Poissons}. The lemma is proved.
 \end{proof}

 Suppose  $supp_\gg(V)$ has at least three simple factors. We can find a Levi subalgebra $\gg'\cong sl_2^3$ such that a highest weight vector $v\in V$ generates a $\gg'$-submodule $V'\cong V_{i,j,k}$ with $i,j,k\ge1$. Hence $V$ is not Poisson by the previous lemma and  Proposition \ref{pr:Levisub}. The proposition is proved.
 \end{proof}

Now we are able to complete the proof of Theorem \ref{th:class:poissonsemisimple}.
We  assume that $supp_\gg(V)$ has two simple factors. 
In order to deal with most cases, it suffices to exhibit a Levi subalgebra $\gg'\subset\gg$ and a simple module $V'\subset_{\gg'} V$ such that $(\gg',V')\in\{(sl_2\times sl_2, V_{i,2}),(sl_2\times sp(4), V_{1,\omega_1})\}$  to show that $(V,\gg)$ is not Poisson by Proposition \ref{pr:Levisub}. Since the choice is obvious in a large number of cases, and   a complete list would, therefore, be rather long, we will list  only the non-obvious choices. All these special cases except for the first one require to us to consider Levi subalgebras with three simple factors.

\noindent(a) If $\gg=so(2n+1)\oplus \gg_2$ and $V=V_{\omega_1,\lambda}$ choose $\gg'=sl_2\times sl_2$ generated by the second node of the Dynkn diagram of $so(2n+1)$, resp.  a node $i$ of the diagram associated to  $\gg_2$ such that $(\lambda,\alpha_i)\ge 1$.  Note that $E_{-\alpha_1}(v_{\omega_1})\in V_{\omega_1}$ generates a three-dimensional simple $sl_2$-module for the subalgebra corresponding to the second node of the Dynkin diagram. Hence, we find a  $sl_2\times sl_2$-submodule $V'\cong V_{2,i}\subset V_{\omega_2,\lambda}$ and $V$ is not Poisson by  Proposition \ref{pr:list semisimple levis} (a).


\noindent(b) If $\gg=sl_n\times \gg_2$, $n\ge 4$ and $V=V_{\omega_2,\lambda}$, (resp. $V_{\omega_{n-2},\lambda}$) choose $\gg'=(sl_2\times sl_2)\times \gg_2$ generated by the first and third nodes of the Dynkn diagram $A_{n-1}$  (resp. the last and third to last nodes) and  $\gg_2$. Let $v_{\omega_2}$ be a highest weight vector in $V_{\omega_2}$. Note that $E_{-\alpha_2}(v_{\omega_2})\in V_{\omega_2}$ generates a four-dimensional simple $sl_2\times sl_2$-module $V_{1,1}$. Hence, we find a  $sl_2\times sl_2\times \gg_2$ submodule $V'\cong V_{1,1,\lambda}\subset V_{\omega_2,\lambda}$ and $V$ is not Poisson by  Proposition \ref{pr:cubesarenotpoisson}. Similarly we obtain that  $V_{\omega_{n-2},\lambda}$ is not Poisson.

\noindent(c) If $\gg=so(2n)\oplus \gg_2$,  and $V=V_{\omega_1, \lambda}$, consider the Levi subalgebra of $so(2n)$, isomorphic to $sl_4$, generated by the  $(n-2)$nd, $(n-1)$st and the $n$th nodes of the Dynkin diagram $D_n$. It can be easily observed that if $v\in V_{\omega_1}(\omega_1)$ is a highest weight vector, then $v'=E_{\alpha_{n-3}}\circ\ldots\circ E_{\alpha_1}(v)$ generates a simple $sl_4$-module $V_{\omega_2}$.  We obtain that $V$ is not Poisson by applying the argument in case (b).

\noindent(d) If $\gg=so(8)\oplus \gg_2$ and $V=V_{\omega_i, \lambda}$ , $i=3,4$ or $\gg=so(10)\oplus \gg_2$ and $V=V_{\omega_i, \lambda}$ , $i=4,5$ we argue analogous to case (c).

\noindent(e)  If $\gg=E_6\oplus \gg_2$ and $V=V_{\omega_1,\lambda}$ consider the Levi subalgebra $ sl_4 \subset E_6$ generated by the second, third and fourth nodes of the Dynkin diagram $E_6$
 (in the notation of \cite{Bou}). If $v\in V_{\omega_1}(\omega_1)$ is a highest weight vector, then $v'=E_{\alpha_{3}}\circ E_{\alpha_1}(v)$ generates a simple $sl_4$-module  $V_{\omega_2}$.  We obtain that $V$ is not Poisson by applying the argument in case (b).

The proof of Theorem \ref{th:class:poissonsemisimple} is now complete. 
 \end{proof}

\section{Quantum Symmetric Algebras}
\label{se:prlime}
\subsection{The Quantum Group $U_q(g)$ and its Modules}
\label{se:q-group}

We start with the definition of the quantized enveloping algebra
associated with a complex reductive Lie algebra $\gg$ (our standard
reference here will be \cite{brown-goodearl}). Let $\hh\subset \gg$
be a Cartan subalgebra,  $P(\gg)$ the weight lattice, as introduced above, and let $A=(a_{ij})$ be the Cartan matrix
for $\gg$. Additionally, let $(\cdot,\cdot)$ be the standard non-degenerate symmetric bilinear form on $\hh$.
 
The {\it quantized enveloping algebra} $U$ is a $\CC(q)$-algebra generated
by the elements $E_i$ and $F_i$ for $i \in [1,r]$, and $K_\lambda $ for $\lambda \in P(\gg)$,
subject to the following relations:
$K_\lambda  K_\mu = K_{\lambda +\mu}, \,\, K_0 = 1$
for $\lambda , \mu \in P$; $K_\lambda E_i =q^{(\alpha_i\,,\,\lambda)} E_i K_\lambda,
\,\, K_\lambda F_i =q^{-(\alpha_i\,,\,\lambda)} F_i K_\lambda
$
for $i \in [1,r]$ and $\lambda\in P$;
\begin{equation}
\label{eq:upper lower relations}
E_i,F_j-F_jE_i=\delta_{ij}\frac{K_{\alpha_i}- K_{-\alpha_i}}{q^{d_i}-q^{-d_i}}
\end{equation}
for $i,j \in [1,r]$, where  $d_i=\frac{(\alpha_i\, ,\,\alpha_i)}{2}$;
and the {\it quantum Serre relations}
\begin{equation}
\label{eq:quantum Serre relations}
\sum_{p=0}^{1-a_{ij}} (-1)^p 
E_i^{(1-a_{ij}-p)} E_j E_i^{(p)} = 0,~\sum_{p=0}^{1-a_{ij}} (-1)^p 
F_i^{(1-a_{ij}-p)} F_j F_i^{(p)} = 0
\end{equation}
for $i \neq j$, where
the notation $X_i^{(p)}$ stands for the \emph{divided power}
\begin{equation}
\label{eq:divided-power}
X_i^{(p)} = \frac{X^p}{(1)_i \cdots (p)_i}, \quad
(k)_i = \frac{q^{kd_i}-q^{-kd_i}}{q^{d_i}-q^{-d_i}} \ .
\end{equation}

%

The algebra $U$ is a $q$-deformation of the universal enveloping algebra of
the reductive Lie algebra~$\gg$, so it is commonly
denoted by $U = U_q(\gg)$.
It has a natural structure of a bialgebra with the co-multiplication $\Delta:U\to U\otimes U$
and the co-unit homomorphism  $\varepsilon:U\to \QQ(q)$
given by
\begin{equation}
\label{eq:coproduct}
\Delta(E_i)=E_i\otimes 1+K_{\alpha_i}\otimes E_i, \,
\Delta(F_i)=F_i\otimes K_{-\alpha_i}+ 1\otimes F_i, \, \Delta(K_\lambda)=
K_\lambda\otimes K_\lambda \ ,
\end{equation}
\begin{equation}
\label{eq:counit}
\varepsilon(E_i)=\varepsilon(F_i)=0, \quad \varepsilon(K_\lambda)=1\ .
\end{equation}
In fact, $U$ is a Hopf algebra with the antipode anti-homomorphism $S: U \to U$ given by 
\begin{equation}
\label{eq:antipode}
 S(E_i) = -K_{-\alpha_i} E_i, \,\, S(F_i) = -F_i K_{\alpha_i}, \,\,
S(K_\lambda) = K_{-\lambda}\ .
\end{equation}

Let $U^-$ (resp.~$U^0$; $U^+$) be the $\QQ(q)$-subalgebra of~$U$ generated by
$F_1, \dots, F_r$ (resp. by~$K_\lambda \, (\lambda\in P)$; by $E_1, \dots, E_r$).
It is well-known that $U=U^-\cdot U^0\cdot U^+$ (more precisely,
the multiplication map induces an isomorphism $U^-\otimes U^0\otimes U^+ \to U$).



We will consider the full sub-category  $\OO_{f}$ of the category
$U_q(\gg)-Mod$. The objects of $\OO_{f}$ are finite-dimensional
$U_{q}(\gg)$-modules $V^q$ having a weight decomposition
$$V^q=\oplus_{\mu\in P} V^q(\mu)\ ,$$
where each $K_\lambda$ acts on each {\it weight space} $V^q(\mu)$ by
the multiplication with $q^{(\lambda\,|\,\mu)}$  (see e.g.,
\cite{brown-goodearl}[I.6.12]). The category $\OO_{f}$ is semisimple
and the irreducible objects $V^q_\lambda$ are  generated by highest
weight spaces $V^q_\lambda(\lambda)=\CC(q)\cdot v_\lambda$, where
$\lambda$ is a {\it dominant weight}, i.e, $\lambda$ belongs to
$P^+=\{\lambda\in P:(\lambda\,|\,\alpha_i)\ge 0~ \forall ~i\in
[1,r]\}$, the monoid of dominant weights.

 By definition,the universal $R$-matrix $R \in U_{q}(\gg)\widehat \otimes U_{q}(\gg)$$R$ has can be decomposed as
\begin{equation}
\label{eq:Jordan}
R=R_0R_1=R_1R_0
\end{equation}
 where $R_0$ is ''the diagonal part'' of $R$, and  $R_1$ is unipotent, i.e., $R_1$ is a formal power series
\begin{equation}
\label{eq:R1}
R_1=1\otimes 1+(q-1)x_1+ (q-1)^2x_2+\cdots \ ,
\end{equation}
where all $x_k\in {U'}^-_k\otimes_{\CC[q,q^{-1}]} {U'}^+_k$, where
${U'}^-$ (resp. ${U'}^+$) is the integral form of $U^+$, i.e.,
${U'}^-$ is a $\CC[q,q^{-1}]$-subalgebra of $U_q(\gg)$ generated by
all $F_i$ (resp. by all $E_i$) and ${U'}^-_k$ (resp. ${U'}^+_k$) is
the $k$-th graded component under the grading $deg(F_i)=1$ (resp.
$deg(E_i)=1$).

By definition, for any $U^q,V^q$ in $\OO_f$ and any highest weights elements $u_\lambda\in U^q(\lambda)$,  $v_\mu\in V^q(\mu)$ we have
$R_0(u_\lambda\otimes v_\mu)=q^{(\lambda\,|\,\mu)}u_\lambda\otimes v_\mu$.

Let $R^{op}$ be the opposite element of $R$, i.e, $R^{op}=\tau (R)$, where  $\tau:U_{q}(\gg)\widehat \otimes U_{q}(\gg)\to U_{q}(\gg)\widehat \otimes U_{q}(\gg)$
is the permutation of factors. Clearly, $R^{op}=R_0R_1^{op}=R_1^{op}R_0$.

Following \cite[Section 3]{DR1},  define $D\in  U_{q}(\gg)\widehat \otimes U_{q}(\gg)$ by
\begin{equation}
\label{eq:D}
D:=R_0\sqrt{R_1^{op}R_1}=\sqrt{R_1^{op}R_1}R_0 \ .
\end{equation}

Clearly, $D$ is well-defined because $R_1^{op}R_1$ is also unipotent as well as its square root.  By definition,
$D^{2}=R^{op}R$, $D^{op}R=RD$.

Furthermore, define
\begin{equation}
\label{eq:hat R}
\widehat R:=RD^{-1}=(D^{op})^{-1}R=R_1\left(\sqrt{R_1^{op}R_1}\right)^{-1}
\end{equation}
It is easy to see that
\begin{equation}
\label{eq:unitary}
 \widehat R^{op}=\widehat R^{\,\,-1}
\end{equation}
According to \cite[Proposition 3.3]{DR1}, the pair $(U_{q}(\gg), \widehat R)$ is a {\it coboundary} Hopf algebra.

The braiding in the category $\OO_{f}$ is defined by $\RR_{U^q,V^q}:U^q\otimes V^q\to V^q\otimes U^q$, where
$$\RR_{U^q,V^q}(u\otimes v)=\tau R(u\otimes v)$$
for any $u\in U^q$, $v\in V^q$, where $\tau:U^q\otimes V^q\to V^q\otimes U^q$ is the ordinary permutation of factors.

Denote by  $C\in Z(\widehat{U_{q}(\gg)})$   the {\it quantum
Casimir} element which acts on any irreducible $U_q(\gg)$-module
$V^q_\lambda$ in $\OO_f$ by the scalar multiple
$q^{(\lambda\,|\,\lambda+2\rho)}$, where $2\rho$ is the sum of
positive roots.

The following fact is  well-known.
\begin{lemma}
\label{le:braiding casimir} One has $\RR^2=\Delta(C^{-1})\circ
(C\otimes C)$. In particular, for each $\lambda,\mu,\nu\in  P_+$ the
restriction of $\RR^2$ to the $\nu$-th isotypic component
$I^\nu_{\lambda,\mu}$ of the tensor product $V^q_\lambda\otimes V^q_\mu$
is scalar multiplication by
$q^{(\lambda\,|\,\lambda)+(\mu\,|\,\mu)-(\nu\,|\,\nu))+(2\rho\,|\,\lambda+\mu-\nu)}$.

\end{lemma}

This allows to define the diagonalizable $\CC(q)$-linear  map
$D_{U^q,V^q}:U^q\otimes V^q\to U^q\otimes V^q$ by $D_{U^q,V^q}(u\otimes v)=
D(u\otimes v)$ for any objects $U^q$ and $V^q$ of $\OO_f$. It is easy to
see that the operator $D_{V^q_\lambda,V^q_\mu}:V^q_\lambda\otimes V^q_\mu\to
V^q_\lambda\otimes V^q_\mu$ acts on the $\nu$-th isotypic component
$I^\nu_{\lambda,\mu}$ in $V^q_\lambda\otimes V^q_\mu$ by the scalar
multiplication with
$q^{\frac{1}{2}(\,(\lambda\,|\,\lambda)+(\mu\,|\,\mu)-(\nu\,|\,\nu)\,)+(\rho\,|\,\lambda+\mu-\nu)}$.

For any  $U^q$ and $V^q$ in $\OO_f$ define the {\it normalized braiding} $\sigma_{U^q,V^q}$ by
\begin{equation}
\label{eq:sigma}
\sigma_{U^q,V^q}(u\otimes v)=\tau \widehat R (u\otimes v) \ ,
\end{equation}

Therefore, we have by (\ref{eq:hat R}):
\begin{equation}
\label{eq:formula for sigma}
\sigma_{U^q,V^q}=D_{V^q,U^q}^{-1} \RR_{U^q,V^q}=\RR_{U^q,V^q}D_{U^q,V^q}^{-1} \ .
\end{equation}

We will will sometimes write $\sigma_{U^q,V^q}$ in a more explicit way:
\begin{equation}
\label{eq:sigma root}
\sigma_{U^q,V^q}=\sqrt{\RR_{V^q,U^q}^{-1}\RR_{U^q,V^q}^{-1}} \RR_{U^q,V^q}=\RR_{U^q,V^q}\sqrt{\RR_{U^q,V^q}^{-1}\RR_{V^q,U^q}^{-1}}
\end{equation}


The following fact is an obvious corollary of (\ref{eq:unitary}).
\begin{lemma}
\label{le:symm comm constraint} $\sigma_{V^q,U^q}\circ
\sigma_{U^q,V^q}=id_{U^q\otimes V^q}$ for any  $U^q,V^q$ in $\OO_f$. That is,
$\sigma$ is a symmetric commutativity constraint.
\end{lemma}

We also have the following coboundary relation (even though we will not use it).

\begin{lemma} \cite[section 3]{DR1}
Let $A^q,B^q,C^q$ be objects of $\OO_{f})$. Then, the following diagram commutes:
\begin{equation}
\label{eq:comm square}
\begin{CD}
A^q\otimes B^q \otimes C^q@>\sigma_{12,3}>>C^q\otimes A^q \otimes B^q\\
@V\sigma_{1,23}VV@VV\sigma_{23}V\\
B^q\otimes C^q \otimes A^q@>\sigma_{12}>>C^q\otimes B^q \otimes A^q
\end{CD}
\end{equation}
where we abbreviated
$$\sigma_{12,3}:=\sigma_{A^q\otimes B^q,C^q}:(A^q\otimes B^q) \otimes C^q\to  C^q\otimes (A^q \otimes B^q),$$
$$\sigma_{1,23}:=\sigma_{A^q, B^q\otimes C^q}:A^q \otimes (B^q\otimes C^q) \to  (B^q\otimes C^q)\otimes A^q \ .$$

\end{lemma}

\begin{remark} If one replaces the braiding  $\RR$ of $\OO_f$  by its inverse $\RR^{-1}$,
the symmetric commutativity constraint $\sigma$ will not change.
\end{remark}
 
\subsection{Braided Symmetric and Exterior Powers}
\label{se:BESP}
In this section we will use the notation and conventions of Section \ref{se:q-group}.

For any morphism $f:V^q\otimes V^q\to V^q\otimes V^q$ in $\OO_f$ and $n>1$
we denote by $f^{i,i+1}$, $i=1,2,\ldots,n-1$ the morphism
$V^{q,\otimes n}\to V^{q,\otimes n}$ which acts as $f$ on the $i$-th and
the $i+1$st factors. Note that $\sigma_{V^q,V^q}^{i,i+1}$ is always an
involution on $V^{q,\otimes n}$.

\begin{definition}
\label{def:symmetric power} For an object $V^q$ in $\OO_f$ and $n\ge
0$ define the {\it braided symmetric power} $S_\sigma^nV^q \subset
V^{q,\otimes n}$ and the {\it braided exterior power}
$\Lambda_\sigma^nV^q\subset V^{q,\otimes n}$  by:
$$S_\sigma^nV^q=\bigcap_{1\le i\le n-1}  (Ker~ \sigma_{i,i+1}-id)=\bigcap_{1\le i\le n-1} (Im~\sigma_{i,i+1}+id)\ ,$$
$$\Lambda_\sigma^nV^q=\bigcap_{1\le i\le n-1} (Ker~ \sigma_{i,i+1}+id)=\bigcap_{1\le i\le n-1} (Im~\sigma_{i,i+1}-id) ,$$
where we abbreviated $\sigma_{i,i+1}=\sigma_{V^q,V^q}^{i,i+1}$.
\end{definition}

\begin{remark} Clearly, $-\RR$ is also a braiding on $\OO_f$ and $-\sigma$ is the corresponding normalized braiding.
Therefore, $\Lambda_\sigma^nV^q=S_{-\sigma}^nV^q$ and
$S_\sigma^nV^q=\Lambda_{-\sigma}^nV^q$. That is, informally speaking,
the symmetric and exterior powers are mutually ''interchangeable''.

\end{remark}

\begin{remark} Another way to introduce the symmetric and exterior squares involves the well-known fact that the braiding $\RR_{V^q,V^q}$
is a semisimple operator $V^q\otimes V^q\to V^q\otimes V^q$, and all the
eigenvalues of $\RR_{V^q,V^q}$ are of the form $\pm q^r$, where $r\in
\ZZ$. Then {\it positive} eigenvectors of $\RR_{V^q,V^q}$ span
$S_\sigma^2V^q$ and {\it negative} eigenvectors  of $\RR_{V^q,V^q}$ span
$\Lambda_\sigma^2V^q$.
\end{remark}

Clearly, $S_{\sigma}^{0}V^q=\CC(q)\ ,\ S_{\sigma}^{1}V^q=V^q\ , \Lambda_{\sigma}^{0}V^q=\CC(q)\ ,\ \Lambda_{\sigma}^{1}V^q=V^q$, and
$$S_\sigma^2V^q=\{v\in V^q\otimes V^q\,|\, \sigma_{V^q,V^q}(v)=v\},~\Lambda_\sigma^2V^q=\{v\in V^q\otimes V^q\,|\, \sigma_{V^q,V^q}(v)=-v\} \ .$$





The following fact is obvious.

\begin{proposition}
\label{pr:injections} For each $n\ge 0$ the association $V^q\mapsto
S^n_\sigma V^q$ is a functor from $\OO_f$ to $\OO_f$ and the
association $V^q\mapsto \Lambda_\sigma^n V^q$ is a functor from $\OO_f$
to $\OO_f$. In particular, an embedding $U^q\hookrightarrow V^q$ in the
category $\OO_f$ induces injective morphisms
$$S_{\sigma}^n U^q \hookrightarrow S_{\sigma}^n V^q, ~ \Lambda_{\sigma}^n U^q\hookrightarrow \Lambda_{\sigma}^n V^q \ .$$

\end{proposition}

\begin{definition}
\label{def:symmetric algebra}
For any  $V^q\in Ob(\OO)$   define the {\it braided symmetric algebra} $S_\sigma(V^q)$ and the {\it braided  exterior algebra} $\Lambda_\sigma(V^q)$ by:
\begin{equation}
\label{eq:injections}
S_\sigma(V^q)=T(V^q)/\left< \Lambda_\sigma^{2}V^q\right>,~\Lambda_\sigma(V^q)=T(V^q)/ \left<S_\sigma^{2}V^q\right> \ ,
\end{equation}
where $T(V^q)$ is the tensor algebra of $V^q$ and $\left<I\right>$ stands for the two-sided  ideal in $T(Vv)$ generated by a subset $I\subset T(V^q)$.
\end{definition}

Note that the algebras $S_\sigma(V^q)$ and $\Lambda_\sigma(V^q)$ carry a natural $\ZZ_{\ge 0}$-grading:
$$S_\sigma(V^q)=\bigoplus_{n\ge 0} S_\sigma(V^q)_n, ~~~\Lambda_\sigma(V^q)=\bigoplus_{n\ge 0} \Lambda_\sigma(V^q)_n\ ,$$
since the respective ideals in $T(V^q)$ are homogeneous.

\smallskip

Denote by  $\OO_{gr,f}$ the sub-category of $U_q(\gg)-Mod$ whose objects are $\ZZ_{\ge 0}$-graded:
$$V^q=\bigoplus_{n\in \ZZ_{\ge 0}}  V_n^q\ ,$$
 
where each $V^q_n$ is an object of $\OO_f$; and morphisms are those homomorphisms of  $U_q(\gg)$-modules which preserve the $\ZZ_{\ge 0}$-grading.

Clearly,  $\OO_{gr,f}$ is a tensor category under the natural
extension of the tensor structure of $\OO_f$. Therefore, we can
speak of algebras and co-algebras in $\OO_{gr,f}$.

%
%
%

By the very definition,  $S_{\sigma}(V^q)$ and $\Lambda_{\sigma}(V^q)$ are algebras in $\OO_{gr,f}$.

\begin{proposition} The assignments $V^q\mapsto S_\sigma(V^q)$ and $V^q\mapsto \Lambda_\sigma(V^q)$ define functors from $\OO_f$
to the category of algebras in $\OO_{gr,f}$.
\end{proposition}

We conclude the section with two important features of braided symmetric exterior powers and algebras.

\begin{proposition}\cite[Prop.2.11 and Eq. 2.3]{BZ}  
Let $V^q$ be an object of $\OO_f$ and $V^*$ its dual in $\OO_f$. We have the following $U_q(\gg)$-module isomorphisms.
\begin{equation}
\label{eq:double duals}
(S_\sigma^n V^{q,*})^*\cong S_\sigma(V^q)_n, ~(\Lambda_\sigma^n V^{q,*})^*\cong \Lambda_\sigma(V^q)_n  \ .
\end{equation} 

\end{proposition}

\begin{proposition}\cite[Prop.2.13]{BZ}
\label{pr:n-th symm-power} For any  $V^q$ in $\OO_{f}$ each embedding
$V^q_{\lambda}\hookrightarrow V^q$ defines embeddings
$V^q_{n\lambda}\hookrightarrow S_{\sigma}^n V^q $ for all  $n\ge 2$. In
particular, the  algebra $S_{\sigma}(V^q)$ is infinite-dimensional.
\end{proposition}



\subsection{The Classical Limit of Braided Algebras}
\label{se:classical limit}

In this section we will discuss the specialization of the braided symmetric and exterior algebras at $q=1$, the classical limit. All of the results in this section are either well known or proved in \cite{BZ}. For a more detailed discussion of the classical limit we refer the reader to \cite[Section 3.2]{BZ}.  

We  will first introduce the notion of an {\it almost equivalence of categories}:

\begin{definition}
We say that a functor $F:\mathcal C\to \mathcal D$ is {\it almost equivalence} of $\mathcal C$ and $\mathcal D$ if: 

\noindent (a) for any objects $c,c'$ of $\mathcal C$ an isomorphism $F(c)\cong F(c')$ in $\mathcal D$ implies that $c\cong c'$ in  ${\mathcal C}$;

\noindent (b) for any object in $d$ there exists an object $c$ in ${\mathcal C}$ such that $F(c)\cong d$ in ${\mathcal D}$. 

\end{definition} 

Denote by $\overline \OO_f$ the full (tensor) sub-category category of $U(\gg)-Mod$, whose objects $\overline V$ are finite-dimensional $U(\gg)$-modules having a weight decomposition $\overline V=\oplus_{\mu\in P} \overline V(\mu)$.  The following fact will be the first result of this section.

\begin{proposition}\cite[Cor 3.22]{BZ}
\label{pr:almost equivalent O bar O}
The categories  $\OO_f$ and  $\overline \OO_f$ are almost equivalent. Under this almost equivalence a simple $U_q(\gg)$-module $V_\lambda$ is mapped to the simple $U(\gg)$-module $\overline V_\lambda$.
\end{proposition}
Let $V\cong \bigoplus_{i=1}^{n} V_{\lambda_{i}}  \in \OO_{f}$.  We call $\overline V \cong \bigoplus_{i=1}^{n} \overline V_{\lambda_{i}}  \in \overline \OO_{f}$ the {\it classical limit} of $V$ under the above almost equivalence.

\begin{proof}
First, we have to introduce the notion of  $({\bf k},{\bf A})$-algebras and investigate their properties.
Let ${\bf k}$ be a field and ${\bf A}$ be a local subring of ${\bf k}$. Denote by $\mm$ the only maximal ideal in ${\bf A}$ and by $\tilde {\bf k}$ the  residue field of ${\bf A}$, i.e., $\tilde {\bf k}:={\bf A}/\mm$.

We say that an ${\bf A}$-submodule $L$ of a ${\bf k}$-vector space $V$ is an {\it ${\bf A}$-lattice} of $V$ if  $L$ is a free ${\bf A}$-module and ${\bf k}\otimes_{\bf A} L=V$, i.e., $L$ spans $V$ as a ${\bf k}$-vector space. Note that for any ${\bf k}$-vector space $V$ and any ${\bf k}$-linear basis ${\bf B}$ of $V$ the ${\bf A}$-span $L={\bf A}\cdot {\bf B}$ is an ${\bf A}$-lattice in $V$. Conversely, if $L$ is an ${\bf A}$-lattice in $V$, then any ${\bf A}$-linear basis ${\bf B}$ of $L$ is also a ${\bf k}$-linear basis of $V$.

Denote by $({\bf k},{\bf A})-Mod$ the  category whose objects are pairs $\VV=(V,L)$ of a ${\bf k}$-vector space $V$ and an ${\bf A}$-lattice $L\subset V$  of $V$; an arrow $(V,L)\to (V',L')$ is any ${\bf k}$-linear map $f:V\to V'$ such that $f(L)\subset L'$.

Clearly, $({\bf k},{\bf A})-Mod$ is an Abelian category. Moreover, $({\bf k},{\bf A})-Mod$ is ${\bf A}$-linear because each  $Hom(\UU,\VV)$ in $({\bf k},{\bf A})-Mod$ is an ${\bf A}$-module.

It can be easily verified that $({\bf k},{\bf A})-Mod$ is a symmetric tensor category (\cite[Lemma 3.14]{BZ}.
We have the following fact.
\begin{lemma}\cite[Lemma 3.12]{BZ}
\label{le:trafo of bases} The forgetful functor  $({\bf k},{\bf A})-Mod \to  {\bf k}-Mod$ given by $(V,L)\mapsto V$ is an almost equivalence of symmetric tensor categories. 
\end{lemma}
Define a functor $\FF:({\bf k},{\bf A})-Mod \to \tilde {\bf k}-Mod$ by:
$$\FF(V,L)=L/\mm L\ $$
for any object 
$(V,L)$ of $({\bf k},{\bf A})-Mod$ and for any morphism $f:(V,L)\to (V',L')$ we define 
$\FF(f):L/\mm L\to L'/\mm L'$ to be a natural  $\tilde {\bf k}$-linear map.

\begin{lemma}\cite[Lemma 3.14]{BZ}
\label{le:specialization functor} $\FF:({\bf k},{\bf A})-Mod \to \tilde {\bf k}-Mod$ is a tensor functor and  almost equivalence.
\end{lemma}
Let $U$ be a ${\bf k}$-Hopf algebra and let $U_{\bf A}$ be a Hopf ${\bf A}$-subalgebra of $U$. This means that $\Delta(U_{\bf A})\subset U_{\bf A}\otimes_{\bf A} U_{\bf A}$ (where $U_{\bf A}\otimes_{\bf A} U_{\bf A}$ is naturally an ${\bf A}$-sub-algebra of $U\otimes_{\bf k} U$), $\varepsilon(U_{\bf A})\subset {\bf A}$, and $S(U_{\bf A})\subset U_{\bf A}$.  
We will refer to the above pair  $\UU=(U,U_{\bf A})$ as to $({\bf k},{\bf A})$-{\it Hopf algebra} (please note that $U_{\bf A}$ is not necessarily a free ${\bf A}$-module, that is, $\UU$ is not necessarily a $({\bf k},{\bf A})$-module).

Given $({\bf k},{\bf A})$-Hopf algebra $\UU=(U,U_{\bf A})$, we say that an  object $\VV=(V,L)$ of $({\bf k}, {\bf A})-Mod$  is a $\UU$-module if $V$ is a $U$-module and $L$ is an $U_{\bf A}$-module.

Denote by $\UU-Mod$ the category which objects are $\UU$-modules  and arrows are those morphisms of $({\bf k}, {\bf A})$-modules which commute with the $\UU$-action. 

 Clearly, for $({\bf k},{\bf A})$-Hopf algebra $\UU=(U,U_{\bf A})$ the category $\UU-Mod$ is a tensor (but not necessarily symmetric) category. 
 
 For each $({\bf k},{\bf A})$-Hopf algebra $\UU=(U,U_{\bf A})$ we define $\overline \UU:=U_{\bf A}/\mm U_{\bf A}$. Clearly, $\overline \UU$ is a Hopf algebra over $\tilde {\bf k}={\bf A}/\mm$.

The following fact is obvious.

\begin{lemma}\cite[Lemma 3.15]{BZ}
\label{le:Hopf modules}
 In the notation of Lemma \ref{le:specialization functor}, for any   $({\bf k},{\bf A})$-Hopf algebra $\UU$ the functor $\FF$   naturally extends to a tensor functor
\begin{equation}
\label{eq:dequantization functor}
\UU-Mod\to \overline \UU-Mod \ .
\end{equation}

\end{lemma}

Now let ${\bf k}=\CC(q)$ and ${\bf A}$ be the ring of all those rational functions in $q$ which are defined at $q=1$. Clearly, ${\bf A}$ is a local PID with maximal ideal $\mm=(q-1){\bf A}$ (and, moreover, each ideal in ${\bf A}$ is of the form $\mm^n=(q-1)^n{\bf A}$). Therefore,  $\tilde {\bf k}:={\bf A}/\mm=\CC$.

Recall from Section \ref{se:q-group} the definition of  the quantized universal enveloping algebra $U_{q}(\gg)$.. Denote $h_\lambda=\frac{K_\lambda-1}{q-1}$ and let $U_{{\bf A}}(\gg)$ be the ${\bf A}$-algebra generated by all  $h_\lambda$, $\lambda\in P$  and all $E_ i, F_i$.

Denote by $\UU_{q}(\gg)$ the pair $(U_{q}(\gg), U_{{\bf A}}(\gg))$.

\begin{lemma} 
\label{le:Hopf pair}
(a) The pair $\UU_{q}(\gg)=(U_{q}(\gg), U_{{\bf A}}(\gg))$ is a $({\bf k},{\bf A})$-Hopf algebra (\cite[Lemma 3.16]{BZ}).

\noindent(b) We have $\overline{\UU_q(\gg)}=U(\gg)$ (\cite[Lemma 3.17]{BZ}).

\end{lemma}

Let  $V_\lambda\in Ob(\OO_f)$ be an irreducible   $U_{q}(\gg)$-module with highest weight $\lambda\in P^+$ and let $v_{\lambda}\in V_{\lambda}$ be a highest weight vector.  Define $L_{v_{\lambda}}= U_{{\bf A}} (\gg) \cdot v_{\lambda}$.

\begin{lemma} 
\cite[Lemma 3.18]{BZ}
$(V_{\lambda},L_{v_{\lambda}})\in \OO_{f}(\UU_{q}(\gg))$.

\end{lemma}

The following fact is obvious and, apparently, well-known.

 \begin{lemma}\cite[Lemma 3.19]{BZ}
 \label{le: equiv of OOs}

(a) Each object $(V_{\lambda},L_{v_{\lambda}})$ is irreducible in $\OO_{f}(\UU_{q}(\gg))$; and each irreducible object of $\OO_{f}(\UU_{q}(\gg))$ is isomorphic to one of $(V_{\lambda},L_{v_{\lambda}})$.

\noindent (b)  The category $\OO_{f}(\UU_{q}(\gg))$  is semisimple.

\noindent (c) The forgetful functor  $(V ,L)\mapsto V$ is an almost equivalence of tensor categories $\OO_{f}(\UU_{q}(\gg))\to \OO_{f}$.
 \end{lemma}
 
 We also have the following fact.
 
 \begin{lemma} \cite[Lemma 3.21]{BZ}
\label{le:dequantization equivalence}
(a) The restriction of the functor $\UU_q(\gg)-Mod\to U(\gg)-Mod$ defined by (\ref{eq:dequantization functor}) to the sub-category  $\OO_{f}(\UU_{q}(\gg))$ is a tensor functor 
\begin{equation}
\label{eq:dequantization q=1}
\OO_{f}(\UU_{q}(\gg))\to \overline \OO_{f} \ .
\end{equation}

\noindent (b) The functor (\ref{eq:dequantization q=1}) is an almost equivalence of  categories. 

\end{lemma}

Combining Lemma  \ref{le: equiv of OOs} and Lemma \ref{le:dequantization equivalence} we obtain Proposition \ref{pr:almost equivalent O bar O}.
\end{proof}

The  following result relates the classical limit of braided symmetric algebras and Poisson algebras..

\begin{theorem}\cite[Theorem 2.29]{BZ}
\label{th:flat poisson closure} Let $V$ be an object of $\OO_f$ and let a  $\overline V$ in $\overline \OO_f$ be  be the classical limit of $V$. Then:

The classical limit $\overline{S_\sigma(V)}$ of the braided symmetric algebra $S_\sigma(V)$ is a quotient of the symmetric algebra $S(\overline V)$. In particular, $\dim_{\CC(q)} S_\sigma(V)_n=\dim_{\CC}(\overline{S_\sigma(V)})\le \dim_{\CC} \overline{S(\overline V)}_n$.

Moreover, $\overline{S_\sigma(V)}$ admits a Poisson structure defined by  $\{u,v\}=r^-(u\wedge v)$, where $r^-$ is an anti-symmetrized $r$-matrix. 
 
\end{theorem}

 
 

  
  
  

\subsection{Flat  Modules over Reductive Lie Algebras}
\label{subsect:flat modules}
In \cite{BZ} we introduce the notion of flatness of a $U_q(\gg)$-module. In this setion we will recall the definition and basic properties of flat modules and then proceed to classify all flat modules over $U_q(\gg)$, where $\gg$ is any semisimple Lie algebra.
 
 We view $S_{\sigma}(V^q)$ and $\Lambda_{\sigma}(V^q)$ as deformations of the quadratic algebras  $S(V)$ and $\Lambda(V)$ respectively, where $V$ denotes the classical limit of $V^q$.
In \cite{BZ} we show that 
$$\dim  ~S_{\sigma}^n V^q=\dim  ~S_{\sigma}(V^q)_n \le \binom{\dim V^q+n-1} n $$
for all $n$.   

Therefore, it is natural  to make the following definition.

 \begin{definition}
\label{definition:flatness}
A finite dimensional $U_q(\gg)$-module is flat, if and only if
$$\dim ~S^n_\sigma V^q=\binom{\dim V^q +n-1}{n}  $$
for all $n\ge 0$; i.e., the braided symmetric power $S^n_\sigma V^q$ is isomorphic (as a vector space) to the ordinary symmetric  power $S^n V^q$.

\end{definition}

The following theorem is our main result.

\begin{maintheorem}
\label{the:class. flat}
Let $\gg$ be a semisimple Lie algebra and $U_q(\gg)$ its quantized enveloping algebra. A simple $U_q(\gg)$-module $V$ is flat if and only if its classical limit $\overline V$ is Poisson as a $U(\gg)$-module.
\end{maintheorem} 
\begin{proof}

The "only if" assertion follows immediately from the following result.

\begin{proposition} \cite[Theorem 2.29]{BZ}
\label{th:flat->Poisson}
If $V^q$ is an object of $\OO_f$ and $V^q$ is flat, then $S(V)$ is Poisson.  
\end{proposition}

\begin{proof}
Theorem \ref {th:flat poisson closure} asserts  that the classical limit of $S_\sigma(V^q)$ is a Poisson algebra. If $V^q$ is flat  then $(S (V), r^-)$ is a Poisson algebra and $V$ Poisson by Proposition \ref{pr:cv=0=poisson} .
 \end{proof}
It therefore, remains to show that if a simple $\gg$-module $V$ is Poisson , then $V^q$ is a flat $U_q(\gg)$-module. Following the strategy of Section \ref{se: poisson over simple} we will first  consider the case when $\gg$ is a simple Lie algebra and $V$ a simple $\gg$-module.
 First assume that $\gg$ is a simple Lie algebra and $V$ a simple $\gg$-module.  Recall from Corollary \ref{co: poisson=rigidity} that a simple $\gg$-module $V_\lambda$ is Poisson, if and only if $V_\lambda$ is rigid, hence  the assertion of Theorem \ref{the:class. flat} follows immediately from the following result.

\begin{proposition}
Let $\gg$ be a simple Lie algebra and $V_\lambda$ a simple $\gg$-module. The $U_q(\gg)$-module $V_\lambda^q$ is flat if $V_\lambda$ is rigid.
\end{proposition}
 \begin{proof}
Indeed \cite[Theorem 2.36]{BZ} asserts that $V^q$ is flat, if $V$ is rigid:  We have $\dim ~S^n_\sigma V^q=\binom{\dim V^q+1}{n}$   and $\overline{S^n_\sigma V^q}\cong S^n V$ for
$n=0,1,2$. Employing a well known result by  Drinfeld (\cite[Theorem 1]{DR2}) about quadratic algebras it is  shown  in \cite[Proposition 2.33]{BZ} that  $V^q$ is flat if and only if $S^3_\sigma V^q=\binom{\dim(V^q)+2}{3}$. Since dequantization is an almost equivalence of  the tensor categories $\OO_f$ and $\OO_f(\gg)$(Lemma \ref{le: equiv of OOs} (c)), we obtain that in the notation around Lemma \ref{le:gen symmetric cube} the multiplicity of $V^q_\nu$ in $S^2_\sigma V^q_\lambda$(resp. $\Lambda^2_\sigma V^q_\lambda$) is $c^+_{\lambda,\nu}$ (resp.   
$c^-_{\lambda,\nu}$). 

Denote by $ c^{3,\sigma}_{\lambda,\mu}$ (resp. $c^{3}_{\lambda,\mu}$)  the multiplicity of $V^q_\mu$ in $S^3_\sigma V_\lambda^q$ (resp. of $V_\mu$ in $S^3 V_\lambda$).
We derive, arguing analogously to the proof of Lemma   \ref{le:gen symmetric cube} that 
$d_\lambda^\mu\le c^{3,\sigma}_{\lambda,\mu}$ for all $\mu\in P^+(\gg)$. Since $c^{3,\sigma}_{\lambda,\mu}\le c^{3}_{\lambda,\mu}$ for all $\mu\in P^+(\gg)$, we obtain that if  $V_\lambda$ is rigid and $d_\lambda^\mu= c^{3}_{\lambda,\mu}$, then $c^{3,\sigma}_{\lambda,\mu}=c^{3}_{\lambda,\mu}$ for all $\mu\in P^+$ and hence $V^q_\lambda$ is flat. The proposition is proved.
\end{proof}

 Now consider the case when $\gg$  is semisimple and $V$ a simple $\gg$-module. Theorem \ref{th:class:poissonsemisimple}  asserts that if $supp(\gg)$ is not simple and $V$ is Poisson, then $supp(\gg)\cong sl_n\times sl_m$ for some $m,n\ge 1$ and $V$ isomorphic to the natural module $V_{\omega_1,\omega_1}$.
We show in \cite[Proposition 2.38]{BZ} that the natural $U_q(sl_m\times sl_n)$-module is flat, its braided symmetric algebra isomorphic to the algebra of quantum $m\times n$-matrices.  Theorem \ref{the:class. flat} is proved.  
 \end{proof}

 \begin{remark}
 A straightforward argument shows that if $\gg$ is a reductive Lie algebra, then a $U_q(\gg')$-module $V^q$ is flat if $V^q|_{U_q(\gg')}$ is a flat $U_q(\gg')$-module, where $\gg'\subset \gg$ is the maximal semisimple subalgebra of $\gg$. 
 \end{remark}
 
\section{Deformations of Symmetric Algebras of Poisson Modules} 
\label{se:constru.of BSA}

In this  section we will explicitly construct the braided symmetric algebras of flat modules, employing the relationship between geometrically decomposable modules and Abelian nil-radicals.  
  \subsection{Quantum Radicals as Symmetric Algebras}
  \label{se: quantum schubert}
  
 Let $U_q(\gg)$  be the quantized enveloping algebra corresponding to a Lie algebra $\gg$  introduced in Section \ref{se:q-group}.  Denote, as above,  by $W$ the Weyl group of $\gg$ generated by the simple reflections  $s_i$ for $i\in[1,r]$. Corresponding to each $i\in [1,r]$ there exist maps $T_{i}:U_q(\gg)\to U_q(\gg)$ defined on the generators of $U_q(\gg)$ in the following way:

\begin{equation}
\label{eq: T_i on Ui}
T_{i}(E_{i})=-F_{i} K_{\alpha_{i}}\ ,T_{i}(F_{i})=-K_{\alpha_{i}}^{-1}E_{i}\ , T_{i}(E_{j})=\sum_{k=0}^{-a_{ij}} (-1)^{k-a_{ij}}q_{i}^{-k}  E_{i}^{-a_{ij}-k}E_{j}E_{i}^{k}\ ,
\end{equation}
$$T_{i}(F_{j})=\sum_{k=0}^{-a_{ij}} (-1)^{k-a_{ij}}q_{i}^{k} F_{i}^{k}F_{j}F_{i}^{a_{ij}-k}\ , T_{i}(K_{\lambda})=K_{\sigma_{i}(\lambda}\ .$$ 
 For every element $w\in W$ with presentation $w=s_{i_{1}}\ldots s_{i_{k}}$ we define $T_{w}$ as $T_{w}=T_{s_{i_{1}}}\cdots T_{s_{i_{k}}}$. 
We will need the following well known fact.
 \begin{lemma}\cite[8.18]{Jan} 
 If $w\in W$, then $T_w$ is independent of the choice of reduced expression; i.e. if $w=s_{i_1}\ldots s_{i_k}$ and  $w=s_{j_1}\ldots s_{j_k}$ are reduced expressions of $w\in W$, then 
 $$T_{s_{i_1}}\ldots T_{s_{i_k}}=T_{s_{j_1}}\ldots T_{s_{j_k}}\ .$$ 
 \end{lemma}

Recall from Section \ref{se:prlime} that $U^+$ denotes the  subalgebra of $U_q(\gg)$ generated by the $E_{i}$ for $i\in [1,r]$,  $U^{-}$ the  subalgebra of $U_q(\gg)$ generated by the $F_{i}$ for $i\in [1,r]$  and  $U_q(\bb_-)$ - the subalgebra of $U_q(\gg)$ generated by all $K_\lambda$ and all $F_i$.  

   Recall (see e.g.\cite[ch. 8]{Jan}) that we can associate to each reduced expression of the longest element $w_{0}\in W$  a PBW-basis  of $U_q(\gg)$ in the following way:  
Let $w_{0}=\sigma_{i_{1}} \ldots \sigma_{i_{k}}$ be a presentation of the longest word in $W$. It is well known that the set of positive roots  $R^{+}$ of the Lie algebra $\gg$ can be ordered in the following way:
$$\alpha_{(1)}=\alpha_{i_{1}}<\alpha_{(2)}=s_{i_{1}}\circ \alpha_{2}< ,\ldots<\ldots <\alpha_{(k)} =s_{i_{1}}\ldots s_{i_{k-1}} \alpha_{i_{k}}\ ,$$
where $\alpha_{j}$ denotes the $j$-th simple root.

We  define for each presentation of $w_{0}$  a set of  positive roots spanning $U^{+}$    following \cite[ch.8]{Jan}:
$$E_{\alpha_{(1)}}= E_{i_{1}}\ ,~E_{\alpha_{(2)}}=T_{i_{1}}(E_{i_{2}})\ ,~\ldots ,E_{\alpha_{(k)}}= T_{i_{1}}\cdots T_{i_{k-1}}(E_{i_{k}}) \ .$$

Similarly, we define a set of negative roots spanning $U^{-}$ by:
$$F_{\alpha_{(1)}}= F_{i_{1}}\ ,~F_{\alpha_{(2)}}T_{i_{1}}(F_{i_{2}})\ ,~\ldots, F_{\alpha_{(k)}}= T_{i_{1}}\cdots T_{i_{k-1}}(F_{i_{k}}) \ .$$

The following is the key definition for this section.

\begin{definition}
For every  element $w\in W$ in the Weyl group we define the  quantum Schubert cell $U(w)$ as 

$$U(w)=T_{w^{-1}}(U_q(\bb^-))\cap U^+\ .$$
  \end{definition}  

We also have the following alternative description of quantum Schubert cells.

\begin{lemma}
\label{le:upperdefofqsc}
Let $w\in W$ and $w_0$ be the longest element in $W$. Denote $w'=w_0w$. We have 
$$U(w)=T_{(w')^{-1}}(U_q(\bb^+))\cap U^+\ .$$ 
\end{lemma}  

\begin{proof}
  We have $T_{(w')^{-1}}=T_{w_0w^{-1}}=T_{w_0}T_{w^{-1}}$, $T_{w_0w^{-1}}^{-1}= T_{w}T_{w_0}^{-1}$, since $w_0^{-1}=w_0$. Note that $T_{w_0}(U_q(\bb^+)=(U_q(\bb^-)$. Therefore, 

$$U(w)= T_{w^{-1}}(U_q(\bb^-))\cap U^+=T_{w^{-1}w_0}(U_q(\bb^+))\cap U^+\ .$$
\end{proof}

We will now consider quantum Schubert cells $U(w_\Delta)$ where  $w_{\Delta}\in W$  corresponds to subsets $\Delta\subset [1,r]$ in the following way. Let $W_{\Delta}$ be the subgroup of Delta generated by the simple reflections $s_i$ for $i\in \Delta$ and denote by $w_{0,\Delta}$ the longest element of $W_\Delta$. The element $w_\Delta=w_{0,\Delta}w_0$ is commonly referred to as a {\it parabolic element} of $W$. If $\mathfrak{p}_{\Delta}$ is the standard parabolic subalgebra of $\gg$ associated with $\Delta$, then $w_{0,\Delta}$ is the longest element of its Levi subalgebra $\mathfrak{l}_{\Delta}$.
 Denote the nil-radical by  $rad_{\Delta}$.  Recall that $\mathfrak{p}_{\Delta}$ splits as a semi-direct product $\mathfrak{p}_{\Delta}\cong \mathfrak{l}_{\Delta} \ltimes rad_{\Delta}$ (see e.g. \cite{Hu}). Additionally recall that any Hopf algebra $H$  algebra acts on itself via the adjoint action:
 
 \begin{equation}
 \label{eq:adjoint action}
 ad(a).b=a_{(1)}b S(a_{(2)}) 
 \end{equation}
  The following theorem is the first main result of this section.  

\begin{maintheorem}
\label{th: Abelian ->flat quadratic}
(a) Let $\gg'$ be a reductive Lie algebra, $\mathfrak{p}_{\Delta}$ a parabolic subalgebra with Levi $\mathfrak{l}_{\Delta}$ and radical $rad_{\Delta}$. If $rad_{\Delta}$ is an Abelian Lie algebra, then $U(w_\Delta)$ is a flat quadratic $q$-deformation of the  symmetric algebra $S(rad_{\Delta})$. 

\noindent(b) The quantum Schubert cell $U(w_\Delta)$ is a  $\ZZ_{\ge0}$-graded $U_q(\mathfrak{l}_\Delta)$ module algebra and  $U(w_\Delta)$ is the braided symmetric algebra of the $U_q(\mathfrak{l}_{\Delta})$-module $U(w_\Delta)_1$.  

\noindent (c) Moreover, let $\gg_{\Delta}$ be the maximal semsimple submodule of $\mathfrak{l}_\Delta$. Then,  $U(w_\Delta)$ is a  $\ZZ_{\ge0}$-graded $U_q(\gg_\Delta)$ module algebra and  $U(w_\Delta)$ is the braided symmetric algebra of the $U_q(\gg_{\Delta})$-module $U(w_\Delta)_1$.  

\end{maintheorem}

 \begin{proof}

In order to prove Theorem \ref{th: Abelian ->flat quadratic} (a)  we have to show that the classical limit $q\to 1$ of $U(w_\Delta)$ is $S(rad_{\Delta})$.
  We call a root $\alpha\in R(\gg)$ {\it radical}, if $\alpha\notin R(\gg')\cap span_{\ZZ}(\Delta)$.  Recall that $rad_{\Delta}$ is spanned by $E_{\alpha}$ where $\alpha$ is radical.  We obtain the following well known characterization of Abelian radicals.

\begin{lemma} 
\label{le:Abelian rads}
Let $\gg'$ be a redcutive  Lie algebra, $\mathfrak{p}_{\Delta}$ a parabolic subalgebra with Levi $\mathfrak{l}_{\Delta}$ and radical $rad_{\Delta}$. The radical $rad_{\Delta}$ is Abelian, if  and only if all radical roots are of the form $\alpha=\alpha_{i} +\sum_{j\ne i}c_{j}\alpha_{j}$.
\end{lemma}

 Theorem  \ref{th:PBW-Uw} yields that $U(w_\Delta)$ is generated as an algebra by the $E_{\alpha}$ for which $\alpha\in R^+(\gg)$ is a radical root.
 
 We need the following well-known fact.
 
\begin{lemma} 
\label{le:q-commutators}
 \cite[Lemma 2.2]{DKP} Let  $R^+_w=w(R^-)\cap R^+$. Then,
\begin{equation}
\label{eq:q-commutators}
ad(E_{\alpha})(E_{\beta})=[E_{\alpha},E_{\beta}]_{q}=E_{\alpha}E_{\beta}-q^{(\alpha,\beta)}E_{\beta}E_{\alpha} \in \text{span}(E_{\gamma_{1}}\ldots E_{\gamma_{k}})\ ,
\end{equation}
where $\alpha<\gamma_{1}\le\ldots\le \gamma_{k}<\beta$ and $\gamma_{1}+\ldots \gamma_{k}=\alpha+\beta$.
\end{lemma}
 
We now obtain from  Lemma  \ref{le:Abelian rads} and Lemma \ref{le:q-commutators} that if $rad_{\Delta}$ is Abelian, then  $U(w_\Delta)$ is a quadratic algebra and its classical limit is $S(rad_{\Delta})$. Theorem \ref{th: Abelian ->flat quadratic} (a) is proved.

 Let us now prove part (b). Note first that  $U_q(\mathfrak{l}_{\Delta})$ acts adjointly on $U(w_\Delta)$ by Theorem \ref{th:PBW-Uw}.
We now obtain that if $rad_{\Delta}$ is Abelian, then Lemma \ref{le:Abelian rads}  and Lemma \ref{le:q-commutators} imply that  $ad(U_q(\mathfrak{l}_{\Delta})(U(w_\Delta))_{i})\subset (U(w_\Delta))_{i}$, and that, hence, $U(w_\Delta)$ is a graded $U_q(\mathfrak{l}_{\Delta})$-module algebra.
 Denote by $\pi_{\Delta}$ the canonical $U_q(\mathfrak{l}_{\Delta})$-module homomorphism  $\pi_{\Delta}: T((U(w_\Delta))_{1})\to  U(w_\Delta)$.  We obtain from Theorem \ref{th: Abelian ->flat quadratic} that the classical limit of the kernel $\ker(\pi)$ is equal to the $U(\mathfrak{l}_{\Delta})$-ideal generated by $\Lambda^{2}rad_{\Delta}$. Howe proves in \cite[ch. 4.6]{Ho} that $rad_{\Delta}$ is  weight-multiplicity-free and simple as a  $\mathfrak{l}_{\Delta}$-module, as well as a $\mathfrak{\gg}_{\Delta}$-module, that means all weight-spaces are  one-dimensional. Recall the following  well-known fact.

\begin{lemma}
\label{le:mult-free tensor}
Let $\gg$ be a reductive Lie algebra. Then, $\dim(Hom_{\gg}(V_{\lambda}\otimes V_{\lambda}, V_{\mu}))\le\dim V_{\lambda}(\mu-\lambda)$.  
\end{lemma}

The  lemma implies that   $rad_{\Delta}\otimes rad_{\Delta}$ is multiplicity-free as a $\mathfrak{l}_{\Delta}$-module. Employing Lemma \ref{le:dequantization equivalence}, we obtain that   $(U(w_\Delta))_{1}\otimes  (U(w_\Delta))_{1}$ contains a unique $U_q(\mathfrak{l}_{\Delta})$-submodule $Ext^2_{q}(rad_{\Delta})$ such that its classical limit is isomorphic to $\Lambda^{2} rad_{\Delta}$. This  implies that $\Lambda^{2}_{\sigma} (U(w_\Delta))_{1}=Ext^2_{q}(rad_{\Delta})=\ker(\pi)\cap (U(w_\Delta))_{1}\otimes (U(w_\Delta))_{1}$. Therefore,  $U_q(rad_{\Delta})= S_{\sigma} (U(w_\Delta)_{1})$. Theorem  \ref{th: Abelian ->flat quadratic} (b) is proved.

Part (c) can be proved analogously to part (b). Theorem  \ref{th: Abelian ->flat quadratic} is proved.
 \end{proof} 
  
 Call a $U_q(\gg)$-module {\it geometrically decomposable} if its classical limit is geometrically decomposable as a $U(\gg)$-module.
Theorem  \ref{th: Abelian ->flat quadratic} (c) has the following consequence. 
 
 \begin{corollary}
 Let $\gg$ be a semisimple Lie algebra and let $V^q$ be a simple geometrically decomposable $U_q(\gg)$-module. There exists a simple Lie algebra $\gg'$ and a parabolic element $w_\Delta\in W(\gg')$ such that  $U_q(\gg)\cong U_q(\gg_\Delta)$  and   the braided symmetric algebra $S_\sigma V^q\cong U(w_\Delta)$ as $U_q(\gg_\Delta)$-modules.
 \end{corollary}
 
 \begin{proof}
 The corollary follows immediately from Theorem  \ref{th: Abelian ->flat quadratic}  and the description of geometrically decomposable modules as Abelian radicals in Section \ref{se: poisson over simple}.
 \end{proof}   
  
Many of the  braided symmetric algebras obtained by the construction of Theorem  \ref{th: Abelian ->flat quadratic} are well known examples of quantized coordinate rings of classical varieties. Our theory presents a unifying construction of these important examples. We have the following list according to \cite[ch. 5]{GY}:

\begin{itemize}
\item If $\gg=sl_k$ and $\Delta=\{1,\ldots, n\}/\{i\}$, then $U(w_\Delta)=\CC_q[Mat_{i\times (n-i)}]$, the algebra of quantum $i\times (n-i)$-matrices.

\item If $\gg=so(2n+1)$ and $\Delta=\{2,\ldots,n\}$, then $U(w_\Delta)$ is the algebra of the odd-dimensional Euclidean space  $\OO_{\frac{q}{2}}^{2N-1}(\CC)$ introduced in \cite{RTF} (see also \cite{M}).

\item If $\gg=sp(2n)$ and $\Delta=\{2,\ldots n\}$, then $U(w_\Delta)$ is the  algebra of quantum symmetric matrices introduced in \cite[Theorem 4.3 and Proposition 4.4]{Nou} and by \cite{Kam}.

\item If $\gg=so(2n)$ and $\Delta=\{2,\ldots , n\}$, then $U(w_\Delta)$ is the algebra of the even-dimensional Euclidean space  $\OO_{\frac{q}{2}}^{2N-2}(\CC)$ introduced in \cite{RTF} (see also \cite{M}). 

\item  If $\gg=so(2n)$ and $\Delta=\{1,\ldots , n-1\}$ or $\Delta=\{1,\ldots, n-2,n\}$, then $U(w_\Delta)$ is the  algebra of quantum antisymmetric matrices introduced in \cite[Section 1]{Str}.

\item If $\gg=E_6$ and $\Delta=\{2,\ldots, 6\}$, resp. $\Delta=\{1,\ldots,5\}$ or $\gg=E_7$ and $\Delta=\{1,\ldots 6\}$, then we obtain quantum algebras $U(w_\Delta)$, which apparently have not been  studied previously.
\end{itemize}

We will now extend the result of   Theorem  \ref{th: Abelian ->flat quadratic} to some subalgebras, when $rad_\Delta$ is of Heisenberg type; i.e., the derived subalgebra $[rad_\Delta, rad_\Delta]\subset rad_\Delta$ is one-dimensional.  
Recall that an algebra $U$ is called filtered, if $U=\bigcup_{i=0}^\infty U_i$ with $U_i\subset U_{i+1}$  and $U_i\cdot U_{j}\subset U_{i+j}$. 
The associated graded algebra $gr(U)$ of $U$ is defined as $gr(U)=\bigoplus_{i=0}^{\infty} U_{i}/U_{i-1}$, where we set $U_{-1}=\{0\}$. 
The following result is the second main result of this section.

\begin{theorem}
\label{th:Heisenberg}
Let $\Delta\subset[1,r]$ and let $\mathfrak{p}_{\Delta}$  be the corresponding parabolic subalgebra, and $rad_{\Delta}$ its nil-radical. If  $U(w_\Delta)$ is a filtered $U_q(\mathfrak{l}_{\Delta})$-module algebra, 

\noindent(a) then $gr(U(w_\Delta))$ is a $U_q(\mathfrak{l}_\Delta)$-module algebra and  a flat $q$-deformation of $S(rad_{\Delta})$,

\noindent(b) and $gr(U(w_\Delta))$ is a $U_q(\gg_\Delta)$-module algebra, where $\gg_\Delta$ is the maximal semisimple subalgebra of $\mathfrak{l}_\Delta$.  \end{theorem}  
  
 \begin{proof} 
 
The following fact is well known.

\begin{lemma}
\label{le:Uto gr}
\noindent(a) If $U$ is a filtered $k$-algebra, there are isomorphisms of vector spaces  $\tilde\phi_{n}: U_{n}\to \bigoplus_{i=0}^{n}\tilde U_{i}$ where $\tilde U_{i}=U_{i}/U_{i-1}$.  

\noindent(b)   The isomorphisms $\tilde\phi_{n}$ induce an isomoprhism of  $k$-algebras  $\phi:U\to gr(U)$.

\end{lemma}

 Another well known and important fact is the following.
 \begin{lemma}
 \label{le: grmodal}
 Let $A$ be a Hopf algebra, and $U$ be a filtered $A$-module algebra. Then $\phi: U\to gr(U)$  is an isomorphism of $A$-modules.
 \end{lemma}

Recall that an algebra $A$ is called quadratic-linear, if $A$ is the quotient of a free algebra $\CC\langle x_{1},\ldots x_{n}\rangle$ by an ideal generated by elements of $\bigoplus_{i=0}^{2} (\CC\langle x_{1},\ldots x_{n}\rangle)_{i}$, where  $(\CC\langle x_{1},\ldots x_{n}\rangle)_{i}$ denotes the $i$-th graded component. The following proposition is a key step to proving Theorem  \ref{th:Heisenberg}. 

\begin{proposition}
\label{pr:gr Uwmodal}
Let $ \Delta\in [1,r]$. If $U(w_\Delta)$ is quadratic linear, i.e. $[\overline E_{\alpha},\overline E_{\beta}]_{q}\in (U(w_\Delta))_{2}$  for all radical roots $\alpha,\beta$  and $U(w_\Delta)$ is a filtered $U_q(\mathfrak{l}_{\Delta})$-module algebra, then $gr(U(w_\Delta))$ is a $U_q(\mathfrak{l}_{\Delta})$-module algebra. 
\end{proposition}

\begin{proof}
If $U(w_\Delta)$ is quadratic linear, then $U(w_\Delta)$ is a filtered Hopf algebra. Hence,  $gr(U(w_\Delta))$ is  a $U_q(\mathfrak{l}_{\Delta})$-module algebra by Lemma \ref{le: grmodal}. 
The proposition is proved.
\end{proof}

  Note first that  if $U(w_\Delta)$ is filtered, then $U(w_\Delta)$ must be quadratic linear.  Hence, $gr(U(w_\Delta))$ is quadratic, and  Theorem \ref{th:PBW-Uw} and Lemma \ref{le:q-commutators} yield that its classical limit is $S(rad_{\Delta})$. Theorem \ref{th:Heisenberg} (a) is proved.
  
Part(b) can be proved analogously.  Theorem \ref{th:Heisenberg} is proved.
    \end{proof}

We now obtain the construction of the braided symmetric algebra of the natural module of $U_q(sp(2n)$. 

\begin{corollary}
\label{co: nat.mod. of sp2n}
 Let $\gg=sp(2n)$, $\gg'=sp(2n+2)$ and $\Delta=\{2,\ldots, n+1\}$. The braided symmetric algebra $S_\sigma (V_{\omega_1})$ of the natural $U_q(sp(2n))$-module $V_{\omega_1}$ is isomorphic to $U(w_\Delta)$ as a $U_q(sp(2n)$-module.  
\end{corollary}

\begin{proof}
Note that  the maximal semisimple subalgebra $\gg_\Delta$ of $\mathfrak{l}_\Delta$ is isomorphic to  $sp(2n)$. 
Using Theorem \ref{th:Heisenberg},  we have to show that  $U(w_\Delta)$ is a filtered $U_q(\mathfrak l_{\Delta})$-module algebra. 
The radical roots corresponding to $\Delta$ are of the form $\alpha_{1}+\sum_{i=2}^{j} \alpha_{i}$, $ \alpha_{1}+\ldots \alpha_{j}+ 2\alpha_{j+1}+\ldots +2\alpha_{n}+\alpha_{n+1}$  for $j\le n$, and $\alpha_{max}=2\alpha_{1}+\ldots 2\alpha_{n}+\alpha_{n+1}$. For convenience we will fix a reduced expression of $w_0$ such that the roots of $sp(2n)$ are ordered as
\begin{equation}
\label{eq:order of sp(2n)-roots}
\alpha_1<\alpha_1+\alpha_2<\ldots <\sum_{i=1}^{n+1}\alpha_i< \sum_{i=1}^{n-1}\alpha_i+2\alpha_n+\alpha_{n+1}<\ldots <\alpha_{max}<\ldots\ .
\end{equation}

We need the following fact.

\begin{lemma}
The quantum Schubert cell $U(w_\Delta)$ is quadratic linear. 
\end{lemma}

\begin{proof}
It is easy to see from \eqref{eq:order of sp(2n)-roots} that if $\alpha,\beta\le \alpha_{max}$ one cannot find $\alpha<\gamma_1\le\ldots \le \gamma_k<\beta$ with $k\ge 3$ such that $\alpha+\beta=\gamma_1+\ldots \gamma_k$. The assertion follows.
\end{proof}

 The lemma implies that $U(w_\Delta)$ is filtered. It remains to investigate the $U_q(\mathfrak l_{\Delta})$-action.
\begin{lemma}
Let $\gg'=sp(2n+2)$ and  $\Delta=\{2,3,\ldots,n+1\}\subset[1,n+1]$. Then, $U(w_\Delta)$ is a filtered $U_q(\mathfrak l_{\Delta})$-module algebra. 
\end{lemma} 

\begin{proof}
 
 It is obvious from the definition of the adjoint action \eqref{eq:adjoint action} and the defining relations of $U_q(\gg')$ (see Section \ref{se:q-group}) that $ad(F_i)(U(w_\Delta)_m)\subset U(w_\Delta)_m$ for all $m\in \ZZ_{\ge0}$ and $i\in\Delta$ as well as $ad(K_\lambda)(U(w_\Delta)_m)\subset U(w_\Delta)_m$ for all $\lambda\in P(\gg')$. It remains to check that $ad(E_i)(E_{\alpha_{i_1}}\ldots E_{\alpha_{i_m}})\in U(w_\Delta)_m $  for $i\in\Delta$ and radical roots $\alpha_{i_1},\ldots,\alpha_{i_m}$. We prove this by induction on $m$.
 
 Let $m=1$. Note that  $ad(E_i)(E_\alpha)\subset U(w_{\Delta})$ for all radical roots $\alpha$ and all $i\in [2,n+1]$ by Theorem \ref{th:PBW-Uw} (b). If $\alpha=\alpha_{max}$, then and \eqref{eq:q-commutators} and \eqref{eq:order of sp(2n)-roots} imply that  $ad(E_i)(E_\alpha)=0\in U(w_\Delta)$. If $\alpha<\alpha_{max}$, then $(\alpha,\omega_1)=1$ and hence one cannot find $\gamma_1\le \ldots\le\gamma_k\le \alpha_{max}$ with $k\ge 2$ such that $\alpha+\alpha_i= \gamma_1+ \ldots+\gamma_k$, since $(\alpha+\alpha_i,\omega_1)=1$ and $( \gamma_1+ \ldots+\gamma_k,\omega_1)\ge k$. We obtain that $ad(E_i)(U(w_\Delta)_1\subset U(w_\Delta)_1$.
   
 Let $m>1$ and  note that
 $$ad(E_i)(E_{\alpha_{i_1}}\ldots E_{\alpha_{i_m}})=ad(E_i)(E_{\alpha_{i_1}}) E_{\alpha_{i_2}}\ldots E_{\alpha_{i_m}}+q^r E_{\alpha_{i_1}} ad(E_i)(E_{\alpha_{i_2}}\ldots E_{\alpha_{i_m}})\ ,$$
for some $r\in\ZZ$. The assertion follows immediately from the inductive hypothesis. 
 The lemma is proved.     
 \end{proof}

Note that $rad_{\Delta}\cong V_{\omega_{1}}\oplus V_{0}$ as a $sp(2n)$-module. Hence, $(gr(U(w_\Delta))_{1}\cong V_{\omega_{1}}^q\oplus V_{0}^q$ as $U_q(sp(2n)$-modules. Therefore $gr(U(w_\Delta))$ is a flat deformation of $S(V_{\omega_{1}}\oplus V_{0})$ by Theorem \ref{th:Heisenberg}. 

Denote by $Sym_2$ the $U_q(sp(2n)$-module homomorphism  $Sym_2:(V^q_{\omega_{1}}\oplus V^q_{0})^{\otimes 2}\to gr(U(w_\Delta))_2$, given by the relations  defining the quadratic algebra $gr(U(w_\Delta))$. Note that
$$\ker(Sym_2)\subset V^q_{\omega_1}\otimes V^q_{\omega_1}\oplus V^q_0\otimes V^q_{\omega_1}\oplus V^q_{\omega_1}\otimes V^q_0\subset (V^q_{\omega_{1}}\oplus V^q_{0})^{\otimes 2}\ .$$

We have $\ker(Sym_2)\cong V^q_{\omega_2}\oplus V^q_{\omega_1}\oplus V^q_0$ and obtain that 
$$\ker(Sym_2)\cap  (V^q_{\omega_1}\otimes V^q_{\omega_1})\cong V^q_{\omega_2}\oplus V^q_0\ .$$

Since $V_{\omega_1}$ is multiplicity-free and weight-multiplicity-free  and $V^q_{\omega_2}\oplus V^q_0\cong \Lambda^2_\sigma V^q_{\omega_1}$, we obtain from Lemma \ref{le:mult-free tensor} that, analogous to the proof of Theorem  \ref{th: Abelian ->flat quadratic}, indeed $\Lambda^2_\sigma V_{\omega_1}=\ker(Sym_2)$. Thus, the subalgebra of $gr(U(w_\Delta))$ generated by $V_{\omega_1}^q$ is the braided symmetric algebra $S_\sigma (V_{\omega_1}^q)$.
 
Corollary \ref{co: nat.mod. of sp2n} is proved.
\end{proof}

\begin{remark}
The braided symmetric algebra $S_\sigma (V_{\omega_1}^q)$ of the natural $U_q(sp(2n)$-module can be obtained, by an argument similar to the proof of Corollary \ref{co: nat.mod. of sp2n}, as the quotient of $gr(U(w_\Delta))$ by the two-sided ideal generated by the copy of the trivial module $V_0$ in $gr(U(w_\Delta))_1$.
\end{remark}

\begin{problem}
Describe the $q$-deformed symmetric algebras associated to radicals of Heisenberg type, where the quantum radical $U(w_\Delta)$ satisfies the assumptions of Theorem \ref{th:Heisenberg}. These algebras cannot be braided symmetric algebras, but should provide examples for some more general concept of quantum symmetric algebra.
\end{problem}
 \subsection{Quantum Schubert Cells: PBW-Theorem and Levi action}
 \label{se: qschcells}
 Let $\gg$ be a complex reductive Lie algebra, and let $W$ be the Weyl group of $\gg$.
In this section we prove a PBW-type theorem for quantum Schubert cells $U(w)\subset U_q(\gg)$ associated to $w\in W$ and show that  if $w_\Delta$ is a parabolic element of the $W$, then the Hopf subalgebra $U_q(\mathfrak{l}_\Delta)\subset U_q(\gg)$ acts adjointly on $U(w_\Delta)$ (for definitions see the beginning of Section \ref{se: quantum schubert}).






  The following theorem is the main result of this section.
 \begin{theorem}
\label{th:PBW-Uw}
\noindent(a) Let $w\in W$ be an element of the Weyl group $W$ and $U(w)$ the corresponding quantum Schubert cell. The monomials
$E_{\alpha_{(1)}}^{\l_{(1)}}\ldots E_{\alpha_{(k)}}^{\l_{(k)}}$, satisfying $\l_{(i)}=0$ if $w(\alpha)\in R^{-}(\gg)$, 
form a $\CC(q)$-linear basis of $U(w)$.

\noindent(b)  If $w=w_{\Delta}\in W$ is parabolic, then $U_q(\mathfrak{l}_\Delta)$ acts adjointly  on $U(w_\Delta)$. 
\end{theorem}

 \begin{proof} 
 
 Prove (a) first.  Recall the PBW-theorem for $U_q(\gg)$.
  
  \begin{proposition} \cite[8.24]{Jan})
 \label{pr: PBW}
 
 \noindent(a) The monomials $E_{\alpha_{(1)}}^{\l_{(1)}}\ldots E_{\alpha_{(k)}}^{\l_{(k)}}F_{\alpha_{1}}^{m_{(1)}}\ldots  E_{\alpha_{(k)}}^{m_{(k)}} K^{\mu}$, with $\l_{(i)}, m_{(i)}\in \ZZ_{\ge0}$ and $\mu\in P(\gg)$, form a $\CC(q)$-linear basis of  $U_q(\gg)$. 
 
 \noindent(b) The monomials $E_{\alpha_{(1)}}^{\l_{(1)}}\ldots E_{\alpha_{(k)}}^{\l_{(k)}}$ with $\l_{(i)} \in \ZZ_{\ge0}$   form a $\CC(q)$-linear basis of  $U^{+} $. Similarly, the monomials $F_{\alpha_{(1)}}^{\l_{(1)}}\ldots F_{\alpha_{(k)}}^{\l_{(k)}}$ with $m_{(i)} \in \ZZ_{\ge0}$   form a $\CC(q)$-linear basis of  $U^{-} $.
  \end{proposition}

Denote by $\ell(w)$ the length of an element $w\in W$. 
The following fact relates  quantum  Schubert cells.

\begin{proposition}
\label{pr: relative Schubert cells}
\noindent (a) $U(w_0)=U^+$.

\noindent (b) Let $w,w'\in W$ such that $w_0=ww'$ and $\l(w)+\l(w')=\l(w_0)$. Then 
  $$U^+=U(w_0)=U(w) T_w(U(w'))\ , U(w)\cap T_w(U(w'))=\CC(q)\cdot {\bf 1}\subset U_q(\gg)\ .$$ 
\end{proposition}

\begin{proof}
We need the following fact.

\begin{lemma}
\label{weyl-action on roots} 
Let $\alpha\in R^+(\gg)$ and let $w\in W$. 
\begin{enumerate}
\item If $w(\alpha)\in R^+(\gg)$, then $T_w(E_\alpha)\in U^+$.
\item If $w(\alpha)\in R^-(\gg)$, then $T_w(E_\alpha)\in U_q(\bb^-)$.
\end{enumerate}
\end{lemma}

\begin{proof}
Note first the following fact.

\begin{lemma} 
\label{le:U+-reflections} 
\noindent(a) Let $\beta\in R^+(\gg)$. Then, $T_i(E_\beta)\in U^+$ if $s_i(\beta)\in R^+(\gg)$. \\
\noindent(b) Let $\beta\in R^+(\gg)$. Then, $T_i(E_\beta)\in U(\bb^-)$ if $s_i(\beta)\in R^-(\gg)$. Moreover,  $T_i( E_\beta)=-F_iK_{\alpha_i}$.\\

\noindent(c) Let $\beta\in R^+(\gg)$. Then $T_i(F_\beta K_\lambda) \in U_q(\bb^-)$.
\end{lemma}

\begin{proof}
 Prove (a) first. Let  $\beta=s_{i_1}\ldots s_{i_k}(\alpha_j)$ and let $w=s_i s_{i_1}\ldots s_{i_k}$(not necessarily reduced). One has $w(\alpha_j)=s_i(\beta)\in R^+(\gg)$.  It is well known that $w(\alpha_k)\in R^+(\gg)$ implies that $\ell(ws_k)=\ell(w)+1$. Hence there exist $w, w'\in W$ such that $w_0=ws_jw'$ with $\ell(w_0)=\ell(e)+\ell(w')+1$. That implies that for some choice of reduced expression, $E_{w(\alpha_j)}=T_w(E_j)=T_i(E_\alpha)\in U^+$. Part (a) is proved. 
 
 Prove (b) now. Note that if  $\beta\in R^+(\gg)$ and $s_i(\beta)\in R^-(\gg)$, then $\beta=\alpha_i$. The assertion follows from the definition of  the $T_i$ in\eqref{eq: T_i on Ui}. Part (b) is proved.

 In order to prove part (c) note first that  the $T_i$ are algebra homomorphisms and that $T_i(K_\lambda)\in U_q(\bb^-)$. The assertion now follows from an argument analogous to the proof of (a).   The lemma  is proved.
\end{proof}

Now, let $w=s_{i_1}\ldots s_{i_k}$ be a reduced expression of $w$.  If $w(\alpha)\in R^+(\gg)$, then $s_{i_j}\ldots s_{i_k}(\alpha)\in R^+(\gg)$ for all $1\le j\le k$. Indeed if  $s_{i_j}\ldots s_{i_k}(\alpha)\in R^-(\gg)$ for some $1\le j\le k$, then there exist $j_1\ge1$ such that $s_{i_{j_1+1}}\ldots s_{i_k}(\alpha)=\alpha_{\ell}$ and $s_{i_{j_1}}=s_{\ell}$. Hence, $s_{i_1}\ldots s_{i_{j_1-1}}(-\alpha_\ell)=w(\alpha)\in R^+(\gg)$. Recall the well known {\it exchange property} of the Weyl group: Let  $\hat w=s_{i_m}\ldots s_{i_{k-1}}s_{i_{k}}$ be a reduced expression of $w\in W$. If $w(\alpha_i)\in R^-(\gg)$, then there exists $m\le r\le k$ such that 
 
 $$ s_{i_r}\ldots s_{i_{k+1}}=s_{i_{r+1}}\ldots s_{i_k} s_{i}\ .$$

In our case  we obtain that $s_{i_1}\ldots s_{i_{j_1-1}}$ has a reduced expression $s_{i_1}\ldots s_{i_{j_1-1}}=s_{m_1}\ldots s_{m_{j_1-2}} s_{\ell}$, hence $w$ has an expression 
$$w=s_{m_1}\ldots s_{m_{j_1-2}} s_{\ell}s_{\ell} s_{i_{j_1+1}}\ldots s_{i_k}=s_{m_1}\ldots s_{m_{j_1-2}} s_{i_{j_1+1}}\ldots s_{i_k}\ ,$$
contradicting the assumption that $w=s_{i_1}\ldots s_{i_k}$  was reduced. It follows now inductively from Lemma \ref{le:U+-reflections}  that   $T_w(E_\alpha)\in U^+$.  Part (a) is proved.

Prove (b) now. If  $w(\alpha)\in R^-(\gg)$, and $w=s_{i_1}\ldots s_{i_k}$ is a reduced expression, then we can find, as in part (a) $1\le j\le k$ such that $w=s_{i_1} \ldots s_{i_{j-1}}s_{\ell}s_{i_{j+1}}\ldots s_{i_k}$, $s_{i_{j+1}}\ldots s_{i_k}(\alpha)=(\alpha_\ell)$. Employing  the exchange   property as in part (a) we have $s_{i_{j_1}}\ldots s_{i_{j-1}}(-\alpha_\ell)\in R^-(\gg)$ and $s_{i_{j_2}}\ldots s_{i_k}(\alpha)\in R^+(\gg)$ for $1\le j_1<j<j_2\le k$. Arguing as in the proof of Lemma \ref{le:U+-reflections}  (a) we obtain that $T_{i_{j+1}}\ldots T_{i_k}(E_\alpha)=E_{\ell}$, hence  Lemma \ref{le:U+-reflections}(b) and (c) yield that $T_w(E_\alpha)\in U_q(\bb^-)$. Lemma  \ref{weyl-action on roots} is proved.
\end{proof}

Now we are ready to complete the proof of Proposition \ref{pr: relative Schubert cells}. Part(a) follows directly from Lemma \ref{weyl-action on roots}(b), since $w_o(\alpha)\in R^-(\gg)$ for all $\alpha\in R^+(\gg)$.  

Prove (b) now. Let  $w_o=ww'$ and $\ell(w_0)=\ell(w)+\ell(w')$. It follows from   Lemma \ref{weyl-action on roots}  that $T_{w^{-1}}(E_\alpha)\in U^+$ or $T_{w^{-1}}(E_\alpha)\in U_q(\bb^-)$. Since $T_{w_0}(E_\alpha)\in U_q(\bb^-)$ we obtain that  $T_{w^{-1}}(E_\alpha)\in U(w')$ and $E_\alpha\in T_w(U(w'))$, if $T_{w^{-1}}(E_\alpha)\in U^+$.  Similarly we obtain that  $E_\alpha\in U(w)$ if  $T_{w^{-1}}(E_\alpha)\in U_q(\bb^-)$. The fact that the $T_i$ are algebra homomorphisms and the PBW-theorem (Proposition \ref{pr: PBW}(b)) now imply that  $U^+=U(w_0)=U(w)T_w(U(w'))$.   

It is easy to see that $U(w)\cap T_w(U(w')=\CC(q)\cdot{\bf 1}$, because 
$$\CC(q)\cdot {\bf 1}\subset T_{w^{-1}}(U(w)\cap T_w(U(w'))\subset U_q(\bb^-)\cap U^+=\CC(q)\cdot {\bf 1}\ .$$

Part (b) and Proposition  are proved. 
\end{proof}

We can now complete the proof of Theorem \ref{th:PBW-Uw} (a).
 Note that if $w^{-1}(\alpha_{j})\in R^{-}(\gg)$, then $T_{w^{-1}}(E_{\alpha_{j}})\in U_q(\bb^{-})$. Since $T_{w^{-1}}$ is an algebra automorphism we obtain that 
the  monomials
$E_{\alpha_{1}}^{\l_{1}}\ldots E_{\alpha_{k}}^{\l_{k}}$ with $\l_{i}=0$ if $w^{-1}(\alpha_{i})\in R^{+}(\gg)$ are elements of $U(w)$. $U(w)$ is an algebra, hence the linear span of the above monomials is contained in $U(w)$.   
The monomials are linearly independent by  Proposition \ref{pr: PBW}, hence it remains to show that they span $U(w)$.  Choose $w'\in W$ such that $w'w=w_{0}$ and  $\l(w')+\l(w)=\l(w_0)$. We showed in the proof of Proposition \ref{pr: relative Schubert cells} that  $E_{\alpha_{1}}^{\l_{1}}\ldots E_{\alpha_{k}}^{\l_{k}}\in T_w(U(w'))$ if $\l_{i}=0$ whenever $w^{-1}(\alpha_{i})\in R^{-}(\gg)$. Hence we can write each $u\in U^+$ by Proposition \ref{pr: PBW} as  $u=\sum_{i=1}^k u_i u_i'$, where the $u_i\in U(w)$ are linearly independent and $u_i'\in T_w(U(w'))$. It follows immediately that $T_{w^{-1}} (u)\in U_q(\bb^-)$, if and only if $T_{w^{-1}}u_i'\in U_q(\bb^-)$; i.e., if  $u_i'\in\CC(q)\cdot{\bf 1}$ for all $i$ by Proposition \ref{pr: relative Schubert cells}(b). Theorem \ref{th:PBW-Uw} (a) is proved.

We will now prove Theorem \ref{th:PBW-Uw} (b). It suffices to show that the $E_i, F_i$, $i\in\Delta$ and $K_\lambda$, $\lambda\in P(\gg)$  which generate $U(\mathfrak{l}_\Delta)$ act adjointly on $U(w_\Delta)$.

 \begin{proposition}
 \label{pr: Ui-action}
 \noindent (a) Let $w\in W$. Then $K_\lambda$, $\lambda\in P(\gg)$ acts adjointly on $U(w)$.
\noindent (b) Let $w_\Delta\in W$ be parabolic and let  $w_0=w_{0,\Delta}w_\Delta$. If  $i\in \Delta$, then,$E_i$ and $F_i$ act adjointly on $U(w_\Delta)$.  
 \end{proposition} 
 
\begin{proof}
Prove (a) first. Let $w_0=w'w$.
In order to prove the assertion it suffices by  Lemma \ref{le:upperdefofqsc} to show that $T_{w'^{-1}}(ad(x)(K_\lambda))\in U^+$ for all $x\in U^i$ and $u\in U^w$. 
We obtain that  $K_{\lambda}$, $\lambda\in P(\gg)$ acts  on  $U(w)$ since
 $$T_{w'^{-1}}(ad(K_{\lambda})u)=T_{w'^{-1}}(K_{\lambda})T_{w'^{-1}}u  T_{w'^{-1}}(K_{\lambda})\in U^+$$
  for all $u\in U(w)$.  Part(a) is proved.

Prove (b) now. We need the following fact.
\begin{lemma}
\label{le: Levi onsimple roots}
 Let $w\in W$ be an  element of the Weyl group $W$, and  $\alpha_{i},\alpha_{j}$  simple roots such that $w(\alpha_{i})=-\alpha_{j}$. Then, $T_{w}(E_{i})=- F_{j}K_{\alpha_j}$ and $T_{w}(F_{i})= -K_{-\alpha_j}E_{i}$. \end{lemma}  
 
 \begin{proof}
 
 Recall that by  the exchange property (see proof of Lemma \ref{weyl-action on roots})
we can choose a reduced expression $w=w's_i= s_{i_1}\ldots s_{i_k}s_i$ of $w$ such that $s_{i+m}\ldots s_{i_k}(\alpha_i)\in R^+(\gg)$. Note that if $w'(\alpha_i)=\alpha_j$, then $T_{w'}(E_i)=E_j$ and $T_{w'} (F_i)=F_j$ We compute using \eqref{eq: T_i on Ui}
$$ T_w(E_i)= T_w'(T_i(E_i))=T_{w'}(-F_i K_{\alpha_i})=-F_j K_{\alpha_j}\ , $$
$$T_w(F_i)=-T_w'(T_i(F_i))=T_{w'}(K_{-\alpha_i}E_i)=-K_{-\alpha_j} E_j\ .$$ 
 
The lemma is proved.
 \end{proof}

Suppose that $i\in\Delta$. Note that 
 $w_{0,\Delta}^{-1}(\alpha_i)=-\alpha_j$ with $j\in \Delta$, and hence $ T_{w_(0,\Delta)^{-1}}(F_i)=K_{\alpha_j}E_j$ by Lemma \ref{le: Levi onsimple roots}. 
 
Let $u\in U(w_\Delta)$.  We show that  $ad(F_{i})u\in U^w$, if $u\in U(w)$ by computing
 
 $$ T_{w_{0,\Delta}^{-1}}(ad(F_{i})u)= T_{w_{0,\Delta}^{-1}}(F_{i})  T_{w_{0,\Delta}^{-1}}(u)  T_{w_{0,\Delta}^{-1}}(K_{-\alpha_{i}})- T_{w_{0,\Delta}^{-1}}(u) T_{w_{0,\Delta}^{-1}}(F_{i}) T_{w_{0,\Delta}^{-1}}(K_{\alpha_{i}})=$$
 $$-K_{-\alpha_{j}}E_{j} T_{w_{0,\Delta}^{-1}}(u) K_{\alpha_{j}} -T_{w_{0,\Delta}^{-1}}(u)K_{-\alpha_{j}} E_{j}K_{\alpha_{j}} \in U_q(\nn^{+})\ ,$$
because $K_{-\alpha_i} m K_{\alpha_i}=q^r m$ for every monomial $m$ in $U(w)$. Note that the proof does not require $w_{\Delta}$ to be a parabolic element. However, the assumption will be needed to prove the assertion for the action of $E_i$.
 
 Choose $i\in \Delta$ and choose a reduced expression $w_0=w_{0,\Delta} w$ such that $E_{w_{0,\Delta}}=E_i$. To complete the proof of the proposition it suffices by Theorem \ref{th:PBW-Uw}(a)  to show that $ad(E_i)(E_{\alpha_{i_1}}\ldots E_{\alpha_{i_\ell}})\in U(w)$, if $i\in\Delta$ and $E_{\alpha_{i_j}}\in U(w)$ for $j\in[1,\ell]$. We use induction on $\ell$.

Consider the case $\ell=1$; i.e., we have to show that $ad(E_i)(E_\alpha)\in U(w)$ if $i\in\Delta$ and $E_{\alpha}\in U(w)$.  Note that by our choice of $w_0$ we have that  if $ \alpha_i< \alpha$ for some root $\alpha$, then  $E_\alpha\in U(w)$ by  Theorem \ref{th:PBW-Uw}(a).  Lemma \ref{le:q-commutators} yields that

$$ad(E_i)(E_\alpha)\in \text{span}(E_{\gamma_{1}}\ldots E_{\gamma_{k}})\ ,$$
where $\alpha<\gamma_{1}\le\ldots\le \gamma_{k}<\beta$, and  therefore $ ad(E_i)(E_\alpha)\in U(w)$ as desired.

Now consider the case when $\ell>1$. Note that by a straightforward calculation

$$ ad(E_i)(E_{\alpha_{i_1}}\ldots E_{\alpha_{i_\ell}})= ad(E_i)(E_{\alpha_{i_1}} (E_{\alpha_{i_2}}\ldots E_{\alpha_{i_\ell}})+ q^{r} E_{\alpha_{i_1}} ad(E_i)(E_{\alpha_{i_2}}\ldots E_{\alpha_{i_\ell}})$$ 
 for some $r\in \ZZ$, and hence $ ad(E_i)(E_{\alpha_{i_1}}\ldots E_{\alpha_{i_\ell}})\in U(w)$ by the inductive hypothesis. 
Part (b) is proved. Proposition \ref{pr: Ui-action} is proved.
 \end{proof}   

Theorem \ref{th:PBW-Uw} (b) is proved.
 \end{proof}

  



  \section{Proof of Theorem \ref{th: class of Poisson}}
 \label{se:proof of classification} 

\subsection{Necessary Conditions}
\label{se:nec.cond}
In this section we establish necessary conditions a weight $\lambda\in P(\gg)$ has to satisfy if  the simple module $V_\lambda$ is Poisson; i.e., we prove the "only if" assertion of  the equivalence of (a) and (f) in Theorem  \ref{th: class of Poisson}.  
Before we proceed with the proof of Theorem \ref{th: class of Poisson} we will have introduce some convenient notation.  Since any  finite-dimensional module $V$    over a semisimple Lie algebra $\gg$ splits as  a direct sum of weight  spaces $V=\bigoplus_{\mu\in P(\gg)} V(\mu)$, we will use the abbreviation $V(\mu)^c=\bigoplus_{\nu\ne \mu} V(\nu)$, as the "standard" complement of $V(\mu)$. Additionally we will use the  notation and results from  {\bf Appendix} \ref{se:appendix}. 

First we have to calculate ${\bf c}$ explicitly.

 \begin{lemma}
 \label{le:explicit r-}
 Let $\gg$ be a semisimple Lie algebra of rank $r$ and $c$ its Casimir element.Then (up to a constant multiple)
\begin{equation}
\label{eq:explicit r-}
{\bf c}=[c_{12},c_{23}]=\sum_{\alpha,\beta\in R^+}\frac{(\alpha,\alpha)(\beta,\beta)}{4}E_\alpha\wedge [E_{-\alpha}, E_{\beta}]\wedge E_{-\beta}\ .
\end{equation}
\end{lemma} 

 \begin{proof}
 Choose a basis $H_1,\ldots H_r$ for $\hh$ which is orthonormal with respect to the Killing form. It is well known that the Casimir element $c$ is up to a constant $c=\sum_{\alpha\in R} (\alpha,\alpha)E_\alpha\otimes E_{-\alpha}+\sum_{i=1}^r H_i\otimes H_i$. We calculate
 
 \begin{equation}
 \label{eq:c-explixit}
  {\bf c}=[c_{12}, c_{23}]=\sum_{\alpha,\beta\in R} (\alpha,\alpha)(\beta,\beta) E_{\alpha}\otimes [E_{-\alpha},E_\beta]\otimes E_{-\beta}
  \end{equation}
  $$+\sum_{\alpha\in R, i=1}^r(\alpha,\alpha) \left(E_{\alpha}\otimes [E_{-\alpha},H_i]\otimes H_i+(\alpha,\alpha) H_i\otimes [H_i,E_\alpha]\otimes E_{-\alpha}\right)\ . 
 $$
 
 It is easy to see that for all the summands $X\otimes Y\otimes Z$ we have $\{X,Y,Z\}\cap\{E_\alpha:\alpha\in R^+\}\ne \emptyset$ and $\{X,Y,Z\}\cap\{E_{-\alpha}:\alpha\in R^+\}\ne \emptyset$. The element ${\bf c}\in \gg^{\otimes 3}$ is skew-symmetric by Lemma \ref{le:c-inv.and skew}(a), hence we can write 
 ${\bf c}=\sum_{\alpha,\beta\in R^+} E_{\alpha}\wedge X_{\alpha,\beta}\wedge E_{-\beta}$. 
 
 It follows from \eqref{eq:c-explixit} that  or all $\alpha,\beta\in R^+$:
 $$ E_{\alpha}\wedge X_{\alpha,\beta}\wedge E_{-\beta}= (\alpha,\alpha)(\beta,\beta) E_{\alpha}\otimes [E_{-\alpha},E_\beta]\otimes E_{-\beta}+ other terms\ .$$
 This yields that $E_{\alpha}\wedge X_{\alpha,\beta}\wedge E_{-\beta}= 6 \sum_{\alpha,\beta\in R^+}(\alpha,\alpha)(\beta,\beta) E_\alpha\wedge [E_{-\alpha}, E_{\beta}]\wedge E_{-\beta}$. Rescaling shows that the lemma is proved.
 \end{proof}

 The following result will now allow us  to observe that large classes of simple modules are ``too big'' to be Poisson.
 
 \begin{lemma}
 \label{le: doubleweightroot}
Let  $\gg$ be  a complex simple Lie-algebra,  $P(\gg)$ its weight-lattice and $R(\gg)\subset P(\gg)$ the corresponding root-system with basis $S=\{\alpha_{1},\ldots \alpha_{n}\}$. Denote by $w_{0}\in W$ the longest element of the Weyl group $W$. Let  $\lambda\in P^{+}(\gg)$ be a dominant weight, such that $w_{0}(\lambda)=-\lambda$. If  $V_{\lambda}$ is Poisson, then  $(2\lambda-\alpha_{i})\in R(\gg)\cup \{0\}$ for all $\alpha_i$ such that $(\lambda,\alpha_i)\ne 0$.
\end{lemma}

\begin{proof}
 Let $v_{\lambda}\in V_{\lambda}(\lambda)$ be a highest weight vector, and suppose that $2\lambda-\alpha_{i}\notin R(\gg)\cup \{0\}$. Let $\alpha_{i}$ be a simple root such that $(\lambda,\alpha_i)\ne 0$. Since $\lambda$ is dominant,  $(\lambda,\alpha_i)>0$. Set $v'=F_{-\alpha_{i}}(v)\ne0$.  We have, by assumption,  $V_{\lambda}(-\lambda)\ne 0$, since $w_{0}(\lambda)=-\lambda$, and clearly $(\alpha_{i}|\-\lambda)< 0$. Therefore,  we obtain for all $v''\in V_{\lambda}(-\lambda)$  
$$ E_{\alpha_{i}}\wedge F_{\alpha_{i}}\wedge H_{\alpha_{i}}(v\wedge v'\wedge v'')=c v\cdot v'\cdot v''+ \overline u\ ,$$

where $\overline u\in (V(\lambda)\cdot V(\lambda-\alpha)\cdot V(-\lambda))^{c}$.
If $2\lambda-\alpha_{i}\notin R(\gg)\cup \{0\}$, then  $E_{\alpha}(v'')\notin V(\lambda-\alpha_{i})$ for all $\alpha\in R(\gg)$. Additionally, if $2\lambda-\alpha_{j}\notin R(\gg)\cup \{0\}$ for all $j\in[1,n]$, then  $2\lambda$ is not a root, either, and hence  $E_{\alpha}(v'')\notin V(\lambda)$ for all $\alpha\in R(\gg)$. Therefore,  ${\bf c} (v\wedge v'\wedge v'')=c v\cdot v'\cdot v''+\overline u'$, where $\overline u'\in (V(\lambda)\cdot V(\lambda-\alpha)\cdot V(-\lambda))^{c}$, as defined in the {\bf Appendix} in Lemma  \ref{le:powerdecomp} and \eqref{eq: complement}.
We obtain that ${\bf c}(v\wedge v'\wedge v'')\ne 0$ and, hence, $V_{\lambda}$ is not Poisson. The lemma is proved.
\end{proof}
 
Lemma  \ref{le: doubleweightroot} has the following consequence.

\begin{proposition}
\label{pr:non flatif too big}
Let $\gg$ be a complex  simple Lie algebra, not isomorphic to $E_{6}$ or $sl_{n}(\CC)$, and let $\lambda\in P^+(\gg)$. If $V_\lambda$ is Poisson, then  $2\lambda-\alpha_i\in R(\gg)$ for all simple roots $\alpha_i$ such that $(\lambda,\alpha_i)\ne 0$. \end{proposition}

\begin{proof}

Recall the following well known fact.

\begin{lemma}
\label{le:wnoton fund weights}
Let $\gg$ be a simple Lie algebra of rank $r$ and let $P(\gg)$ be  its weight lattice, spanned by the fundamental weights $\omega_i$, $i=1,\ldots r$, labeled according to  \cite[Tables]{Bou}. Denote by $W(\gg)$ the Weyl group and by $w_0$ the longest element of $W(\gg)$. We have $w_0^2=1\in W(\gg)$ and:

\noindent(a) if $\gg=sl_{r+1}(\CC)$, then $w_0(\omega_i)=-\omega_{r-i}$.

\noindent(b) if $\gg=E_6$, then $w_0(\omega_1)=-\omega_6$, $w_0(\omega_3)=-\omega_5$ and $w_0(\omega_2)=-\omega_2$ and $w_0(\omega_3)=-\omega_3$.

\noindent(c) if $\gg\ne sl_{r+1}, E_6$, then $w_0(\omega_i)=-\omega_{i}$ for all $i$.
\end{lemma}
 Lemma \ref{le:wnoton fund weights} yields that the assumptions of   Lemma \ref{le: doubleweightroot} are satisfied for all $\lambda\in P^+(\gg)$, if $\gg$ is not isomorphic to either $sl_n$ or $E_6$.  Therefore,  Proposition \ref{pr:non flatif too big} now follows from  Lemma \ref{le: doubleweightroot}.   \end{proof}

Another very useful result is  the following. 
\begin{lemma}
\label{le; sl2-poisson}
Let $\gg=sl_{2}$ and $V_{\l}$ be a $\l+1$-dimensional simple $sl_{2}$-module. It is Poisson if and only if $\l\le2$.
\end{lemma} 
\begin{proof}
If $\l=0,1$, then $V_{\l}$ is Poisson, because $\Lambda^{3}V_{\l}=0$. If $\l=2$, then  $\Lambda^3V_2\cong V_0$ is a trivial module and since $S^3 V_2\cong V_6\oplus V_2$ we obtain that $Hom_\gg(\Lambda^3 V_2,S^3 V_2)=\{0\}$ and hence $V_2$ is Poisson.
Now, suppose that $\ell\ge 3$. Let $E,F,H$ be a standard basis of $sl_{2}$. Since $\Lambda^3 (sl_2)=span(E\wedge F\wedge H)$ we have ${\bf c}=   E\wedge F\wedge H$. Choose a weight basis $v_{0},\ldots, v_{\l}$ of $V_{\l}$ such that $H(v_{i})=(\l-2i)v_{i}$, $E(v_{i})=iv_{i-1}$ and $F(v_{i})=(\l-i)v_{i+1}$.     
  If, however, $\l\ge 3$, then it is easy to see that 
$${\bf c}(v_{1}\wedge v_{0}\wedge v_{\l})=   E\wedge F\wedge H(v_{1}\wedge v_{0}\wedge v_{\l})= (-\l^{2})v_{1}\cdot v_{0}\cdot v_{\l}\ne 0\  .$$
 
 Hence, $V_{\l}$ is not Poisson, if $\l\ge 3$. The lemma is proved.
\end{proof}

Another, very powerful tool will be a special case of Proposition  \ref{pr: subalgebras}. Recall that a parabolic subalgebra $\pp$ of a semisimple Lie algebra $\gg$ is a Lie subalgebra containing  a Borel subalgebra of $\gg$, and that $\pp$ splits as a semidirect product $\gg= \mathfrak{l}\ltimes \nn$ of a reductive Lie-algebra $\mathfrak{l}$, the {\it Levi subalgebra} and a nilpotent Lie algebra $\nn$, the {\it nil-radical}. Let $\mathfrak{l}\cong \gg'\oplus\zz$, where $\gg'$ is a semisimple Lie algebra and $\zz$ is the center of $\mathfrak{l}$.  Recall that  if $W$ is a simple $\gg$-module, then the $W$-isotypic component
 of a $\gg$-module $V$ is the submodule of $V$ isomorphic to $(Hom_{\gg}(W,V))\otimes W$.  
 
\begin{proposition}
\label{pr:Levisub}
Let $\gg$ be a  semisimpleLie algebra, and let $\pp$ be a parabolic subalgebra with Levi subalgebra $\mathfrak{l}\cong \gg'\oplus \zz$, where $\gg'$ is semisimple and $\zz$ is the center of $\mathfrak{l}$. Let $V$ be a $\gg$-module, and let $V$ split as an $\mathfrak{l}$-module into a direct sum of isotypic components $V\cong_{\mathfrak{l}} \bigoplus V_{i}$. If $V$ is a Poisson $\gg$-module, then each $V_{i}$ must be a Poisson  $\mathfrak{l}$-module. Moreover, $V_i$ must be Poisson as a $\gg'$-module.
\end{proposition}   
   
  \begin{proof} 
 The Lie algebra $\gg$ splits, in the notation of Proposition  \ref{pr: subalgebras}, as a vector space into $\gg=\gg'\oplus \gg'^\perp$ where $\gg'^\perp=\zz\oplus \nn$. Choose an  $\mathfrak{l}$-isotypic component $V_i$.
  Denote by $\mathfrak{B}$ be the set of all weights $\beta\in \hh^{*}$ such that the  weight spaces $V_{i}(\beta)\ne 0$.  Consider the decomposition $  V\cong V_{\mathfrak{B}}\oplus V_{\mathfrak{C}}$ where $ V_{\mathfrak{B}}\cong \bigoplus_{\beta\in \mathfrak{B}}   V(\beta)$ and $ V_{\mathfrak{C}}\cong\bigoplus_{\gamma\notin \mathfrak{B}} V(\gamma) $. We need the following fact. 
\begin{lemma}
\label{le:root preserving}
 Let $\gg$ and $\gg'$ be as assumed in Proposition \ref{pr:Levisub}. Let $V$ be a finite-dimensional $\gg$-module, and $U\subset V$ an isotypic $\gg'$-module component of $V$. Let $\beta\in \mathfrak{B}$ and  $u(\beta)\in U$ be a weight vector. Then $E_{\alpha}(v_{\beta})\in V_{\mathfrak{B}}\backslash \{0\}$ implies that $\alpha\in R(\gg')$.  
 \end{lemma}

\begin{proof}
Let $\alpha\in R(\gg)$, and $u_{\beta}\in U(\beta)$ for some $\beta\in  \mathfrak{B}$. Then, $E_{\alpha}(v_{\beta})\in V(\alpha+\beta)$.  Clearly $(\alpha+\beta)\in  \mathfrak{B}$ implies that $\alpha=(\alpha+\beta)-\beta$ lies in the $\ZZ$-linear span of $R(\gg')$, or, equivalently, in the  $\ZZ$-linear span of a basis of $R(\gg')$.  Since a root-system is determined uniquely by its basis, this implies that  $\alpha\in R(\gg')$. Therefore, $E_{\alpha}(U)\subset  V_{\mathfrak{C}}$, if $\alpha \in (R(\gg)\backslash R(\gg'))$.
 Lemma \ref{le:root preserving} is proved.
\end{proof}

We can now apply Proposition \ref{pr: subalgebras} by choosing $\gg_{sub}=\mathfrak{l}$ and  $V_1=V_i$. Therefore $V$ is Poisson if and only if $V_1$ is Poisson as a $\mathfrak{l}$-module. The $\mathfrak{l}$-module $V_1$ is Poisson if and only if $V_1$ is Poisson as a $\gg'$-module (Lemma \ref{pr:ssvred}). Proposition \ref{pr:Levisub} is proved.  
 \end{proof}

We can now derive the following criterion, which allows us to reduce the classfication-problem to a few cases.  
\begin{lemma}
\label{le:sl2-reduct}
  Let $\gg$ be a simple Lie algebra, and let $\lambda\in P^{+}(\gg)$ be a dominant weight . If $V_{\lambda}$ is Poisson, then $(\lambda|\alpha)\le 2$ for all  roots $\alpha\in R^+(\gg)$. 

\end{lemma}

\begin{proof} Consider the subalgebra $\gg_{\alpha}=E_{\alpha_{i}}\oplus E_{-\alpha_{i}}\oplus\hh\subset\gg$. Clearly,  $\gg_\alpha$ is isomorphic to  $sl_{2}\oplus\CC^{rank(\gg)-1}$. Denote by $\gg_\alpha^{\perp}$ the vector space complement of $\gg_\alpha$ spanned by the $E_\beta$ for $\alpha\ne\pm \beta\in R(\gg)$. Let $v_{\lambda}\in V_{\lambda}(\lambda)$ be a highest weight vector in $V_{\lambda}$. Then, $v$ generates a simple $((\lambda,\alpha_{i})+1)$-dimensional $\gg'$-module $V_{\l}$ and $V_\lambda=V_\l\oplus V_\l^\perp$ as $\gg_\alpha$-modules. We have $\gg_\alpha^{\perp}(V_\l)\subset V_\l^\perp$ by Lemma \ref{le:root preserving}.
 We can now apply Proposition \ref{pr: subalgebras} and obtain that if $V$ is Poisson, then $V_{\l}$ is Poisson by Proposition \ref{pr: subalgebras} as a $\gg'$-module. Hence,  $V_\l$ is Poisson as a $sl_2$-module  by Proposition  \ref{pr:ssvred}. This implies that  $(\lambda|\alpha)+1) \le 3$  by Lemma \ref{le; sl2-poisson}. Part(a) is proved. The lemma is proved.
\end{proof}

We will now address the necessary conditions on $\lambda\in P^+(\gg)$ by type of Lie algebra.

\subsubsection{The case of $\gg=sl_n$}
\label{susec:A_n} 
 \begin{claim}
 \label{claim:sl-n}
 Let $\gg=sl_{n}$. If $V_\lambda$ is Poisson then $\lambda\in\{ \omega_1,2\omega_1, \omega_2, \omega_{n-2},\omega_{n-1},2\omega_{n-1}\}$

 \end{claim}
 \begin{proof}
  Lemma \ref{le:sl2-reduct} has the has the following  consequence.

\begin{lemma}
\label{le:weight<3}
Let $\gg=sl_{n}$ and let $\lambda=\sum_{i=1}^{n-1}\l_{i}\omega_{i}$ be a dominant weight. If $V_{\lambda}$ is Poisson,  then  $\sum_{i=1}^{n-1}\l_{i}\le 2$. 
\end{lemma}

\begin{proof}

Recall that if $\gg=sl_n$, i.e. of type $A_{n-1}$ , then the highest root $\alpha_{max}= \sum_{i=1}^{n-1}\alpha_i$. Since $(\alpha_i,\omega_j)=\delta_{ij}$ for all $i,j\in[1,n-1]$, we obtain that $(\alpha_{max}|\lambda)=\sum_{i=1}^{n-1} \l_i$. The assertion now follows from  Lemma \ref{le:sl2-reduct} (b).
\end{proof}
  
It remains to prove that $V_\lambda$ is not Poisson if $\lambda=\omega_i+\omega_j$, $i\ne j$, $\lambda=2\omega_k$, $2\le k\le n-2$ or $\lambda=\omega_\l$ with $3\le \l\le n-3$. We will consider them case by case.

\begin{lemma}
\label{le;adjoint not Poisson}
\noindent(a) Let $\gg=sl_{n}$, $n\ge 3$ and let $\lambda=\omega_{1}+\omega_{n-1}$. Then, $V_{\lambda}$, the adjoint module, is not Poisson.

\noindent(b)Let $\gg=sl_{n}$, $n\ge 3$, and let $\lambda=\omega_{i}+\omega_{j}$, $1\le i<j\le n-1$. Then $V_{\lambda}$ is not Poisson. 
\end{lemma}

\begin{proof}
Prove (a) first. Denote by $w_{0}$ the longest element of the Weyl group $W$. It is well known that $w_{0}(\omega_{i})= -\omega_{n-i}$, and hence $w_{0}(\lambda)=-\lambda$.  We know that $\lambda=\omega_{1}+\omega_{n-1}=\alpha_{max}$, the highest root, since $V_{\lambda}$  is the adjoint module.  It is easy to see that  $2\alpha_{max}-\alpha_{i}$ is not a root for all $i\in [1,n-1]$. Therefore, $V_{\lambda}$ is not Poisson by Lemma \ref{le: doubleweightroot}. Part(a) is proved.

Prove (b) next.  Consider the Levi subalgebra $\mathfrak{l}_{i,j}$ of $\gg$ obtained by removing the first  $i-1$ nodes and the last $n-j-1$ nodes from the Dynkin diagram $A_{n-1}$. Clearly,  $\mathfrak{l}_{i,j}\cong sl_{j-i+1}\oplus \CC^{n-j+i-1}$, where $\CC$ denotes the trivial $sl_{n}$-module.  Any vector $0\ne v_\lambda\in V_{\lambda}(\lambda)$ generates  an adjoint $sl_{j-i+1}$-module, which is not Poisson by part (a).  Therefore, $V_{\lambda}$ is not Poisson by Proposition \ref{pr:Levisub} The lemma is proved.
\end{proof}

Now, consider simple modules of highest weight $\omega_{i}$ and $2\omega_{i}$ for $3\le i\le n-3$.
 
 \begin{lemma}
 \label{le:omega3np}
 \noindent(a) Let $\gg=sl_{6}$ and $\lambda=\omega_{3}$.  The simple module $V_{k\omega_{3}}$ is not Poisson for all $k\ge 1$.
 
 \noindent(b)  Let $\gg=sl_{n}$, $n\ge 6$ and let $3\le i\le n-3$. If $V_{\lambda}$ is Poisson, then $\lambda\ne k\omega_{i}$ for $3\le i\le n-3$.
 \end{lemma}

 \begin{proof}
 Prove (a) first. We have $\omega_{3}=\frac{1}{2} \alpha_{1}+\alpha_{2}+\frac{3}{2}\alpha_{3}+\alpha_{4}+\frac{1}{2}\alpha_{5}$.  It is easy to see that $2k \omega_{3}-\alpha_{i}$ is not a root for all $i\in [1,5]$, since the highest root is $\alpha_{max}=\sum_{i=1}^{5}\alpha_{i}$. Part(a) is proved.
 
Prove (b) next.  Consider the Levi subalgebra $\mathfrak{l}_{i-2,i+2}$ of $\gg$ obtained by removing the first  $i-3$ nodes and the last $n-i-3$ nodes from the Dynkin diagram $A_{n-1}$. Clearly,  $\mathfrak{l}_{i-2,i+2}\cong sl_{6}\oplus \CC^{n-6}$ .  Any vector $0\ne v_{k\omega_i}\in V_{k\omega_{i}}(k\omega_{i})$ generates  an  $sl_{6}$-module isomorphic to $V_{k\omega_3}$, which is not Poisson by part(a).  Therefore,  $V_{k\omega_{i}}$ is not Poisson by Proposition \ref{pr:Levisub}. 
The lemma is proved.  
 \end{proof}
 
 The following lemma addresses the last case.

\begin{lemma}
\label{le:2omega 2}
\noindent(a) If $\gg=sl_4$, then $V_{2\omega_2}$ is not Poisson.

\noindent(b) If $\gg=sl_n$, $n\ge 4$, then $V_{2\omega_2}$ and $V_{2\omega_{n-2}}$ are not Poisson.
\end{lemma}

\begin{proof}
Prove (a) first. We have $2\omega_{2}= \alpha_{1}+2\alpha_{2}+\alpha_{3}$, and $w_0(2\omega_2)=-2\omega_2$ by Lemma \ref{le:wnoton fund weights}. It is easy to see that $4\omega_{3}-\alpha_{i}$ is not a root for all $i\in [1,3]$, since the highest root is $\alpha_{max}=\sum_{i=1}^{3}\alpha_{i}$. Part(a) is proved.

Prove (b) now. Consider the Levi subalgebras $\mathfrak{l}_{1,3}$  and  $\mathfrak{l}_{n-3,n-1}$ of $\gg$ obtained by removing  the last $n-3$ nodes (resp. the first $n-3$) from the Dynkin diagram $A_{n-1}$.   Clearly,  $\mathfrak{l}_{1,3}\cong\mathfrak{l}_{n-3,n-1}\cong sl_{3}\oplus \CC^{n-3}$. Any vector $0\ne v_{2\omega_2}\in V_{2\omega_{2}}(2\omega_{2})$ (resp.  $0\ne v_{2\omega_{n-2}}\in V_{2\omega_{n-2}}(2\omega_{n-2})$) generates an $sl_{3}$-module $V_{2\omega_{2}}$ which is not Poisson by part (a). Therefore, $V_{2\omega_2}$  is not Poisson by Proposition \ref{pr:Levisub}.  The lemma is proved.  

\end{proof}

Claim \ref{claim:sl-n} is proved. 
 \end{proof}

\subsubsection{The case of $ \gg=so(2n+1)$}
\label{se:so(2n+1)}
 
\begin{claim}
\label{claim:so(2n+1)}
Let $\gg=so(2n+1)$. If $V_\lambda$ is Poisson, then $\lambda=\omega_1$ or $(\gg,V_\lambda)=(so(5), V_{\omega_2})$.
\end{claim}
\begin{proof}

Consider first the case of $\gg=so(5)$.
 Let $\{\alpha_{1},\alpha_{2}\}$ be a basis of the rootsystem $R(\gg)$. The fundamental weights are $\omega_{1}=\alpha_{1}+\alpha_{2}$ and $\omega_{2}=\frac{\alpha_{1}}{2}+\alpha_{2}$ and the highest root is $\alpha_{max}=\alpha_{1}+2\alpha_{2}=2\omega_{2}$.  It is easy to verify that  if  $\lambda\in P^+(\gg)$ and the weight $2\lambda-\alpha_{i}\in R(\gg)$ for some  $i=1,2$ imply that  $\lambda\in\{\omega_{1},\omega_2\}$.
 Employing Proposition \ref{pr:non flatif too big} we obtain immediately that if $V_{\lambda}$ is Poisson, then $\lambda\in\{ \omega_{1},\omega_2\}$.

We now consider $\gg=so(2n+1)$, $n\ge3$. Let $R(\gg)$ be the corresponding root system and let $\{\alpha_{1}, \alpha_{2},\ldots, \alpha_{n}\}$ be a basis. The fundamental weights are, for $1\le i\le n-1$, 
$$\omega_{i}=\alpha_{1}+2 \alpha_{2}+\ldots (i-1)\alpha_{i-1}+i(\alpha_{i}+\ldots \alpha_{n})$$
$$\omega_{n}=\frac{1}{2} (\alpha_{1}+2\alpha_{2}+\ldots n\alpha_{n})\ .$$
 The highest root is $\alpha_{max}=\alpha_{1}+2\alpha_{2}+\ldots +2\alpha_{n}$.  It is easy to verify that  for  $\lambda\in P^+(\gg)$ and  $2\lambda-\alpha_{i}\in R(\gg)$ for some  $i\in [1,n]$ imply that $\lambda=\omega_{1}$ or, in the case $n=3$, $\lambda=\omega_{3}$. We now obtain immediately from Proposition \ref{pr:non flatif too big}  that  if $V_{\lambda}$ is Poisson and $n\ge 4$  then $\lambda=\omega_{1}$.  

Consider the case $n=3$. We have to show that $V_{\omega_3}$ is not Poisson. Note that  $\lambda=\omega_{3}=\frac{1}{2}\alpha_{1}+\alpha_{2}+\frac{3}{2}\alpha_{3}$ and $\alpha_{max}=\alpha_1+2\alpha_2+2\alpha_3$. Let $v\in V_{\omega_{3}}(\omega_3)$, $v'=E_{\alpha_{3}}(v)$ and $v''=E_{\alpha_{max}}(v')$. Clearly $v'\ne 0 $, $v''\ne 0$, $v'\in V_{\omega_{3}}(\frac{1}{2}\alpha_{1}+\alpha_{2}+\frac{1}{2}\alpha_{3})$ and $v''\in V_{\omega_{3}}(- \omega_3)$. Consider roots, $\alpha,\beta\in R^+(\gg)$. We have for all $\alpha,\beta\in R(\gg)$, $w(\gamma_i)\in V(\gamma_i)$,   

$$ E_\alpha\wedge[E_{-\alpha},E_\beta]\wedge E_{-\beta}(v\wedge v'\wedge v'')=\sum_{i=1}^k w(\gamma_i)\cdot w(\gamma'_i)\cdot w(\gamma''_i)\ ,$$ 
  where $w(\gamma_i)\in V(\gamma_i)$, $w(\gamma'_i)\in V(\gamma'_i)$,  and $w(\gamma''_i)\in V(\gamma''_i)$ such that $\gamma_i+\gamma_i'+\gamma_i''=\frac{1}{2}\alpha_{1}+\alpha_{2}+\frac{1}{2}\alpha_{3}$. We obtain that $(\gamma_i,\gamma_i',\gamma_i'')=(\omega_3, \frac{1}{2}\alpha_{1}+\alpha_{2}+\frac{1}{2}\alpha_{3},-\omega_3)$, if and only if $(\alpha,\beta)\in \{(\alpha_3,\alpha_3),(\alpha_{max},\alpha_{max})\} $. Denote by $${\bf c'}=\frac{(\alpha_3,\alpha_3)^2}{4}E_{\alpha_{3}}\wedge \check\alpha_3 \wedge E_{-\alpha_{3}}+\frac{(\alpha_{max},\alpha_{max})^2}{4}E_{\alpha_{max}}\wedge \check\alpha_{max} \wedge E_{-\alpha_{max}}\ .$$ We obtain that 
 $$({\bf c}-  {\bf c'})(v\wedge v'\wedge v'') \in
 \left(V_{\omega_{3}}(\omega_3)\cdot V_{\omega_{3}}(\frac{1}{2}\alpha_{1}+\alpha_{2}+\frac{1}{2}\alpha_{3})\cdot V_{\omega_{3}}(-\omega_3)\right)^{c}\subset S^3 V_{\omega_3}\ .$$
A straightforward calculation shows that ${\bf c'}(v\cdot v'\cdot v'')\ne 0$.

This implies  that  ${\bf c}(v\wedge v'\wedge v'')\ne 0$ and proves that  $V_{\omega_{3}}$ is not Poisson.  Claim \ref{claim:so(2n+1)} is proved.  
  \end{proof}

 \subsubsection{Type $\gg=so(2n)$}
 \label{se:so(2n)}
 \begin{claim}
 \label{claim:so(2n)}
 Let $\gg=so(2n)$. If $V_\lambda$ is Poisson, then $\lambda=\omega_1 $ or $n\in\{4,5\}$ and $\lambda\in\{\omega_{n},\omega_{n-1}\}$.
 \end{claim}
 \begin{proof}
 Let $\{\alpha_{1},\alpha_{2},\ldots,\alpha_{n}\}$ be a basis of $R(\gg)$. The fundamental weights have the form 
  $$\omega_{i}=\alpha_{1}+2\alpha_{2}+(i-1)\alpha_{i-1}+i(\alpha_{i}+\ldots +\alpha_{n-2})+\frac{1}{2}i(\alpha_{n-1} +\alpha_{n})\ ,$$ 
  $$\omega_{n-1}=\frac{1}{2}\left(\alpha_{1}+2\alpha_{2}+\ldots+(n-2)\alpha_{n-2}+\frac{1}{2}n\alpha_{n-1}+\frac{1}{2}(n-2)\alpha_{n}\right)$$ $$\omega_{n}= \frac{1}{2}\left(\alpha_{1}+2\alpha_{2}+\ldots+(n-2)\alpha_{n-2}+\frac{1}{2}(n-2)\alpha_{n-1}+\frac{1}{2}n\alpha_{n}\right)\ .$$ 
  
  The highest root is $\alpha_{max}=\alpha_{1}+2\alpha_{2}+\ldots +2\alpha_{n-2}+\alpha_{n-1}+\alpha_{n}$. It is easy to verify that  if the weight $2\lambda-\alpha_{i}\in R(\gg)$ for some  $\lambda\in P^+(\gg)$, then $\lambda=\omega_{1}$ or, in the case of $n=4,5$,  $\lambda=\omega_{n}$ and $\lambda=\omega_{n-1}$. Therefore,  we  obtain immediately from Proposition \ref{pr:non flatif too big}  that if $V_{\lambda}$ is Poisson, then $\lambda=\omega_{1}$, or $n\in\{4,5\}$ and $\lambda\in\{\omega_{n},\omega_{n-1}\}$.
 
 Claim \ref{claim:so(2n)} is proved. \end{proof}

  \subsubsection{Type $\gg=sp(2n)$}
  
  \begin{claim}
  \label{claim:sp(2n)}
  Let $\gg=sp(2n)$. If $V_\lambda$ is Poisson, then $\lambda=\omega_1$ or $(\gg,V)=(sp(4), V_{\omega_2})$.
  \end{claim}
  \begin{proof}
If $n=2$, we have $sp(4)\cong so(5)$ and  we obtain from Section \ref{se:so(2n+1)}  that   $V_{\lambda}$ is not Poisson, unless $\lambda=\omega_{i}$ for $i=1,2$. 
 
 Now, let us continue with the case of $\gg=sp(2n)$, $n\ge3$. Let $\{\alpha_{1},\ldots ,\alpha_{n}\}$ be a basis of $R(\gg)$. The fundamental weights have the form 
 $$\omega_{i}=\alpha_{1}+2\alpha_{2}+\ldots (i-1)\alpha_{i-1}+i(\alpha_{i}+\ldots +\alpha_{n-1}+\frac{1}{2}\alpha_{n})$$
 for $i\le n$, and the highest root is $\alpha_{max}=2\alpha_{1}+2\alpha_{2}+\ldots+\alpha_{n}$.   It is easy to verify that  if $\lambda\in P^+(\gg)$ and  $2\lambda-\alpha_{i}\in R(\gg)$ , then $\lambda=\omega_{1}$. Therefore,  we  obtain immediately from Proposition \ref{pr:non flatif too big}  that if $V_{\lambda}$ is Poisson, then $\lambda=\omega_{1}$.
Claim \ref{claim:sp(2n)} is proved.\end{proof} 
 
\subsubsection{The exceptional Lie algebras}

\begin{claim}
\label{claim:exceptional}
Let $\gg$ be an exceptional complex simple Lie algebra. If $V_\lambda$ is a nontrivial simple Poisson module, then $\gg=E_6$ and $\lambda\in\{\omega_1,\omega_6\}$.
\end{claim}
 \begin{proof}
 The Dynkin diagram $E_{6}$ contains two subdiagrams of type $D_{5}$ and one of type $A_{5}$ and Levi subalgebras isomorphic to $so(10)\oplus\CC$ and $sl_6\oplus \CC$. Let $\lambda=\sum_{i=1}^6 \l_i\omega_i\in P^+(\gg)$.  We obtain that $ v_\lambda\in V_\lambda(\lambda)$ generates a  $so(10)$-module  $V_{\lambda'}$ ,  where $\lambda'=\l_1\omega_1+\l_3\omega_2+\l_4\omega_4+\l_2\omega_5$ and an  $sl_6$-module $V_{\lambda''}$, where $\lambda''= \l_1\omega_1+\sum_{i=2}^5 \l_{i+1}\omega_i$. Note that we are only adjusting notations from \cite{Bou} for the various root systems $A_5$, $D_5$ and $E_6$.   Applying Proposition \ref{pr:Levisub} to the corresponding Levi subalgebras  and the results of  Sections \ref{susec:A_n}  and \ref{se:so(2n)}, we obtain that if $V_{\lambda}$ is Poisson, then $\lambda=\omega_{i}$, if $i=1,2,6$.  When considering the case $E_{6}$, we cannot apply Proposition \ref{pr:non flatif too big} to all weights $\lambda$, because the longest word in the Weyl-group associated to $E_{6}$ does not send all dominant  weights $\lambda\in P^{+}(\gg)$ to $-\lambda\in P(\gg)$. Since $\omega_{2}=\alpha_{max}$ we immediately  see that $V_{\omega_{2}}$ is the adjoint module. Clearly, $2\alpha_{max}-\alpha_{j}$ is not a root for all $j\in[1,6]$,  and hence  $V_{\omega_{2}}$ is not Poisson by Proposition \ref{pr:non flatif too big}.
      
Next, let $\gg=E_{7}$. Denote by $R(\gg)$ the corresponding root system  and by $P(\gg)$ the weight lattice. It is easy to derive from the tables in  \cite[Tables]{Bou} that there exists no $\lambda\in P^{+}(\gg)$ such that $2\lambda-\alpha_{i}\in R(\gg)$ for any $i\in[1,7]$.  Therefore, Proposition \ref{pr:non flatif too big}  yields that there is no $\lambda\in P^{+}(\gg)$ such that $V_{\lambda}$ is Poisson.

Now,  let $\gg=E_{8}$. Denote by $R(\gg)$ the corresponding root system and by $P(\gg)$ the weight lattice. It is easy to derive from the tables in  \cite[Tables]{Bou} that there exists no  $\lambda\in P^{+}(E_{7})$ such that $2\lambda-\alpha_{i}\in R(\gg)$ for any $i\in[1,8]$.  Therefore, Proposition \ref{pr:non flatif too big}  yields that there is no $\lambda\in P^{+}(\gg)$ such that $V_{\lambda}$ is Poisson.

As the second to last case, we we will consider the case of $\gg=F_{4}$.  Denote by $R(\gg)$ the corresponding root system and by $P(\gg)$ the weight lattice. It is easy to derive from the tables in  \cite[Tables]{Bou} that there exists no $\lambda\in P^{+}(\gg)$ such that $2\lambda-\alpha_{i}\in R(\gg)$ for any $i\in[1,4]$. Therefore, Proposition \ref{pr:non flatif too big}  yields that there is no $\lambda\in P^{+}(\gg)$ such that $V_{\lambda}$ is Poisson.

Finally, let $\gg=G_{2}$.  Denote by $R(\gg)$ the corresponding root system and by $P(\gg)$ the weight lattice, $ \{\alpha_{1},\alpha_{2}\}$ a basis of $R(\gg)$. The fundamental weights are $\omega_{1}=2\alpha_{1}+\alpha_{2}$ and $\omega_{2}=3\alpha_{1}+2\alpha_{2}$, and the highest root is $\alpha_{max}=\omega_{2}=3\alpha_{1}+2\alpha_{2}$.  It is easy to verify that  if $\lambda\in P^+(\gg)$ and $2\lambda-\alpha_{i}\in R(\gg)$ for some $i\in\{1,2\}$ then $\lambda=\omega_{1}$.  Therefore, if $V_{\lambda}$ is Poisson then $\lambda=\omega_{1}$ by Proposition \ref{pr:non flatif too big}. 

Now, let $\lambda=\omega_{1}$.  Let $v\in V_{\omega_{1}}(\omega_1)$  and $v'=E_{\alpha_{1}}(v)\in V_{\omega_{1}}(\alpha_{1}+\alpha_{2})$ and $v''=E_{3\alpha_{1}+2\alpha_{2}}(v')\in  V_{\omega_{1}}(-\omega_1)$. It is easy to see that $v'\ne 0$ and $v''\ne 0$. 
As in the discussion of the $so(7)$-case we have for all $\alpha,\beta\in R(\gg)$, $w(\gamma_i)\in V(\gamma_i)$,   

$$ E_\alpha\wedge[E_{-\alpha},E_\beta]\wedge E_{-\beta}(v\wedge v'\wedge v'')=\sum_{i=1}^k w(\gamma_i)\cdot w(\gamma'_i)\cdot w(\gamma''_i)\ ,$$ 
  and $w(\gamma'_i)\in V(\gamma'_i)$,  $w(\gamma''_i)\in V(\gamma''_i)$ such that $\gamma_i+\gamma_i'+\gamma_i''=\alpha_1+\alpha_2$. We obtain that $(\gamma_i,\gamma_i',\gamma_i'')=(\omega_1, \alpha_1+\alpha_2,-\alpha_1-\alpha_2)$ if and only if $(\alpha, \beta)\in\{(\alpha_1,\alpha_1),(\alpha_{max},\alpha_{max})$. 
  Denote by $${\bf c'}=\frac{(\alpha_1,\alpha_1)^2}{4}E_{\alpha_{1}}\wedge \check\alpha_1 \wedge E_{-\alpha_{1}}+\frac{(\alpha_{max},\alpha_{max})^2}{4}E_{\alpha_{max}}\wedge \check\alpha_{max} \wedge E_{-\alpha_{max}}\ .$$
  This implies that 
 $${\bf c}-{\bf c'}\in \left(V_{\omega_{1}}(\omega_1)\cdot V_{\omega_{1}}(\alpha_{1}+\alpha_{2})\cdot V_{\omega_{1}}(-\omega_1)\right)^{c}\subset S^3 V_{\omega_3}\ .$$

  A straightforward computation yields that  ${\bf c}(v\wedge v'\wedge v'')\ne 0$. Hence, $V_{\omega_1}$ is not Poisson.  That concludes our discussion of the case $\gg=G_2$.
 
Claim \ref{claim:exceptional} is proved.
\end{proof}

\subsection{Sufficient Conditions and Rigidity}
In order to prove the remaining assertions of  Theorem  \ref{th: class of Poisson} it now suffices to show that for all the modules $V_\lambda$ listed in Theorem  \ref{th: class of Poisson} (f), we have that $Hom_\gg(\Lambda^3 V_\lambda, S^3 V_\lambda)=\{0\}$, since this implies that $V_\lambda$ is Poisson by Proposition \ref{pr:0-hom=poi}.  The decomposition of the symmetric and exterior powers of the geometrically decomposable modules and the $sp(2n)$-module $V_{\omega_1}$ are well known. As a reference see (e.g.\cite{Ho}). In particular, in the case of $\gg=sl_n$ the decompositions of symmetric powers of the geometrically decomposable modules $V_\lambda$  are well known results from classical invariant theory, in the remaining cases the symmetric powers have been computed by multiple authors (see e.g. \cite{Sch}).
 
 The decomposition of the exterior powers of the simple geometrically decomposable modules and  the $sp(2n)$-module $V_{\omega_1}$ are computed in a beautiful way by Stembridge in \cite{Stem}; the decomposition however was already well known through the calculations of Lie algebra cohomology of nilradicals by Kostant (\cite{Kos}). One immediately obtains from these results that $Hom_\gg(\Lambda^3 V_\lambda, S^3 V_\lambda)=\{0\}$.

We are, however, interested in proving a stronger result which will be useful in the classification of flat modules in  Section \ref{subsect:flat modules}.

In \cite{BZ} we introduced a lower bound for the dimension of the symmetric and exterior cube, and correspondingly  we construct a  minimal submodule contained in the symmetric and exterior cubes.   
 We will use the notation
\begin{equation}
\label{eq:notation highest vectors}
V^\mu
\end{equation}
for the space of highest weight vectors of weight $\mu$ in a finite-dimensional $\gg$-module
$V$. For dominant $\lambda,\mu,\nu\in P^+(\gg)$ denote
$c_{\lambda,\mu}^\nu=\dim (V_\lambda\otimes V_\mu)^\nu=\dim_\CC (Hom_\gg(V_\nu, V_\lambda\otimes V_\mu)$; i.e.,
$c_{\lambda,\mu}^\nu$ is the tensor product multiplicity.  And for
any $\lambda,\mu\in P^+(\gg)$ denote $c^+_{\lambda;\mu}=\dim (S^2  V_\lambda)^\mu$ and $c^-_{\lambda;\mu}=\dim (\Lambda^2  V_\lambda)^\mu$ , so that $c^+_{\lambda;\mu}+c^-_{\lambda;\mu}=c_{\lambda,\lambda}^\mu$. Ultimately, define: 
$$ d_\lambda^\mu:=\sum_{\nu\in P_+} (c_{\lambda;\nu}^+-c_{\lambda;\nu}^-) c_{\nu,\lambda}^\mu \ ,.  $$
 
 We need the  following definition. 
\begin{definition}
We will call the $\gg$-module $S^3_{low} V_\lambda=\bigoplus_{\mu\in P_+}\CC^{ \max(d_\lambda^\mu,0)}\otimes  V_\mu$ the "lower symmetric cube" of $V_\lambda$, and similarly $\Lambda^3_{low} V_\lambda=\bigoplus_{\mu\in P_+}\CC^{ \max(-d_\lambda^\mu,0)}\otimes  V_\mu$ the "lower exterior cube".
\end{definition}

The definition is motivated by the following fact.

\begin{lemma}
\label{le:gen symmetric cube}
There exist injective homomorphisms of $\gg$-modules from $S^3_{low} V_\lambda\hookrightarrow S^3 V_\lambda$ and $\Lambda^3_{low} V_\lambda\hookrightarrow \Lambda^3 V_\lambda$.
\end{lemma} 

\begin{proof}
  By definition of $S^3 V_\lambda$, one has:
$(S^3  V_\lambda)^\mu=(S^2  V_\lambda\otimes
V_\lambda)^\mu\cap (V_\lambda\otimes S^2  V_\lambda)^\mu$.
Therefore,  we obtain the inequality:
$$\dim (S^3  V_\lambda)^\mu \ge \dim (S^2  V_\lambda\otimes V_\lambda)^\mu
+\dim (V_\lambda\otimes S^2  V_\lambda)^\mu-\dim (V_\lambda\otimes V_\lambda\otimes V_\lambda)^\mu$$
$$=\dim (S^2  V_\lambda\otimes V_\lambda)^\mu-\dim (V_\lambda\otimes \Lambda^2 V_\lambda)^\mu
= d_\lambda^\mu \ . $$
The existence of injective homomorphisms of $\gg$-modules from $S^3_{low} V_\lambda\hookrightarrow S^3 V_\lambda$ follows and the assertion for $\Lambda^3 V_\lambda$ can be proved analogously.
\end{proof}

We say that a $\gg$-module $V_\lambda$ is {\it rigid}, if  $S^3_{low} V_\lambda\cong S^3 V_\lambda$ and $\Lambda^3_{low} V_\lambda\cong \Lambda^3 V_\lambda$. The following theorem is the main result of this section.

\begin{theorem}
 \label{th:gd=gencube}
 Let $\gg$ be a simple Lie algebra. A simple $\gg$-module is rigid, if and only if $(\gg,V)$ is one of the pairs listed in Theorem \ref{th: class of Poisson} (f).
 \end{theorem}

\begin{proof}
We will first prove the "only if" assertion. Note the following fact connecting rigid and Poisson modules.

\begin{proposition}
\label{pr:rigid-Poisson}
If a simple $\gg$-module $V_\lambda$ is rigid, then $Hom_\gg(\Lambda^3 V_\lambda, S^3 V_\lambda)=\{0\}$ and $V_\lambda$ is Poisson.
\end{proposition}

\begin{proof}
It is easy to see that 

$$ \dim(Hom_\gg(\Lambda^3 V_\lambda, S^3 V_\lambda))=\sum_{\mu\in P^+}  \max(d_\lambda^\mu,0)\cdot  \max(-d_\lambda^\mu,0)=0\ .$$

We immediately obtain $Hom_\gg(\Lambda^3 V_\lambda, S^3 V_\lambda)=\{0\}$, and hence, that $V_\lambda$ is Poisson.
\end{proof}

We obtain from Proposition \ref{pr:rigid-Poisson} and the arguments in Section \ref{se:nec.cond} that if $V$ is rigid, then $(\gg,V)$ must be one of the pairs of Theorem \ref{th: class of Poisson} (f).

  The proof  of  the converse will consist of the following  steps for each of the listed simple modules $V_\lambda$.  
  
 \noindent(1) Determine the decomposition  $S^2 V_\lambda$ and $\Lambda^2 V_\lambda$, as well as $S^3 V_\lambda$ and $\Lambda^3 V_\lambda$. recall from above that the splittings of the symmetric and exterior powers of the modules in question are well known (see e.g. \cite[chapter 4]{Ho}).
 
 \noindent(2) Suppose $V_\mu$ appears with multiplicity one in $S^3 V_\mu$. Show that $d_\lambda^\mu=1$. Similarly, if  $V_\mu$ appears with multiplicity one in $\Lambda^3 V_\mu$. Show that $d_\lambda^\mu=-1$.

 In order to further simplify our computations note the following fact. Recall that if $\gg$ is a semisimple Lie algebra and $\tau$ a graph automorphism of the Dynkin diagram, then $\tau$ induces automorphisms $\tau_\gg:\gg\to \gg$, $\tau_P:P(\gg)\to P(\gg)$.  
 
 \begin{lemma}
 \label{le:Dynkinauto}
 Let $\gg$ be a semisimple Lie algebra and let $\tau$ be a graph automorphism of the Dynkin diagram of $\gg$.
 
 \noindent(a)   We have for the tensor product multiplicities $c_{\mu}^\lambda=c_{\tau(\mu)}^{\tau(\lambda)}$. Similarly, if  $S^n V_{\lambda}=\bigoplus_{i} V_{\nu_i}$ then  $S^n V_{\tau(\lambda)}=\bigoplus_{i} V_{\tau(\nu_i)}$ and if  $\Lambda^n V_{\lambda}=\bigoplus_{i} V_{\nu_i}$, then $\Lambda^n V_{\tau(\lambda)}=\bigoplus_{i} V_{\tau(\nu_i)}$.

 \noindent(b) If the simple $\gg$-module $V_\lambda$ is rigid, then $V_{\tau(\lambda)}$ is rigid.
 
 \end{lemma} 
 
\begin{proof}
Part(a) is well known, and (b) follows immediately from (a).
\end{proof}

In order to accomplish Step (2) we will need a generalization of the {\it tensor product stabilization} of the  Littlewood-Richardson rule to other types of simple Lie algebras by Kleber and Vishwanath (\cite{KV} and \cite{Vish}).  
Let $\gg$ be a complex  simple Lie algebra of type $X_n$, $X\in\{A,B,C,D\}$.
We denote for a triple $\lambda,\mu,\nu\in P^+(\gg)=P^+(X_n)$ of dominant weights the tensor multiplicity $c_{\lambda,\mu}^\nu(X_n)$.

 We first recall the Littlewood-Richardson rule for  $\gg=sl_{n+1}(\CC)$; i.e. the Dynkin diagram of $\gg$ is of type $A_n$:  For dominant weights $\lambda=\sum_{i=1}^{n} \l_i\omega_i$, $\mu=\sum_{i=1}^{n} m_i\omega_i$ and $\nu=\sum_{i=1}^{n} n_i\omega_i$ one has $c_{\lambda,\mu}^\nu(A_n)=c_{\lambda,\mu}^\nu(A_{m})$ for all $m\ge n$  if
 
 $$\sum_{i=1}^{n} i( \l_i+m_i)=\sum_{i=1}^{n} i n_i\ .$$
 
 This phenomenon is commonly referred to as {\it tensor product stabilization} (see e.g. \cite{KV}).
 
   Next,  let $\gg$ be a complex  simple Lie algebra of type $X_n$, $X\in\{B,C,D\}$. Following  ideas of \cite[Ch. 6, Corollary 3]{Vish} we call  a weight $\gamma\in P(\gg)=P(X_n)$ $A$-supported, if $\gamma=\sum_{i=1}^{\omega_{n-2}} c_i\omega_i$. This means that $\gamma$ is supported entirely in the $A_k$-part of the Dynkin diagram $X_n$.  Note that if a weight $\gamma=\sum_{i=1}^{\omega_{n}} c_i\omega_i$  is $A$-supported in $P(X_n)$, then it is $A$-supported in $P(X_m)$ for all $m\ge n$.  One obtains the following fact.   

 \begin{proposition}\cite[Ch. 6, Corollary 3, Remark 10.1]{Vish}
\label{pr:series-stab}
Let  $\gg$ be a complex  simple Lie algebras of type $X_n$ where $X\in\{,B,C,D\}$. If $\lambda,\mu,\nu\in P^+(X_n)$ are $A$-supported, then   

$$ c_{\lambda,\mu}^\nu(X_n)=c_{\lambda,\mu}^\nu(X_m)\ ,~m\ge n\ .$$
\end{proposition}

 We will now show, case by case, that the modules in question are indeed rigid. Since the computations are rather long but straightforward, we include the complete proof in only one nontrivial case ($(\gg,V)=(sl_n,V_{2\omega_1})$). Complete calculations can be found in \cite{ZW}.
  
\subsubsection{The case $\gg=sl_n$}

 \begin{proposition}
 \label{prop:An-rigid}
 If $\gg=sl_n$, then the simple module $V=V_\lambda$ is rigid, if $\lambda\in\{\omega_1,2\omega_1\omega_2, \omega_{n-2},\omega_{n-1},2\omega_{n-1}\}$.  
 \end{proposition}
 \begin{proof}
 By Lemma \ref{le:Dynkinauto} it suffices to consider the case of $\lambda\in\{\omega_1,2\omega_1\omega_2\}$.
 
 \begin{lemma}
 Let $\gg=sl_n$, $n\ge 7$. We have the following:
  
 \noindent(a) Let $V=V_{\omega_1}$. We have $S^2 V=V_{2\omega_1}$ , $S^3V=V_{3\omega_1}$, $\Lambda^2 V=V_{\omega_2}$ and $\Lambda^3 V=V_{\omega_3}$.

 \noindent(b)\cite[Theorem 3.1, Theorem 4.4.2]{Ho}  Let $V=V_{2\omega_1}$.  We have $S^2V \cong V_{4\omega_1} \oplus V_{2\omega_2}$,  $\Lambda^2 V \cong V_{2\omega_1+\omega_2}$, $S^3 V\cong V_{6\omega_1}\oplus V_{2\omega_1+2\omega_2}\oplus V_{2\omega-3}$ and $\Lambda^3 V\cong V_{3\omega_1+\omega_3}\oplus V_{3\omega_2}$.
 
 \noindent(c)\cite[Theorem 3.8.1, Theorem 4.4.4]{Ho} Let $V=V_{\omega_2}$. We have $S^2 V \cong V_{2\omega_2}\oplus V_{\omega_4}$, $\Lambda^2 V \cong V_{\omega_1+\omega_3}$, $S^3V\cong V_{3\omega_2}\oplus V_{\omega_2+\omega_4}\oplus V_{2\omega_3}$.
  
 \end{lemma}

 It remains to show the following.
 
 \begin{lemma}
 \label{le:tensor-stab-mult}
 Let $\gg=sl_n$, $n\ge 7$.
 
 \noindent(a) If $\lambda=\omega_1$, then $d_{\omega_1}^\lambda=1$ if $\lambda=3\omega_1$, and $d_{\omega_1}^\lambda=-1$ if $\lambda=\omega_3$.

\noindent(b) If $\lambda=2\omega_1$, then $d_{2\omega_1}^\lambda=1$ if $\lambda\in \{6\omega_1,2\omega_1+2\omega_2, 2\omega_3\}$ and  $d_{2\omega_1}^\lambda=-1$ if $\lambda\in \{ 3\omega_1+\omega_3, 3\omega_2\}$.

\noindent(c) If $\lambda=\omega_2$, then $d_{2\omega_1}^\lambda=1$, if $\lambda\in \{3\omega_2,\omega_2+\omega_4, \omega_6\}$ and $d_{2\omega_1}^\lambda=-1$,  if $\lambda\in \{ 2\omega_1+\omega_4, 2\omega_3\}$.

\end{lemma} 

\begin{proof}

 Computing the decompositions manually (or using the computer algebra system LIE \cite{VL}) for $n=7$, and applying the Littlewood Richardson rule for all $n\ge 7$ we obtain

 We have for $S^2V_{\omega_1}\otimes V_{\omega_1}$: 
 $$c_{2\omega_1,\omega_1}^{3\omega_1}=1\ ,c_{2\omega_1,\omega_1}^{\lambda}=0\ ,$$ if $\lambda\ne \omega_1$. For $\Lambda^2 V_{\omega_1}\otimes V_{\omega_1}$ we obtain:
$$ c_{\omega_2,\omega_1}^{\omega_3}=1\ , c_{\omega_2,\omega_1}^{\lambda}=0\ ,$$ if $\lambda\ne \omega_1$. Part (a) is proved.

 We have for the factors of $S^2 V_{2\omega_1}$:
 
  $$ V_{4\omega_1}\otimes V_{2\omega_1}:\ c_{4\omega_1,2\omega_1}^{6\omega_1}=1\ , c_{4\omega_1,2\omega_1}^{2\omega_1+2\omega_2}=1\ ,c_{4\omega_1,2\omega_1}^{4\omega_1+\omega_2}=1\ ,c_{4\omega_1,2\omega_1}^{\lambda}=0\ ,$$
   for  $\lambda\in P^+(\gg)$, unless $(\lambda,\omega_i)>0$ for some $i\ge 7$. 
 $$ V_{2\omega_2}\otimes V_{2\omega_1}: \ c_{2\omega_2,2\omega_1}^{2\omega_1+2\omega_2}=1\ ,  c_{2\omega_2,2\omega_1}^{\omega_1+\omega_2+\omega_3}=1\ , c_{2\omega_2,2\omega_1}^{2\omega_3}=1\ ,c_{2\omega_2,2\omega_1}^{\lambda}=0\ ,$$
 
   for  $\lambda\in P^+(\gg)$, unless $(\lambda,\omega_i)>0$ for some $i\ge 7$. 
   
  For  $\Lambda^2  V_{2\omega_1}$:
  
  $$V_{2\omega_1+\omega_2}\otimes V_{2\omega_1}: c_{2\omega_1+\omega_2,2\omega_1}^{4\omega_1+\omega_2}=1\ ,c_{2\omega_1+\omega_2,2\omega_1}^{2\omega_1+2\omega_2}=1\ , c_{2\omega_1+\omega_2,2\omega_1}^{3\omega_1+\omega_3}=1\ , c_{2\omega_1+\omega_2,2\omega_1}^{3\omega_2}=1\ ,$$
  $$ c_{2\omega_1+\omega_2,2\omega_1}^{\omega_1+\omega_2+\omega_3}=1\ , c_{2\omega_1+\omega_2,2\omega_1}^{\lambda}=0\ ,    $$
    for  $\lambda\in P^+(\gg)$, unless $(\lambda,\omega_i)>0$ for some $i\ge 7$. Part (b) is proved.

Proceed similarly to prove part(c). 
We can now read off the $d_\lambda^\mu$ mentioned in the Lemma. The lemma is proved.

\end{proof}

We obtain that    the simple $sl_n$-modules $V_{\omega_1}$, $V_{2\omega_1}$ and $V_{\omega_2}$ are rigid for $n\ge 7$. In the case of $n<7$ one can prove the assertion of the proposition by direct computation (e.g. using LIE).  Proposition \ref{prop:An-rigid} is proved.

\end{proof}

 \subsubsection{The  case $\gg=so(k)$}   
 
Next, we will proceed  with the case when, $\gg=so(k)$.
\begin{proposition}
\label{pr:so-rigid} 
Let $\gg=so(k)$ and $V=V_{\omega_1}$. Then $V$ is rigid. Moreover, $V=V_\lambda$ is rigid, if  $\gg=so(8)$ and $\lambda=\{\omega_3,\omega_4\}$ or if $\gg=so(10)$ and $\lambda=\{\omega_4,\omega_5\}$. 
\end{proposition} 
\begin{proof}
\begin{lemma}
\label{le:so-powers}
\cite{Sch}\cite{Stem}
Let $\gg=so(k)$, $k\ge 9$. Then $S^2 V_{\omega_1}\cong   V_{2\omega_1}\oplus V_0$, $\Lambda^2 V_{\omega_1}\cong V_{\omega_2}$, $S^3 V_{\omega_1}\cong  V_{3\omega_1}\oplus V_{\omega_1}$ and $\Lambda^3 V_{\omega_1}\cong V_{\omega_3}$.
 \end{lemma}

 \begin{lemma}
 \label{le:low-so}
Let $\gg=so(k)$, $k\ge 9$. If $\mu\in\{3\omega_1, \omega_1\}$, then $d_{\omega_1}^{\mu}=1$. If $\mu=\omega_3$, then $d_{\omega_1}^{\mu}=-1$.
 \end{lemma}  
  
  \begin{proof}
 Analogous to the proof of Lemma \ref{le:tensor-stab-mult}.
 \end{proof}
 
We obtain the assertion of Proposition \ref{pr:so-rigid} for $k\ge 9$ from Lemma \ref{le:so-powers} and Lemma \ref{le:low-so}.  The assertion for the remaining cases of $k\le 8$ and the case $k=10$, $\lambda\in\{\omega_4,\omega_5\}$ can be verified directly (e.g. using LIE).
  Proposition \ref{pr:so-rigid} is proved.
 \end{proof}
 
 \subsubsection{The case $\gg=sp(2n)$}  
   
\begin{proposition}
\label{pr:sp-rigid} 
Let $\gg=so(2n)$ and $V=V_{\omega_1}$. Then $V$ is rigid.  
\end{proposition}    
\begin{proof}   
   
 \begin{lemma}\cite{Ho},\cite{Stem}
 Let $\gg=sp(2n)$ with $n\ge 4$ and let $V=V_{\omega_1}$. We have  $S^2 V\cong V_{2\omega_1}$ and $\Lambda^2 V\cong V_{\omega_2}\oplus V_0$. Moreover, $S^3 V\cong V_{3\omega_1}$ and $\Lambda^3 V\cong V_{\omega_1}\oplus V_{\omega_3}$. 
 \end{lemma}
 
Moreover we have the following fact.

\begin{lemma}
Let $\gg=sp(2n)$, $n\ge 4$. We have that $d_{\omega_1}^\mu=1$, if $\mu=3\omega_1$ and $d_{\omega_1}^\mu=-1$, if $\mu\in\{\omega_3, \omega_1\}$.
\end{lemma}

 \begin{proof}
 Analogous to the proof of Lemma \ref{le:tensor-stab-mult}.
 \end{proof}

This proves the proposition in the case $n\ge 4$. The assertion for the remaining cases $n\le 3$ can be verified directly using LIE.
Proposition \ref{pr:sp-rigid}  is proved.  
\end{proof}
  
 \subsubsection{Proof of Theorem \ref{th:gd=gencube} and the proof of Theorem \ref{th: class of Poisson}} 
  
  The assertion of Theorem \ref{th:gd=gencube} in the cases of  $\gg=E_6$ and $\lambda=\{\omega_1,\omega_6\}$ can also be verified through direct computation (e.g. using LIE). Theorem \ref{th:gd=gencube} then follows from Propositions  \ref{prop:An-rigid}, \ref{pr:so-rigid}  and \ref{pr:sp-rigid}.  
 \end{proof}
   
Theorem \ref{th:gd=gencube}  and Proposition  \ref{pr:rigid-Poisson} imply that (f) yields (c) and (a).
Thus, we have so far proved the equivalence of parts (a--c), (e) and (f) of Theorem \ref{th: class of Poisson}.  Additionally, we obtain that (f) implies (d) from the explicit computation of the exterior squares in the proof of Theorem \ref{th:gd=gencube}. We complete the proof of  Theorem \ref{th: class of Poisson} by  showing that (d) implies (a).

\begin{proposition}
\label{pr:simple square}
Let $\gg$ be a semisimple Lie algebra and let $V_\lambda$ be a simple $\gg$-module. 
If $\Lambda^2 V_\lambda$   is simple, then $V_\lambda$ is Poisson.  
\end{proposition}

\begin{proof}
Suppose that $V_\lambda$ is simple. Then $c$ acts as multiplication by a constant $\mu$ on $\Lambda^2 V_\lambda$. Therefore, $c_{12}(\overline v)=\mu \overline v$ for all $v\in\Lambda^2 V_\lambda\otimes V_\lambda$ and $c_{23}(\overline v')=\mu \overline v'$ for all $v'\in V_\lambda\otimes \Lambda^2 V_\lambda$. This implies that

$$ [c_{12}, c_{23}](\overline v)=(\mu^2-\mu^2)(\overline v)=0$$
for all $\overline v\in \Lambda^2 V_\lambda\otimes V_\lambda\cap V_\lambda\otimes \Lambda^2 V_\lambda=\Lambda^3 V_\lambda$ and hence, $V_\lambda$ is Poisson.  Proposition \ref{pr:simple square} is proved.

\end{proof}

Proposition \ref{pr:simple square} completes the proof of Theorem  \ref{th: class of Poisson}. 

\endproof

 Theorem  \ref{th: class of Poisson} and Theorem \ref{th:gd=gencube} yield the following corollary, which will play an important role in the classification of flat modules in Section \ref{se:prlime}.
 
 \begin{corollary}
 \label{co: poisson=rigidity}
 Let $\gg$ be a simple Lie algebra and $V$ a simple $\gg$-module. $V$ is Poisson, if and only if $V$ is rigid. 
 \end{corollary}

\section{Open Questions and Conjectures}
\label{se:conj}  
  
In this final chapter we will present a number of conjectures and questions which  will be interesting for future research in the area of braided symmetric algebras. The  classification of flat modules lets us expect that the following conjecture holds.

\begin{conjecture}
\label{conj:reducedsymm to braided}
Let $\gg$ be a reductive Lie algebra and $V^q$ a finite dimensional $U_q(\gg)$-module. The braided symmetric algebra $S_\sigma (V^q)$ is a flat deformation of the reduced symmetric algebra  $\overline{S(V, r^-)}$ of the classical limit $V$ of $V^q$.  
\end{conjecture}

This conjecture is of particular interest because it opens the possibility to address the following not yet investigated question.

\begin{problem} 
\label{pr:flat deformation of algebras} 
Quantize a commutative Poisson algebra ${\mathcal A}$; e.g. a reduced symmetric algebra.  \end{problem}

While this conjecture gives rise to the question of quantization of manifolds with a bracketed structure, 
the classification of flat modules in Theorem \ref{the:class. flat} and  our computation of the braided symmetric and exterior cubes of simple $U_q(sl_2)$-modules in  \cite[Theorem 2.40]{BZ} suggest that  braided symmetric and exterior cubes are rigid in the following sense.

\begin{conjecture}
\label{conj:rigidity}
Let $\gg$ be a simple Lie algebra and $V^q$ a simple $U_q(\gg)$-module and $V$ its classical limit. The classical limit of $S^3_\sigma V^q$ (resp. $\Lambda^3_\sigma V^q$) is isomorphic to $S^3_{low} V$ (resp. $\Lambda^3_{low} V$) as a $U(\gg)$-module. 
\end{conjecture}

It is easy to see that the assertion does not hold for nonsimple $V^q$ or $\gg$ semisimple, as the example of the natural $U_q(sl_m\times sl_n)$-module shows.   
 \section{Appendix}
\label{se:appendix}

We develop in this appendix notation for products and complements of weight-spaces in symmetric and exterior powers of vectorspaces.
Let $V$ be a finite-dimensional vector space, and let $V_{1}, V_{2},\ldots, V_{n}$ be subspaces of $V$. We define
 $V_{1}\cdot V_{2}\cdot\ldots\cdot V_{n}\subset S^{n}V$ and   $V_{1}\wedge V_{2}\wedge \ldots\wedge V_{n}\subset \Lambda^{n}V$ the subspaces generated by elements $v_{1}\cdot v_{2}\cdot\ldots\cdot v_{n}$, resp. $v_{1}\wedge v_{2}\wedge \ldots\wedge v_{n}$, where  $v_{i}\in V_{i}$ for $i\in[1,n]$.

 \begin{lemma}
 \label{le:powerdecomp}
 Let $V$ be a finite-dimensional vector space and let $V_{1},\ldots V_{n}$ be subspaces such that $V\cong\bigoplus_{i=1}^{n} V_{i}$. Denote by $\mathcal{P}(m)$ the set of all increasing $m$-element sequences $\underline p=  ( p_{1},\ldots, p_{m})$, $1\le p_{1}\le \ldots \le p_{m}\le n$ in $[1,n]$. The $m$-th symmetric and exterior powers now admit the following decomposition:
 
 $$S^{m}V\cong\bigoplus_{p\in \mathcal{P}} V_{p_{1}}\cdot\ldots\cdot V_{p_{m}}\ , \Lambda^{m}V\cong\bigoplus_{p\in \mathcal{P}} V_{p_{1}}\wedge\ldots\wedge V_{p_{m}}  \ .$$   
 
 \end{lemma}
 \begin{proof}
 
 Recall that if  $V\cong V_{1}\oplus V_{2}$, then $S^{m}V\cong \bigoplus_{ i=0}^{m} S^{i}V_{1}\boxtimes S^{m-i}V_{2}$, resp. that  $\Lambda^{m}V\cong \bigoplus_{ i=0}^{m} \Lambda^{i}V_{1}\boxtimes \Lambda^{n-i}V_{2}$.  We can now use induction in $n$ to  prove that for $V\cong\bigoplus_{i=1}^{n}V_{i}$ we have 
 $$ S^{m}V\cong \bigoplus_{\l\in \mathcal{L}_{m}}\bigotimes_{i=1}^{n}S^{\l_{i}}V_{i}\ , \Lambda^{m}V\cong \bigoplus_{\l\in \mathcal{L}_m}\bigotimes_{i=1}^{n}\Lambda^{\l_{i}}V_{i}\ ,$$ 
 where $\mathcal{L}_m$ denotes the set of of $n$-element sequences $(\l_{1},\ldots, \l_{n})$ such that $\l_1+\ldots+\l_{n}=m$.
  
 Using the notation $V\cdot\ldots \cdot V=S^{m}V$ and $V\wedge V\wedge\ldots \wedge V=\Lambda^{m}V$ and the fact that $S^{i}V_{1}\boxtimes S^{j}V_{2}\cong S^{i}V_{1}\cdot S^{j}V_{2}$, resp. $\Lambda^{i}V_{1}\boxtimes \Lambda^{j}V_{2}\cong \Lambda^{i}V_{1}\wedge \Lambda^{j}V_{2}$, we obtain the desired result. The lemma is proved.
 \end{proof}

We will use the notation $(V_{p_{1}}\cdot\ldots V_{p_{m}})^{c}\subset S^{m}V$ to denote 
\begin{equation}
\label{eq: complement}
(V_{p_{1}}\cdot\ldots V_{p_{m}})^{c}= \bigoplus_{ \underline p'\ne \underline p\in \mathcal{P}} V(\lambda_{p'_{1}})\wedge\ldots\wedge V(\lambda_{p'_{m}})\ .
\end{equation}


\begin{thebibliography}{xxx}



\bibitem{Baz} Y.~Bazlov, Nichols-Woronowicz algebra model for Schubert Calculus on Coxeter groups, math/0409206.

\bibitem{BD} A.~Belavin and V.~Drinfeld, Triangle equations and simple Lie algebras, Soviet Sci. Rev. Sect. C, Math. Phys. Rev.  4 (1984), 93--165. 

\bibitem{bk1}
A.~ Berenstein and D.~Kazhdan, Geometric and unipotent crystals,
 Geom. Funct. Anal., Special Volume, Part I, (2000), 188--236.

\bibitem{bk2} A.~ Berenstein and D.~Kazhdan, Geometric and unipotent crystals II: from\\
geometric crystals to crystal bases, to appear in Cont. Math., math.QA/0601391.



\bibitem{bz-qclust} A.~ Berenstein, A.~ Zelevinsky, Quantum cluster algebras, Adv. Math. 195 (2005), no.2, 405--455.
\bibitem{BZ}  A.~Berenstein and S. ~Zwicknagl,
              Braided Symmetric and Exterior Algebras, to appear in  Trans. of the Amer. Math. Soc., math.QA/0504155.
\bibitem{Bou} N.~ Bourbaki,  Groupes et Algebres de Lie, Masson, 1981. 
\bibitem{brown-goodearl} K.~Brown and K.~Goodearl,  Lectures on algebraic quantum groups,
Birkh\"auser, 2002.
\bibitem{Br} J.~Burndan, Kazhdan-Lusztig Polynomials and Character Formulae for the Lie Superalgebra $\mathfrak{q}(n)$,  Adv. Math. 182 (2004),28--77.

\bibitem{BGY} K.~Brown, K.~Goodearl, M.~Yakimov,  Poisson structures on affine spaces and flag varieties,I. Matrix Affine Spaces, to appear in Adv. Math.  math.QA/0501109.
\bibitem{DKP} C.~De Concini, V.G.~Kac, C. Procesi, Some Quantum Analogues of Solvable Lie Groups, Geometry and Analysis (Bombay,1992), Tata Inst. Fund. Res. Bombay, 1995, 41--65.
\bibitem {Donin} J. ~Donin,  Double quantization on the coadjoint representation of ${\rm sl}(n)$, \\ Quantum groups and integrable systems, Part I (Prague, 1997), Czech. J. Phys.  47 (1997), no. 11, 1115--1122.
 
\bibitem{DR1} V.~ Drinfel'd, Quasi-Hopf Algebras, Leningrad Math. Journal, 1 (1990), no. 6, 1419--1457.
\bibitem{DR2}V.~ Drinfel'd, Commutation Relations in the Quasi-Classical Case,
 Sel. Math. Sov., 11, no.4 (1992), 317--326.
\bibitem{DurdOzi} M.~ Durdevic, Z.~Oziewicz, Clifford Algebras and Spinors for Arbitrary Braids,  Differential geometric methods in theoretical physics (Ixtapa-Zihuatanejo, 1993).  Adv. Appl. Clifford Algebras, 4 (1994), Suppl. 1, 461--467.

\bibitem{GF} A.~Fokas, I. ~Gel'fand,   Quadratic Poisson algebras and their infinite dimensional extensions, J. Math. Phys. 35 no. 6 (June 1994).

\bibitem{GY} K.R.~Goodearl, M.~Yakimov, Poisson structures on affine spaces and flag\\ varieties. II. general case, math.QA0509075.


\bibitem{Ho} R.~Howe, Perspectives on invariant theory: Schur duality, multiplicity-free\\ actions and beyond, 
 Israel Math. Conf. Proc. 8 1995, 1--182. 
\bibitem{Ho2}R.~Howe, private communication.
\bibitem{Hu} J.~Humphreys, Introduction to Lie algebras and representation theory,\\ Graduate texts in Mathematics 9, Springer, 1972.
\bibitem{Jan} J.C. Jantzen, An Introduction to Quantum Groups, Graduate Studies in\\ Mathematics, American Mathematical Society, 1996.

\bibitem{JMO} N.~Jing, K.~Misra, M.~Okado, q-Wedge modules for quantized enveloping\\ algebras of classical type, J. Algebra 230 (2000), no.2, 518--539. 

\bibitem{Kam} A.~Kamita, Quantum Deformations of Certain Prehomogeneous Vector Spaces III, Hiroshima Math. J. 30 (2000), 79--105.
\bibitem{KV} M.~Kleber, S. Viswanath, Tensor product stabilization in Kac-Moody Lie\\ algebras, to appear in Adv. Math, math.RT/0405444.
\bibitem{Kon} M.~ Kontsevich, Deformation quantization of Poisson manifolds, Lett. Math. Phys. 66 (2003) no. 3, 157--216.
\bibitem{Kos} B.~Kostant, Lie algebra cohomology and the generalized Borel-Weil theorem, Ann. Math. 74 (1961), 329--387.
\bibitem{Lec} C.~ Lecouvey, An algorithm for computing the global basis of a finite dimensional irreducible $U\sb q({\rm so}\sb {2n+1})$ or $U\sb q({\rm so}\sb {2n})$-module,  Comm. Algebra 32 (2004), no. 5, 1969--1996.
\bibitem {L0} G. ~Lusztig, Canonical bases arising from quantized enveloping algebras.
 J. Amer. Math. Soc. 3 (1990), no. 2, 447--498.

 \bibitem{L2} G.~Lusztig, Quantum Groups at Roots of 1, Geometriae Dedicata 35 (1990) 89--114.
\bibitem{M} I.~Musson, Ring theoretic properties of the coordinate rings of quantum symplectic and Euclidean space,  Ring Theory, Proc. Biennial Ohio State-Denison Conf. 1992 (S.K. Jain and S.T.Rizvi, eds.), World Scientific, Singapore, 1993, 248--258.
\bibitem{Nou} M.~ Noumi, Macdonald's Symmetric Polynomials as Zonal Spherical Functions on some Quantum Homogeneous Spaces,  Adv. Math., 123 (1996), 16--77.
\bibitem{RTF} N.~Reshitikhin, L.~A.~ Takhtadzhyan and L.~D.~ Fadeev, Quantization of Lie groups and Lie algebras, Leningrad Math. J. , 1 (1990), 193--225.
\bibitem {R-D}
O.~ Rossi-Doria, A $U_q(sl(2))$-representation with no quantum
symmetric \\ algebra. Atti Accad. Naz. Lincei Cl. Sci. Fis. Mat. Natur. Rend. Lincei (9) Mat. Appl. 10 (1999), no. 1, 5--9.
\bibitem{Sch} W.~Schmid, Die Randwerte holomorpher Funktionen auf hermitesch\\ symmetrischen R\"aumen, Inv. Math.91 (1969-70), 61-80.
\bibitem{Str} E.~ Strickland, Classical Invariant Theory for the Quantum Symplectic Group, Adv. Math. 123 (1996), 78--90.
\bibitem{Stem} J.~ Stembridge,   On the Classifcation of Multiplicity-Free Exterior Algebras, Int. Math. Res. Not.,  no. 40 (2003), 2181--2191. 
\bibitem{VL} M.~VanLeeuwen, LiE, A software package for Lie group computations, http://young.sp2mi.univ-poitiers.fr/~marc/LiE/.
\bibitem{Vish} S.~Vishwanath, Dynkin diagram sequences and stabilization\\ phenomena,math.RT/0505616.
\bibitem{Wor} S.~Woronowicz, Differential Calculus on Matrix Pseudogroups (Quantum Groups), Comm. in Math.  Phys., 122,  (1989), 125--170. 
\bibitem{Zh} R.~Zhang,  Howe Duality and the Quantum General Linear Group, Proc. Amer. Math.Soc., Amer. Math. Soc., Providence, RI, 131, no. 9, 2681--1692.
\bibitem{ZW} S.~Zwicknagl, Equivariant Poisson Algebras and their Deformations, PhD thesis, University of Oregon, Dec. 2006.
\end{thebibliography}
\end{document}